\documentclass[reqno, colorBG]{amsart}
\usepackage{amssymb}
\usepackage{amsmath}
\usepackage[mathscr]{euscript}

\usepackage[small]{caption}
\usepackage{psfrag}
\usepackage{graphicx}
\graphicspath{ {images/} }
\usepackage{color}

\makeatletter
\@addtoreset{equation}{section}
\makeatother

\newtheorem{theorem}{Theorem}[section]
\newtheorem{lemma}[theorem]{Lemma}
\newtheorem{proposition}[theorem]{Proposition}
\newtheorem{corollary}[theorem]{Corollary}

\newtheorem{definition}[theorem]{Definition}
\newtheorem{remark}[theorem]{Remark}

\newcounter{as}[section]

\newcommand{\mc}[1]{{\mathcal #1}}
\newcommand{\mf}[1]{{\mathfrak #1}}
\newcommand{\mb}[1]{{\mathbf #1}}
\newcommand{\bb}[1]{{\mathbb #1}}
\newcommand{\bs}[1]{{\boldsymbol #1}}
\newcommand{\ms}[1]{{\mathscr #1}}

\newcommand{\<}{\langle}
\renewcommand{\>}{\rangle}
\renewcommand{\Cap}{{\rm cap}}

\begin{document}

\title[Metastability of Non-reversible random walks in potential field]
{Metastability of Non-reversible random walks in a potential
  field, the Eyring-Kramers transition rate formula}

\begin{abstract}
  We consider non-reversible random walks evolving on a potential
  field in a bounded domain of $\bb R^d$. We describe the complete
  metastable behavior of the random walk among the landscape of
  valleys, and we derive the Eyring-Kramers formula for the mean
  transition time from a metastable set to a stable set.
\end{abstract}

\author{C. Landim, I. Seo}

\address{\noindent IMPA, Estrada Dona Castorina 110, CEP 22460 Rio de
  Janeiro, Brasil and CNRS UMR 6085, Universit\'e de Rouen, Avenue de
  l'Universit\'e, BP.12, Technop\^ole du Madril\-let, F76801
  Saint-\'Etienne-du-Rouvray, France.  \newline e-mail: \rm
  \texttt{landim@impa.br} }

\address{\noindent Courant Institute of Mathematical Sciences, New York
  University, 251 Mercer Street, New York, N.Y. 10012-1185, USA.
  \newline e-mail: \rm \texttt{insuk@cims.nyu.edu} }


\maketitle

\section{Introduction}
\label{sec:intro}

Metastability has attracted much attention in these last years in
several different contexts, from spin dynamics to SPDEs, from random
networks to interacting particle systems. We refer to the recent
monographs \cite{OV, BH} for references.

At the same time, some progress has been made on the potential theory
of non-reversible Markov chains. Gaudilli\`ere and Landim \cite{GL}
derived a Dirichlet principle for the capacity of non-reversible
continuous-time Markov chains, and Slowik \cite{Slo} proved a Thomson
principle. 

These advances in the potential theory of Markov chains permitted to
derive the metastable behavior of non-reversible dynamics. The
metastable behavior of the condensate in a totally asymmetric
zero-range process evolving on a fixed one-dimensional ring has been
proved in \cite{Lan2}, and Misturini \cite{Mis} derived the
metastable behavior of the ABC model as the asymmetry increases. In
another perspective, Bouchet and Reygner \cite{BR} provided a formula,
in the context of small perturbations of dynamical systems, for the
Eyring-Kramers mean transition time from a metastable set to a stable
set.

Motivated by the evolution of mean-field planar Potts model
\cite{ls2}, whose dynamics can be mapped to a non-reversible cyclic
random walk evolving on a potential field in a simplex, as the
mean-field Ising model \cite{cgov} is mapped to a one-dimensional
reversible random walk on the interval, we examine in this article the
metastable behavior of non-reversible cyclic random walks evolving in
a potential field defined on a bounded domain of $\bb R^d$.

We present a complete description of the metastable behavior of this
model, as it has been done in the reversible setting in \cite{LMT},
following the works of \cite{BEGK1, BEGK2}. In particular, we prove
the Eyring-Kramers transition rate formula \cite{E,K} which provides
the sharp sub-exponential pre-factor to the expectation of the hitting
time of the stable states starting from a metastable state. This is
done in the general case in which several wells may have the same
depth. We refer to \cite{Ber} for a historical review on the
derivation of the Eyring-Kramers formula.

Since the works of Bovier, Eckhoff, Gayrard and Klein \cite{BEGK1,
  BEGK2}, which established the link between potential theory of
Markov chains and metastability, it is known that one of the major
difficulties in the proof of the metastable behavior consists in
obtaining sharp estimates for the capacity between different sets of
wells. In the present non-reversible context, the Dirichlet and the
Thomson principles \cite{GL, Slo} provide double variational formulas
for the capacity in terms of flows and functions. These results also
identify the optimal flows and functions which solve the variational
problem. In particular, the computation of the capacity is reduced to
the determination of good approximations of the equilibrium potentials
between wells and of the associated flows.

It turns outs that for random walks in potential fields \cite{BEGK1,
  BEGK2, LMT}, the equilibrium potential drastically changes from $0$
to $1$ in a mesoscopic neighborhood around saddle points between local
minima, and that all the analysis is reduced to a detailed
examination of the dynamics around the saddles points.

To our knowledge, this work presents the first rigorous derivation of
the Eyring-Kramers formula in a non-reversible setting. It shows that
the role played by the non-negative eigenvalue of the Hessian of the
potential around the saddle point in reversible dynamics is replaced
in non-reversible dynamics by the non-negative eigenvalue of the
Jacobian of the asymptotic drift.

\section{Notation and Results}
\label{sec:nota-result}

\smallskip\noindent{\bf The domain and potential field.} Let $\Xi$ be
an open, bounded and connected domain of $\mathbb{R}^{d}$ with
piecewise $C^1$ boundary, denoted by $\partial\Xi$.  Denote by
$\overline{\Xi}=\Xi\cup\partial\Xi$ the closure of $\Xi$. Let $F:
\overline{\Xi} \to \bb R$ be a potential such that

\begin{enumerate}
\item $F$ is a twice-differentiable function which has finitely many
  critical points at $\Xi$, and no critical points at
  $\partial\Xi$. Furthermore, $\nabla F(\boldsymbol{x})\cdot
  \bs n(\boldsymbol{x})>0$ for all $\boldsymbol{x} \in \partial \Xi$, where
  $\bs n(\boldsymbol{x})$ represents the exterior normal vector to the boundary.
\item The second partial derivatives of $F$ are Lipschitz continuous on
  every compact subsets of $\Xi$.
\item At each local minimum, all eigenvalues of $\mbox{Hess }F$ are
  strictly positive.
\item At each saddle point, one eigenvalue of $\mbox{Hess }F$ is
  strictly negative, all the others being strictly positive.
\end{enumerate}

Consider a sequence of functions $\{G_{N}: N\ge 1\}$, $G_N:
\overline{\Xi} \to \bb R$. Assume that on each compact subsets of
$\Xi$, the sequence is uniformly Lipschitz, and converges uniformly,
as $N\uparrow\infty$, to a continuous function $G: \overline{\Xi} \to
\bb R$.  Let $F_{N} = F+ (1/N) G_{N}$.

\smallskip\noindent{\bf The dynamics.} Consider a cycle
$\gamma=(\boldsymbol{z}_{0} , \boldsymbol{z}_{1}, \dots,
\boldsymbol{z}_{L}= \boldsymbol{z}_{0})$ in $\mathbb{Z}^{d}$ without
self intersections starting from the origin, $\bs z_0 = \bs 0$. Assume
that the points $\bs w_j = \bs z_{j+1} - \bs z_j$, $0\le j <L$,
generate $\bb Z^d$ in the sense that any point $\bs x\in\bb Z^d$ can
be written as a linear combination of the points $\bs w_j$.  Let
$\gamma^{N}$ be the cycle $\gamma$ scaled by $N^{-1}$,
$\gamma^{N}=N^{-1}\gamma$, and let
$\boldsymbol{z}_{i}^{N}=N^{-1}\boldsymbol{z}_{i}$, $0\le i<L$ be the
vertices of $\gamma^N$.  Denote by $\gamma_{\boldsymbol{x}}^{N}$ the
cycle $\gamma^{N}$ translated by $\boldsymbol{x} \in
(N^{-1}\mathbb{Z})^{d}$: $\gamma_{\boldsymbol{x}}^{N} =\{\bs x + \bs z
: \bs z\in \gamma^N\}$.

Denote by $\widetilde{\Xi}_{N}$ the discretization of
$\overline{\Xi}$: $\widetilde{\Xi}_{N}= \overline{\Xi}
\cap(N^{-1}\mathbb{Z})^{d}$.  Define $\widehat{\Xi}_{N}$ as the set of
points $\bs x\in (N^{-1}\mathbb{Z})^d$ such that
$\gamma_{\boldsymbol{x}}^{N} \subset \widetilde{\Xi}_{N}$:
\begin{equation*}
\widehat{\Xi}_{N} \;=\; 
\{\boldsymbol{x}\in (N^{-1}\mathbb{Z})^d :
\gamma_{\boldsymbol{x}}^{N}\subset\widetilde{\Xi}_{N}\}\;,
\quad\text{and let}\quad 
\Xi_{N}=\bigcup_{\boldsymbol{x}\in\widehat{\Xi}_{N}}\gamma_{\boldsymbol{x}}^{N}\;.
\end{equation*}
In other words, $\Xi_N$ is obtained by removing points of
$\widetilde{\Xi}_N$ which are not visited by cycles
$\gamma_{\bs{x}}^N$, $\bs{x}\in \widehat{\Xi}_N$.  The set
$\widehat{\Xi}_N$ can be regarded as the interior of $\Xi_N$.  We
refer to Figure \ref{fig:00} for this construction.  We define below a
$\Xi_N$-valued continuous-time Markov chain.

\begin{figure}
\centering
\includegraphics[scale=0.13]{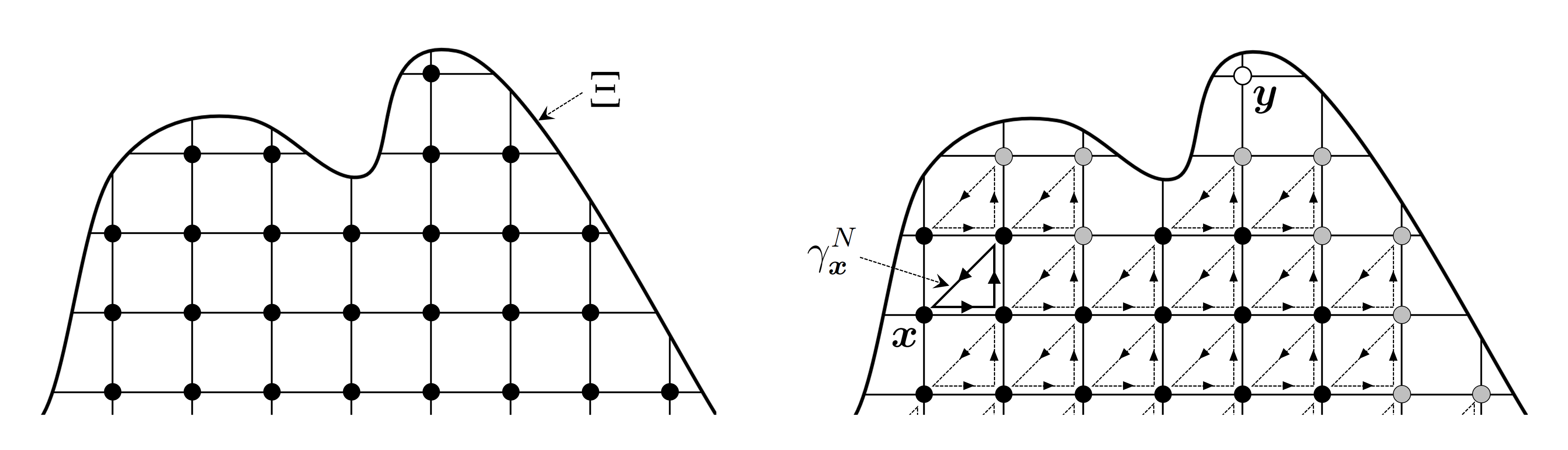}\protect
\caption{\label{fig:00}$\Xi_{N}$ of $\Xi$ with $d=2$
  and $\gamma=((0,0),\,(0,1),\,(1,1),\,(0,\,0)$: The left figure
  represents the set $\Xi$ and its discretization
  $\widetilde{\Xi}_{N}=\overline{\Xi}\cap(N^{-1}\mathbb{Z})^{2}$.  In
  the right figure, points belonging the set $\widehat{\Xi}_{N}$ are
  colored in black. For example, $\boldsymbol{x}\in\widehat{\Xi}_{N}$
  since the cycle $\gamma_{\boldsymbol{x}}^{N}$ is contained in
  $\widetilde{\Xi}_{N}$.  These points also belong to $\Xi_{N}$ by
  definition. In addition, points colored in gray also belong to
  $\Xi_{N}$, since they lie on a cycle $\gamma_{\boldsymbol{z}}^{N}$
  for some $\boldsymbol{z}\in\widehat{\Xi}_{N}$.  Note that
  $\boldsymbol{y}\in\widetilde{\Xi}_{N}$ and
  $\boldsymbol{y}\notin\Xi_{N}$.}
\end{figure}

Let $\mu_N$ be the probability measure on $\Xi_{N}$
given by
\begin{equation}
\label{28}
\mu_{N}(\boldsymbol{x}) \;=\;
\frac 1{Z_{N}} \, \exp\{-NF_{N}(\boldsymbol{x})\}\;,\quad
\boldsymbol{x}\in\Xi_{N}\;,
\end{equation}
where $Z_{N}$ is the normalizing factor. 

For each $\boldsymbol{x}\in\widehat{\Xi}_{N}$, denote by
$\mathcal{L}_{N,\boldsymbol{x}}$ the generator on the cycle
$\gamma_{\boldsymbol{x}}^{N}$ given by
\begin{equation}
\label{29}
(\mathcal{L}_{N,\boldsymbol{x}}f)
(\boldsymbol{x}+\boldsymbol{z}_{j}^{N}) \;=\; 
R_{j}^{N}(\boldsymbol{x})
\left[f(\boldsymbol{x}+\boldsymbol{z}_{j+1}^{N})
-f(\boldsymbol{x}+\boldsymbol{z}_{j}^{N})\right]\;, \quad 0\le j<L\;, 
\end{equation}
where 
\begin{equation*}
R_{j}^{N}(\boldsymbol{x}) \;=\; 
\exp\left\{ -N\left(\overline{F}_{N}(\boldsymbol{x})-
F_{N}(\boldsymbol{x}+\boldsymbol{z}_{j}^{N})\right)\right\} \;, \;\;\;
\overline{F}_{N}(\boldsymbol{x}) \;=\; 
\frac{1}{L}\sum_{i=0}^{L-1}F_{N}(\boldsymbol{x}+\boldsymbol{z}_{i}^{N})\;.
\end{equation*}
We extend the definition of $R_j$ to $(N^{-1} \bb Z)^d$ by setting
$R_j(\bs x)=0$ if $\bs x\not\in \widehat{\Xi}_{N}$.

Clearly, the measure $\mu_N$ restricted to the cycle
$\gamma_{\boldsymbol{x}}^{N}$ is the unique stationary state of the
continuous-time Markov chain whose generator is
$\mathcal{L}_{N,\boldsymbol{x}}$. Denote by $X_N(t)$ the
$\Xi_{N}$-valued, continuous-time Markov chain whose generator is
given by
\begin{equation}
\label{eq: generator}
\mathcal{L}_{N} \;=\; \sum_{\boldsymbol{x}\in\widehat{\Xi}_{N}} 
\mathcal{L}_{N,\boldsymbol{x}}\;.
\end{equation} 
It is easy to check that the measure $\mu_{N}$ given by (\ref{28}) is
a stationary probability measure for the generator $\mathcal{L}_{N}$.
It is reversible if and only if the cycle $\gamma$ has length $2$.

We have three reasons to examine such dynamics.  On the one hand,
cycle dynamics are the simplest generalization of reversible dynamics.
In statistical mechanics, starting from an Hamiltonian, one introduces
a reference measure and then a dynamics which satisfies the detailed
balance conditions to ensure that the evolution is stationary
(actually, reversible) with respect to the reference measure. The
cycle dynamics provide a natural larger class of evolutions which are
stationary with respect to the reference measure. 

These cycle dynamics appeared before in many different contexts. We
refer to Sections 3.3, 5.3 of \cite{klo}, to \cite{Lan2, ls2} and to
the citations of \cite[Section 3.8]{klo} for cycle dynamics in the
context of random walks in random environment and of interacting
particle systems.  Actually, \cite[Lemma 4.1]{lx} asserts that in
finite state spaces the generator of a irreducible Markov chain can be
expressed as the sum of generators of cycle dynamics.

Secondly, cycle random walks is a natural model in which to test the
Dirichlet and the Thomson principles for the capacity in the context
of non-reversible dynamics because these variational problems are
expressed in terms of divergence-free flows whose building blocks are
cycle flows.

Finally, as pointed-out below in Remark \ref{rm6}, in a proper
scaling limit, the cycle dynamics considered here converges to a
non-reversible diffusion. In particular, the approximations of the
optimal flow and of the equilibrium potentials derived in Sections
\ref{sec3}, \ref{sec:flow} provide good insight for the continuum
case.

Denote by $R_N(\bs x, \bs y)$, $\lambda_N(\bs x)$, $\bs x$, $\bs
y\in\Xi_{N}$, the jump rates and the holding rates of the
Markov chain $X_N(t)$, respectively. A simple computation shows 
\begin{equation}
\label{11}
R_{N}(\boldsymbol{x},\boldsymbol{y})
\;=\; \sum_{i=0}^{L-1}\mathbf{1}\{\boldsymbol{z}_{i+1}^{N}- 
\boldsymbol{z}_{i}^{N}=\boldsymbol{y}-\boldsymbol{x}\} 
R_i^N (\bs{x}-\bs{z}_i^N)\;, \quad 
\lambda_N(\bs x) \;=\; \sum_{i=0}^{L-1} R_i^N (\bs{x}-\bs{z}_i^N)   
\;.
\end{equation}

\smallskip\noindent{\bf The law of large numbers}. Denote by
$\mathbb{P}_{\boldsymbol{x}}^{N}$ (resp. $\mb P^N_{\bs x}$), $\bs
x\in\Xi_N$, the law of the Markov chain $X_N(t)$ (resp. the speeded-up
Markov chain $X_N(tN)$) starting at $\boldsymbol{x}$. Expectation with
respect to $\mathbb{P}_{\boldsymbol{x}}^{N}$, $\mb P^N_{\bs x}$ is
represented by $\mathbb{E}_{x}^{N}$, $\mb E^N_{\bs x}$, respectively.

Let $\{\boldsymbol{x}_{N}\}_{N\in\mathbb{N}}$ be a sequence of points
in $\Xi_{N}$ which converges to some $\boldsymbol{x}\in\Xi$. The
sequence $\{\mb P_{\boldsymbol{x}_{N}}^{N} : N\ge 1\}$ converges to
the Dirac mass on the deterministic path $\boldsymbol{x}(t)$, which
solves the ordinary differential equation
\begin{equation*}
\left\{
\begin{aligned}
&\dot{\boldsymbol{x}}(t)\;=\;-\, b(\boldsymbol{x}(t))\;,\\
& \boldsymbol{x}(0)\;=\;\boldsymbol{x}\;,      
\end{aligned}
\right.
\end{equation*}
where 
\begin{equation}
\label{04}
b(\boldsymbol{x}) \;=\; -\sum_{j=0}^{L-1}e^{(\boldsymbol{z}_{j}-
\bar{\boldsymbol{z}})\cdot\nabla F(\boldsymbol{x})}
(\boldsymbol{z}_{j+1}-\boldsymbol{z}_{j})\;, \quad
\bar{\boldsymbol{z}}=\frac{1}{L}\sum_{i=0}^{L-1}\boldsymbol{z}_{i}\;.
\end{equation}

\smallskip\noindent{\bf The time-reversed or adjoint dynamics.}  Denote by
$\mathcal{L}_{N,\boldsymbol{x}}^{*}$, $\mathcal{L}_{N}^{*}$, $\bs
x\in\widehat{\Xi}_{N}$, the adjoints of the generators
$\mathcal{L}_{N,\boldsymbol{x}}$, $\mathcal{L}_{N}$ in $L^2(\mu_N)$,
respectively. An elementary computation shows that for $0\le j<L$,
\begin{equation*}
(\mathcal{L}_{N,\boldsymbol{x}}^{*}f)(\boldsymbol{x}+\boldsymbol{z}_{j}^{N}) 
\;=\; R_{j}^{N}(\boldsymbol{x})
\left[f(\boldsymbol{x}+\boldsymbol{z}_{j-1}^{N})
-f(\boldsymbol{x}+\boldsymbol{z}_{j}^{N})\right]\;, \quad
\mathcal{L}_{N}^{*} \;=\; 
\sum_{\boldsymbol{x}\in\Xi_{N}}\mathcal{L}_{N,\boldsymbol{x}}^{*}\;.
\end{equation*}

Denote by $X_{N}^{*}(t)$ the $\Xi_N$-valued, continuous-time Markov
chain whose generator is $\mathcal{L}_{N}^{*}$.  As for the direct
dynamics, denote by $\mb P^{*,N}_{\bs x}$, $\bs x\in\Xi_N$, the law of
the speeded-up Markov chain $X^*_N(tN)$ starting at $\boldsymbol{x}$.

Let $\{\boldsymbol{x}_{N}\}_{N\in\mathbb{N}}$ be a sequence of points
in $\Xi_{N}$ which converges to some $\boldsymbol{x}\in\Xi$. The
sequence $\{\mb P_{\boldsymbol{x}_{N}}^{*,N} : N\ge 1\}$ converges to
the Dirac mass on the deterministic path $\boldsymbol{x}(t)$, which
solves the ordinary differential equation
\begin{equation*}
\left\{
\begin{aligned}
&\dot{\boldsymbol{x}}(t)\;=\;-\, b^*(\boldsymbol{x}(t))\;,\\
& \boldsymbol{x}(0)\;=\;\boldsymbol{x}\;,      
\end{aligned}
\right.
\end{equation*}
where 
\begin{equation}
\label{04b}
b^*(\boldsymbol{x}) \;=\; -\sum_{j=0}^{L-1}
e^{(\boldsymbol{z}_{j}-\bar{\boldsymbol{z}})\cdot\nabla F(x)}  
\, (\boldsymbol{z}_{j-1}-\boldsymbol{z}_{j}) \;.
\end{equation}

Note that the macroscopic behavior of the dynamics differs completely
from the macroscopic behavior of the time-reversed dynamics. Changing
the clock direction of the jumps affects dramatically the global
behavior of the chain.

\begin{figure}
\centering
\includegraphics[scale=0.17]{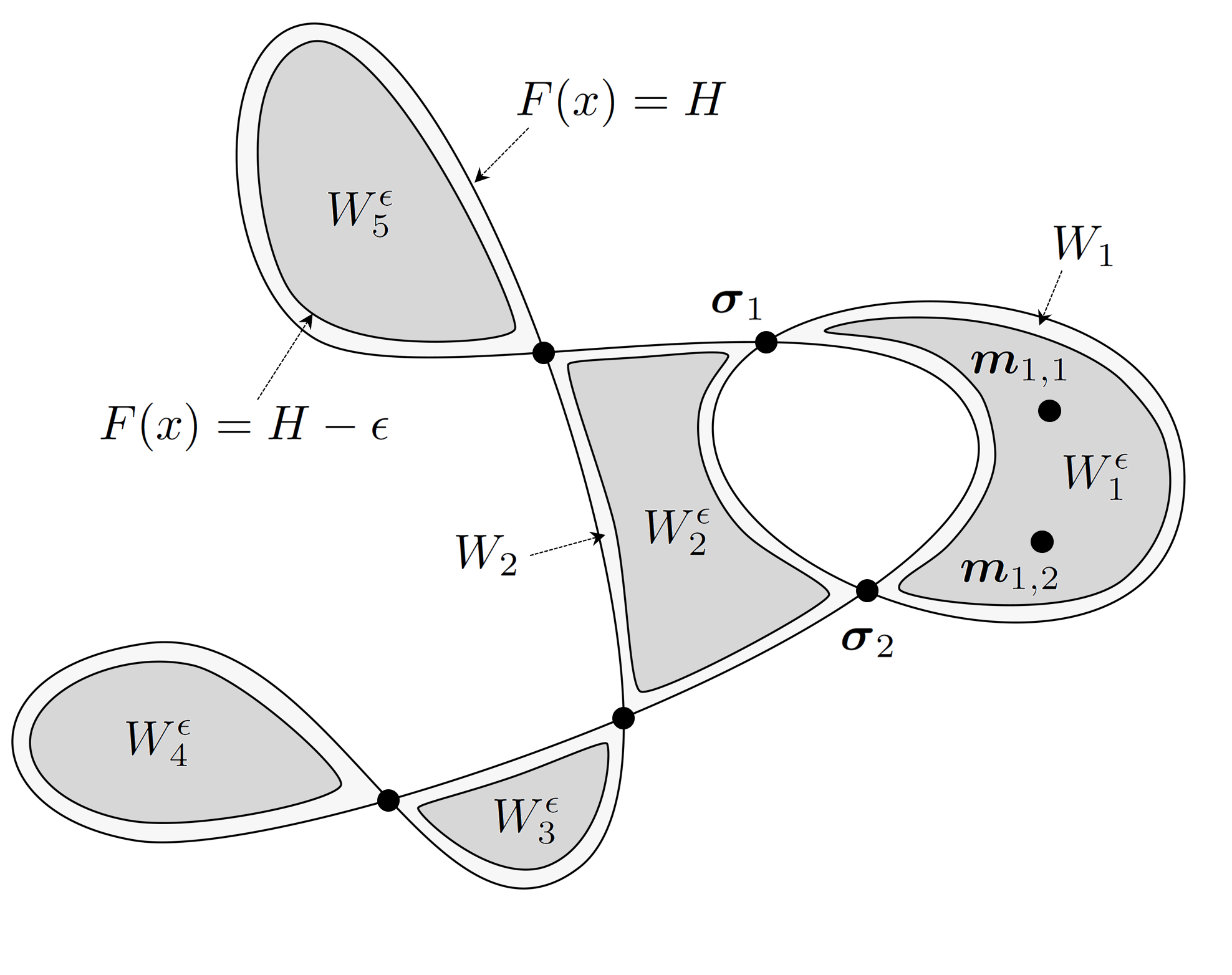}\protect
\caption{\label{fig:21}The structure of the valleys.}
\end{figure}

\smallskip\noindent{\bf The valleys.} Fix $H$ in $\mathbb{R}$, and assume
that the set of saddle points of height $H$, denoted by $\mf S$, is
non-empty:
\begin{equation*}
\mathfrak{S} \;=\; \{\boldsymbol{\sigma}\in\Xi:
\boldsymbol{\sigma}\mbox{ is saddle point of \ensuremath{F}\,and 
\ensuremath{F(\boldsymbol{\sigma})=H}\}} \;\not=\; \varnothing\;.
\end{equation*}
Let
\begin{equation*}
\widehat{\Omega} \;=\; \{\boldsymbol{z}\in\Xi:\,F(\boldsymbol{z})\,\le\,H\}\;, 
\end{equation*}
and let $\Omega$ be one of the connected components of
$\widehat{\Omega}$. The connected components of the interior of
$\Omega$ are denoted by $\mathring{W}_{1}, \dots,\mathring{W}_{M}$,
and the closure of $\mathring{W}_{i}$ is denoted by $W_{i}$. Then,
$\mathfrak{S}_{i,j}=W_{i}\cap W_{j}$, $i\neq j$, is a subset of
$\mathfrak{S}$.  We assume that the sets $\mathfrak{S}_{i,j},\,i\neq
j$, are pairwise disjoint, \textit{i.e}., that no saddle point belongs
to three sets $W_{i}$'s. Fix $\epsilon$ small enough to prevent the
existence of saddle points of $F$ of height between $H-\epsilon$ and
$H$, and denote the connected components of
\begin{equation*}
\Omega^{\epsilon} \;=\; \{\boldsymbol{z}\in\Xi: F(\boldsymbol{z})\,
\le\,H-\epsilon\}
\end{equation*}
by $W_{1}^{\epsilon}, \dots, W_{M}^{\epsilon}$, where
$W_{i}^{\epsilon}\subset W_{i}$. See Figure \ref{fig:21}.

Define the metastable sets $\mathcal{E}_{N}^{1}, \dots, \mathcal{E}_{N}^{M}$
as
\begin{equation*}
\mathcal{E}_{N}^{i} \;=\;  W_{i}^{\epsilon}\cap\Xi_{N} \quad\text{and
  let}\quad 
\mathcal{E}_{N} \;=\; \bigcup_{k=1}^{M}\mathcal{E}_{N}^{k}\;.
\end{equation*}
Denote the deepest local minima of $W_i^{\epsilon}$ by $\{\bs
m_{i,1}, \dots, \bs m_{i,q_{i}}\}$, and set $h_{i}=F(\bs
m_{i,1})$. Let
\begin{equation}
\label{eq:2}
\nu_{i} \;=\;  \sum_{k=1}^{q_{i}} \frac{e^{-G(\bs m_{i,k})}} 
{\sqrt{\det(\mbox{Hess }F) ( \bs m_{i,k})}}\; \cdot
\end{equation}
It is shown in \cite[Section 6]{LMT} that for every $1\le i\le M$,
\begin{equation}
\label{thm: estimate of meta mass}
\frac{Z_{N}}{(2\pi
  N)^{\frac{d}{2}}}e^{Nh_{i}}\mu_{N}(\mathcal{E}_{N}^{i}) \;=\; 
\big[ 1 + o_{N}(1) \big]\, \nu_{i} \;,
\end{equation}
where $o_{N}(1)$ is an expression which vanishes as $N\uparrow\infty$.

Let $\bb A$ be the $d\times d$ matrix given by 
\begin{equation}
\label{45b}
\mathbb{A} \;=\;
\sum_{i=0}^{L-1} (\boldsymbol{z}_{i}-\boldsymbol{z}_{i+1})\,
\boldsymbol{z}_{i}^{\dagger}\;,
\end{equation}
where $M^\dagger$ represents the transposition of the matrix or vector
$M$, and where the points $\bs z_i$ form the cycle $\gamma$ introduced
at the beginning of this section.  Let
$\mathbb{M}_{\boldsymbol{\sigma}} = \mathbb{A}\,
\mathbb{H}_{\boldsymbol{\sigma}}$, where $\bb H_{\bs \sigma} =
(\text{Hess } F)(\bs \sigma)$.  We prove in Lemma \ref{lem: property
  of M} that the matrix $\mathbb{M}_{\boldsymbol{\sigma}}$ has only
one negative eigenvalue, denoted by $- \mu_{\boldsymbol{\sigma}}$.
For each $\boldsymbol{\sigma}\in\mathfrak{S}$, define
\begin{equation}
\label{45}
\omega_{\boldsymbol{\sigma}} \;=\; \frac{\mu_{\boldsymbol{\sigma}}}
{\sqrt{-\det\, \mathbb{H}_{\boldsymbol{\sigma}}}}\,
e^{-G(\boldsymbol{\sigma})} \;, \;\; 
\omega_{i,j}\;=\;\sum_{\boldsymbol{\sigma}\in\mathfrak{S}_{i,j}}
\omega_{\boldsymbol{\sigma}}\;, \;\; \forall i\not = j\;,\;\;\mbox{and}\;\;\;
\omega_{i,i}=0\;,\;\;\forall i\;.
\end{equation}

\smallskip\noindent{\bf Hitting times and capacities.}  
Denote by $H_{\mc A}$, $H^+_{\mc A}$, the hitting time of and the return time to
a subset $\mc A$ of $\Xi_N$, respectively,
\begin{equation}
\label{32}
H_{\mc A} \;=\; \inf\{t>0:\,X_{N}(t)\in \mc A\}\;, \quad
H^+_{\mc A} \;=\; \inf\{t>\tau_1:\,X_{N}(t)\in \mc A\}\;,
\end{equation}
where $\tau_1$ stands for the time of the first jump:
$\tau_1 = \inf\{s>0 : X_N(s) \not = X_N(0)\}$.

For two disjoint subsets $\mc A$, $\mc B$ of $\Xi_{N}$, denote by
$\Cap_N(\mc A , \mc B )$ the capacity between $\mc A$ and $\mc B$ for
the Markov chain $X_N(t)$:
\begin{equation*}
\Cap_N(\mc A ,\mc B ) \;=\; \sum_{\bs x\in A} \mu_N(\bs x) \,
\lambda_N(\bs x) \, \bb P^N_{\bs x} [H_{\mc B} < H^+_{\mc A}]\;.
\end{equation*}

Let $S=\{1,\,2,\,\dots,\,M\}$, $\omega_{i}=\sum_{k\in S}\omega_{i,k}$,
$i\in S$, and $\overline{\omega}_i = \omega_i/\sum_{j\in S}
\omega_j$. Denote by $\{Y(t) : t\ge 0\}$ the $S$-valued,
continuous-time Markov chain which jumps from $i$ to $j$ at rate
$\omega_{i,j}/\overline{\omega}_i$, and by $\mathbf{P}_{k}^{Y}$, $k\in
S$, the law of the chain $Y(t)$ starting from $k$. Note that the
probability measure $\mu(i)=\overline{\omega}_i$, $i\in S$, is
stationary, in fact reversible, for the chain $Y(t)$. Let
$\textup{cap}_{Y}(\cdot,\cdot)$ be the capacity with respect to $Y$:
\begin{equation*}
\Cap_Y(A,B) \;=\; \sum_{i\in A} \omega_i \, 
\mb P^Y_{i} [H_B < H^+_A]\;,
\end{equation*}
where $A$ and $B$ are disjoint subsets of $S$. These capacities can be
computed by solving a system of at most $M-2$ linear equations.

Let $\mathcal{E}_{N}(A)$, $A\subset S$, be given by
\begin{equation*}
\mathcal{E}_{N}(A)\;=\;\bigcup_{i\in A}\mathcal{E}_{N}^{i}\;.
\end{equation*}
The following sharp estimate for capacities between metastable sets is proven in Section \ref{sec7}.

\begin{theorem}
\label{s133} 
For every disjoint subsets $A,\,B$ of $S$, 
\begin{equation*}
\frac{Z_{N}}{(2\pi N)^{\frac{d}{2}-1}}e^{NH}
\textup{cap}_{N}(\mathcal{E}_{N}(A),\mathcal{E}_{N}(B))
\;=\; \big[1+o_N(1)\big] \, \textup{cap}_{Y}(A,B)\;.
\end{equation*}
\end{theorem}

\smallskip\noindent{\bf Metastability.}  Let $\hat{\theta}_{i}=
H-h_{i}$, $i\in S$, be the depth of the metastable set
$\mathcal{E}_{N}^{i}$, and let $\theta_{1}<\theta_{2} <\cdots<
\theta_{l}$ be the increasing enumeration of the sequence
$\hat{\theta}_{i}$, $i\in S$. Hence, $\theta_{i}$ represents the
$i$-th smallest depth of the metastable sets and
\begin{equation*}
\{\hat{\theta}_{i}:i\in S\}\;=\;\{\theta_{m}:1\le m\le l\}\;.
\end{equation*}

Let $T_{m}=\{i\in S: \hat{\theta}_{i}= \theta_{m}\}$ be the set of
metastable sets whose depth is equal to $\theta_m$, and let
$S_{m}=\bigcup_{m\le k \le l}T_{k}$, $1\le m\le l$. Note that
$S_{1}=S$.  The metastable behavior corresponding to the depth
$\theta_{m}$ can be represented as a continuous time Markov chain on
$S_{m}$ where the points of $S_{m+1}$ form the set of absorbing points
for this chain.

For $1\le m\le l$ and $i\neq j\in S_{m}$, let 
\begin{equation}
\label{c_m}
c_{m}(i,j)\;=\; \frac{1}{2} \big\{ \textup{cap}_{Y}
(\{i\},S_{m}\setminus\{i\}) \,+\, \textup{cap}_{Y}
(\{j\},S_{m}\setminus\{j\}) \,-\,
\textup{cap}_{Y}(\{i,j\}, S_{m}\setminus\{i,j\})
\big\}\;.
\end{equation}
Denote by $\Psi_{N}^{(m)}:\Xi_{N}\rightarrow S_{m}\cup\{N\}$ the
projection given by
\begin{equation*}
\Psi_{N}^{(m)}(\boldsymbol{x})\;=\;
\sum_{i\in S_{m}}i\,
\mathbf{1}\{\boldsymbol{x}\in\mathcal{E}_{N}^{i}\}
\;+\; N\, \mathbf{1}\{\boldsymbol{x}\in\Xi_{N}
\setminus\mathcal{E}_{N} (S_{m})\}\;,
\end{equation*}
and by $Y_{N}^{(m)}(t)$ the $S_{m}\cup\{N\}$-valued, hidden
Markov chain obtained by projecting the random walk $X_N(t)$ with
$\Psi_{N}^{(m)}$:
\begin{equation*}
Y_{N}^{(m)}(t)\;=\;\Psi_{N}^{(m)}(X_{N}(t))\;.
\end{equation*}
Here and below we use the notation $X$, $Y$ to represent
continuous-time Markov chains whose state space is are subsets of
$\Xi_N$, $S\cup \{N\}$, respectively.

Denote by $\mathbf{Q}_{k}^{(m)}$, $k\in S_{m}$, the law of the
$S_m$-valued, continuous-time Markov chain which starts from $k$ and
whose jump rates are given by
\begin{equation}
\label{r_m}
r_{m}(i,j)\;=\;\mathbf{1}\{i\in T_{m}\}
\frac{c_{m}(i,j)}{\nu_{i}}\;, \;\; i\neq j\in S_m \;.
\end{equation}
Note that the points in $S_{m}\setminus T_{m}=S_{m+1}$ are absorbing
states for $\mathbf{Q}_{k}^{(m)}$. Finally, let 
\begin{equation}
\label{eq:1}
\beta_{N}^{(m)}\;=\; 2\pi N\, \exp\{\theta_{m}N\}
\end{equation}
and recall from \cite{Lan1} the definition of the soft topology. 

\begin{theorem}
\label{mr1}
Fix $1\le m\le l$, $k\in S_{m}$ and a sequence of points
$\{\boldsymbol{x}_{N}\}$, $\boldsymbol{x}_{N}\in\mathcal{E}_{N}^{k}$
for all $N$. Then, under $\mathbb{P}_{\boldsymbol{x}_{N}}^{N}$, the
law of the rescaled projected process $Y_{N}^{(m)}(\beta_{N}^{(m)}t)$
converges to $\mathbf{Q}_{k}^{(m)}$ in the soft topology.
\end{theorem}

\begin{remark}
\label{rm01}
A computation of the capacity of the chain $Y(t)$ shows that
$c_{1}(i,j)=\omega_{i,j}$.
\end{remark}

\begin{remark}
\label{rm02}
In view of (\ref{c_m}) and (\ref{r_m}), the metastable behavior of the
random walk $X_N(t)$ is similar to the one of the reversible random
walk in a potential field, discussed in \cite{LMT}, except for the
definition of the quantity $\omega_{\boldsymbol{\sigma}}$.

However, the proofs in the non-reversible case present two major
additional difficulties compared to the reversible case. On the one
hand, the computation of the capacities, presented in Theorem
\ref{s133}, which rely on the Dirichlet and on the Thomson principles,
are much more delicate, as these principles involve double
variational problems.

On the other hand, while in the reversible case the asymptotic jump
rates (\ref{r_m}) can be expressed in terms of the capacities,
computed in Theorem \ref{s133}, in the non-reversible case, the
derivation of the asymptotic jump rates requires a detailed analysis
of the behavior of the equilibrium potential. We present in Section
\ref{sec8} a robust framework, which can be useful in other contexts,
to obtain sharp estimates of the mean jump rate in case of the
non-reversible dynamics.
\end{remark}

\smallskip\noindent{\bf{Eyring-Kramers formula}}. Fix $1\le u\le l$
and $i\in T_{u}$.  Select a minimum $\boldsymbol{m}_{i,k}$, $1\le k\le
q_{i}$ of $F$ on $W_{i}^{\epsilon}$ and denote this point by
$\boldsymbol{m}_{i}$. Define the set of local minima of $F$ on
$\Xi\setminus W_{i}$ by $\mathcal{\widehat{M}}_{i}$, and let $\mc M_i$
be the points in $\mathcal{\widehat{M}}_{i}$ which are below $\bs
m_i$:
\begin{equation*}
\mathcal{M}_{i}\;=\;\left\{ \boldsymbol{x}\in\widehat{\mathcal{M}}_{i}
:F(\boldsymbol{x})\le F(\bs m_i) \right\}\;.
\end{equation*}

Denote by $[\boldsymbol{x}]_{N}$ the nearest point in $\Xi_{N}$ of
$\boldsymbol{x}\in\Xi$. If there are several nearest points, choose
one of them arbitrarily. Denote by $\left[\mathcal{M}_{i}\right]_{N}$
the discretization of the set $\mathcal{M}_{i}$:
$\left[\mathcal{M}_{i}\right]_{N}=\left\{
  [\boldsymbol{x}]_{N}:\boldsymbol{x}\in\mathcal{M}_{i}\right\} $.

\begin{theorem}
\label{ek}
For $1\le u\le l$ and $i\in T_{u}$, 
\begin{equation}
\label{ek1}
\mathbb{E}_{[\boldsymbol{m}_{i}]_{N}}^{N}
\left[H_{\left[\mathcal{M}_{i}\right]_{N}}\right]\;=\;
\big[1+o_{N}(1)\big] \, \frac {\nu_i \, \beta_{N}^{(u)}}
{\sum_{j\in S_{u}}c_{u}(i,j)}\;\cdot 
\end{equation}
\end{theorem}

If the potential $F$ has only two local minima and one saddle point between them, the right hand side of the previous equation takes the form of the celebrated Eyring-Kramers formula.  More precisely, assume 
that $M=2$ and that the wells $W_{1}$ and $W_{2}$ contain
only one local minima, denoted by $\boldsymbol{m}_{1}$ and
$\boldsymbol{m}_{2}$, respectively. Assume that
$F(\boldsymbol{m}_{1})\ge F(\boldsymbol{m}_{2})$, and denote by
$\boldsymbol{\sigma}$ the unique saddle point located between
$\boldsymbol{m}_{1}$ and $\boldsymbol{m}_{2}$:
$\{\boldsymbol{\sigma}\}=\overline{W}_{1}\cap\overline{W}_{2}$.  By
\eqref{eq:2}, \eqref{45}, \eqref{eq:1}, \eqref{ek1} and Remark
\ref{rm01},
\begin{align*}
&\mathbb{E}_{[\boldsymbol{m}_{1}]_{N}}^{N}
\left[H_{[\boldsymbol{m}_{2}]_{N}}\right]\\
&\quad =\;
\big[1+o_{N}(1)\big] \, e^{N[F(\boldsymbol{\sigma})-F(\boldsymbol{m}_{1})]} \, \frac{2\pi N}{\mu_{\boldsymbol{\sigma}}}
\sqrt{\frac {- \det(\text{Hess }F)(\bs \sigma)}
{\det(\textup{Hess }F)(\boldsymbol{m}_{1})}}
\,e^{G(\bs{\sigma})-G(\bs{m}_1)}\;.
\end{align*}

The notable difference of this formula with respect to the
Eyring-Kramers formula for the reversible dynamics is the appearance
of $\mu_{\boldsymbol{\sigma}}$, instead of the absolute value of the
negative eigenvalue of the Hessian of the potential $F$ at
$\boldsymbol{\sigma}$. This replacement was anticipated by the recent
work of Bouchet and Reygner \cite{BR} in the context of non-reversible
Freidlin-Wentzell type of diffusions. Another difference is the
appearance of $\exp\{G(\bs{\sigma})-G(\bs{m}_1)\}$. This new term
coincide with the so-called non-Gibbsianness factor of \cite[display
(1.10)]{BR}. This term is introduced in order to take into account the
fact that the invariant measure $\mu_N$ is not exactly Gibbsian. To
our knowledge, Theorem \ref{ek} is the first rigorous proof of the
Eyring-Kramers formula for non-reversible dynamics.

\smallskip\noindent{\bf Applications, remarks and extensions.} We
conclude this section with some comments on the results. 

In a forthcoming paper \cite{ls2}, we use the results presented in
this paper to investigate the metastable behavior of a planar,
mean-field Potts model.

\begin{remark}[Reversibility]
\label{rmf1}
The case in which the cycle $\gamma$ has length $2$ and the dynamics
is reversible case has been considered in \cite{BEGK2}. In this sense,
the results presented here provide a generalization of \cite{BEGK2}.
\end{remark}

\begin{remark}[Diffusive limit]
\label{rm6}
Consider the dynamics defined by the generator \eqref{eq: generator}
with the rates $R^N_j(\bs x)$ replaced by
\begin{equation*}
r_{j}^{N}(\boldsymbol{x}) \;=\; 
\exp\left\{ -\left(\overline{F}_{N}(\boldsymbol{x})-
F_{N}(\boldsymbol{x}+\boldsymbol{z}_{j}^{N})\right)\right\}\;.
\end{equation*}
Note that the factor $N$ in the exponent has been removed. In this
case the rescaled process $X_N(tN^2)$ converges to the diffusion on
$\Xi$ whose generator $\mc L$ is given by
\begin{equation*}
(\mc{L} f) (\bs{x}) \;=\;-\, [\bb{A} \, \nabla F(\bs{x})] \cdot 
(\nabla f) (\bs{x}) \,+\, (1/2) \sum_{1\le i,j\le d}
\bb{S}_{ij} \, (\partial^2_{x_{i},x_{j}} f)(\bs{x})\;,
\end{equation*}
where $\bb S$ is the matrix
\begin{equation*}
\bb{S} \;=\; \sum_{j=0}^{L-1} (\bs z_{j+1} - \bs z_j)\, (\bs z_{j+1} -
\bs z_j)^\dagger\;.  
\end{equation*}
In this context, the matrix $\bb M_{\bs \sigma}$, introduced above
\eqref{45}, is the Jacobian of the drift $b = \bb{A} \, \nabla
F(\bs{x})$ at $\bs{\sigma}$ We investigate in \cite{lms} the metastability behavior
of such diffusions.
\end{remark}

\begin{remark}
\label{rm1}
All results of this article hold if the jump rates
$R_{j}^{N}(\boldsymbol{x})$ introduced in \eqref{29} are replaced by
\begin{equation*}
\widetilde{R}_{j}^{N}(\boldsymbol{x}) \;=\; 
\exp\left\{ -N\left(F_{N}(\boldsymbol{x}+\bar{\boldsymbol{z}}^{N})
-F_{N}(\boldsymbol{x}+\boldsymbol{z}_{j}^{N})\right)\right\} \;,
\end{equation*}
where $\bar{\boldsymbol{z}}^{N} = L^{-1} \sum_{0\le i<L}
\boldsymbol{z}_{i}^{N}$.
\end{remark}

\begin{remark}
\label{rm2}
Although our presentation is limited for a specific level $H$ of
saddle points, the complete description of the structure of the wells
and of the saddle points corresponding to the potential $F(\cdot)$,
presented in the reversible setting in \cite{LMT}, holds for the model
introduced in this article.
\end{remark}

\begin{remark}[Multiple cycles]
\label{rm3}
Let $\gamma^{(1)}, \dots, \gamma^{(K)}$ be cycles on $\mathbb{Z}^{d}$
such that edges of these cycles generate $\mathbb{Z}^d$.  Denote by
$\mathcal{L}_{N}^{(k)}$, $1\le k\le K$, the corresponding cycle
generators and by $\widetilde{\mc L}_N$ the sum of these $K$
generators.  Denote by $\mathbb{A}^{(k)}$ the matrix introduced in
(\ref{45b}) associated to the cycle $\gamma^{(k)}$ and let
$\mathbb{\widetilde{A}} = \sum_{1\le k\le K} \mathbb{A}^{(k)}$.  The
matrix $\widetilde{\mathbb{A}}$ satisfies the condition of Lemma
\ref{lem: property of M} and therefore
$\mathbb{\widetilde{M}}_{\boldsymbol{\sigma}} =
\widetilde{\mathbb{A}}\mathbb{H}_{\boldsymbol{\sigma}}$ has only one
negative eigenvalue, denoted by
$-\widetilde{\mu}_{\boldsymbol{\sigma}}$.  By replacing the matrix
$\mathbb{A}$ by $\widetilde{\mathbb{A}}$, the arguments presented in
the next sections hold for this general model. The only required
modification in the statement of Theorem \ref{mr1} is the replacement
of $\mu_{\boldsymbol{\sigma}}$ in the definition of
$\omega_{\boldsymbol{\sigma}}$ in (\ref{45}) by
$\widetilde{\mu}_{\boldsymbol{\sigma}}$, defined above.
\end{remark}

\begin{remark}[Generator with weights]
\label{rm5}
Let $\widetilde{\mathcal{L}}_{N}$ be the generator given by
\[
\widetilde{\mathcal{L}}_{N} \;=\; 
\sum_{\boldsymbol{x}\in\widehat{\Xi}_{N}} w_{N}(\boldsymbol{x})\,
\mathcal{L}_{N,\boldsymbol{\boldsymbol{x}}}\;,
\]
where the weights $w_{N}:\Xi\rightarrow\mathbb{R}$
satisfy the following two conditions:
\begin{enumerate}
\item The sequence $w_{N}$ converges uniformly on every compact subset
  of $\Xi$ to a smooth function $w:\Xi\rightarrow\mathbb{R}$;
\item The sequence $w_{N}$ is uniformly Lipschitz on every compact
  subset of $\Xi$. 
\end{enumerate}

The core of the proof of Theorems \ref{s133} and \ref{mr1} consists in
a mesoscopic analysis around the saddle points. Under the conditions
above, in a mesoscopic neighborhood of a saddle point $\bs{\sigma}$,
the weights $w_{N}(\boldsymbol{x})$ are uniformly close to
$w(\boldsymbol{\sigma})$ and all the arguments of the next sections
can be carried through. The assertions of Theorems 
\ref{s133} and \ref{mr1} hold for this model by replacing
$\omega_{\boldsymbol{\sigma}}$ in (\ref{45}) by
\begin{equation*}
\omega_{\boldsymbol{\sigma}} \;=\; \frac{\mu_{\boldsymbol{\sigma}}}
{\sqrt{-\det\mathbb{H}_{\boldsymbol{\sigma}}}}\, 
w(\boldsymbol{\sigma})\, e^{-G(\boldsymbol{\sigma})}\;.
\end{equation*}
\end{remark}

The planar Potts model examined in \cite{ls2} is an example
of dynamics which enters in this framework. 

\begin{remark}[Generalized Potential]
\label{rm4}
In \cite{BEGK2}, the authors assumed two properties \textbf{(R1)} and
\textbf{(R2)} for the potential $F_{N}$, which are satisfied by the
potential of Curie-Weiss model with random external field.  We
acknowledge here that our result also holds under their assumption,
without changing the arguments or the statements of the results.
\end{remark}

We present at the end of the next section a sketch of the proof and a
brief description of the content of each section of the article.

\section{Sector condition, flows and capacities}
\label{sec:Preliminaries}

The proofs of Theorem \ref{s133} rely on variational formulas for the
capacities in terms of functions and flows, recently obtained in
\cite{GL, Slo}. We present in this section these formulas as well as
some properties of the generator needed in the next sections.

\smallskip\noindent{\bf Dirichlet Form and Sector Condition.}  Denote
by $\mathcal{D}_{N}(\cdot)$ the Dirichlet form corresponding to the
generator $\mathcal{L}_{N}$, namely
\begin{equation*}
\mathcal{D}_{N}(f) \;=\; 
\left\langle f,-\mathcal{L}_{N}f\right\rangle _{\mu_{N}}\;,
\end{equation*}
where $\< \cdot, \cdot\>_{\mu_N}$ represents the scalar product in
$L^2(\mu_N)$, and $f$ is a real function on $\Xi_{N}$. By (\ref{eq:
  generator}), we can decompose this Dirichlet form as
\begin{equation*}
\mathcal{D}_{N}(f) \;=\; \sum_{\boldsymbol{x}\in\widehat{\Xi}_{N}}
\left\langle f,-\mathcal{L}_{N,\boldsymbol{x}}f\right\rangle
_{\mu_{N}} \;=\;\sum_{\boldsymbol{x}\in\widehat{\Xi}_{N}}
\mathcal{D}_{N,\boldsymbol{x}}(f)\;,
\end{equation*}
where
\begin{equation}
\label{eq: Decomposition of Diri Form -2}
\mathcal{D}_{N,\,\boldsymbol{x}}(f) \;=\; \frac{1}{2 Z_{N}}
\, e^{-N\overline{F}_{N}(\boldsymbol{x})}
\sum_{i=0}^{L-1}\left[f(\boldsymbol{x}+\boldsymbol{z}_{i+1}^{N})-
f(\boldsymbol{x}+\boldsymbol{z}_{i}^{N})\right]^{2}\;.
\end{equation}

The next result states that the generator $\mathcal{L}_{N}$ satisfies a
sector condition. This means that the eigenvalues of the operator
$\mathcal{L}_{N}$ complexified belong to the sector $\{z\in \bb C :
|\text{Im } z| \le 2L \text{Re } z \}$  (cf. \cite[Proposition 2.13]{R})

\begin{lemma}
\label{lem: sector condition}
For every function $f,\,g :\Xi_{N} \to \bb R$,
\begin{equation*}
\left\langle f,-\mathcal{L}_{N}g\right\rangle _{\mu_{N}}^{2} \;\le\; 
4\, L^{2}\, \mathcal{D}_{N}(f)\, \mathcal{D}_{N}(g)\;.
\end{equation*}
\end{lemma}

\begin{proof}
We first prove the sector condition for each
$\mathcal{L}_{N,\boldsymbol{x}}$.  Fix $\bs{x}\in\Xi_{N}$. By
definition, 
\begin{equation*}
\langle f, \mathcal{L}_{N,\boldsymbol{x}}g\rangle _{\mu_{N}}
\;=\;
\frac{1}{Z_N} \, e^{-N\overline{F}_{N}(\boldsymbol{x})} 
\sum_{i=0}^{L-1}f(\boldsymbol{x}+\boldsymbol{z}_{i}^{N})
\left[g(\boldsymbol{x}+\boldsymbol{z}_{i+1}^{N})-
g(\boldsymbol{x}+\boldsymbol{z}_{i}^{N})\right]\;.
\end{equation*}
We may rewrite the previous sum as
\begin{equation*}
\frac{1}{Z_N} \, e^{-N\overline{F}_{N}(\boldsymbol{x})}
\sum_{i=0}^{L-1} \Big[ f(\boldsymbol{x}+\boldsymbol{z}_{i}^{N})
-\frac{1}{L}\sum_{j=1}^{L-1}f(\boldsymbol{x}+\boldsymbol{z}_{j}^{N})\Big]
\left[g(\boldsymbol{x}+\boldsymbol{z}_{i+1}^{N})-
g(\boldsymbol{x}+\boldsymbol{z}_{i}^{N})\right]\;.
\end{equation*}
By the Cauchy-Schwarz inequality and the discrete Poincar\'{e}
inequality we obtain that
\begin{equation*}
\left\langle f,-\mathcal{L}_{N,\boldsymbol{x}}\,g\right\rangle
_{\mu_{N}}^{2} 
\;\le\; 4\, L^{2}\,\mathcal{D}_{N,\boldsymbol{x}}(f)
\, \mathcal{D}_{N,\boldsymbol{x}}(g)\;.
\end{equation*}
The statement of the lemma follows from this estimate and from Schwarz
inequality. 
\end{proof}

\smallskip\noindent{\bf Flows.} Fix a point $\boldsymbol{z}$ in
$\widehat{\Xi}_{N}$. Denote by
$c_{\boldsymbol{z}}(\boldsymbol{z}+\boldsymbol{z}_{i}^{N},
\boldsymbol{z}+\boldsymbol{z}_{i+1}^{N})$, $0\le i<L$, the conductance
of the edge $(\boldsymbol{z}+\boldsymbol{z}_{i}^{N},
\boldsymbol{z}+\boldsymbol{z}_{i+1}^{N})$ induced by the cycle
dynamics $\mc{L}_{N,\bs{z}}$ on $\gamma_{\bs{z}}^{N}$:
\begin{equation}
\label{eq:3}
 c_{\boldsymbol{z}}(\boldsymbol{z}+\boldsymbol{z}_{i}^{N},
\boldsymbol{z}+\boldsymbol{z}_{i+1}^{N})\;=\;\mu_{N}
(\boldsymbol{z}+\boldsymbol{z}_{i}^{N})\,
R_{i}^{N}(\boldsymbol{z})\;=\;Z_{N}^{-1}\,
e^{-N\overline{F}_{N}(\boldsymbol{z})}\;.  
\end{equation}
Note that the conductance is constant over the cycle
$\gamma_{\bs{z}}^N$. We extend the definition of the conductance to
the other edges by setting it to be $0$:
$c_{\boldsymbol{z}}(\boldsymbol{x},\boldsymbol{y})=0$ if
$(\boldsymbol{x},\boldsymbol{y})\not
=(\boldsymbol{z}+\boldsymbol{z}_{i}^{N},
\boldsymbol{z}+\boldsymbol{z}_{i+1}^{N})$ for some $0\le
i<L$. Finally, the conductance $c(\boldsymbol{x},\boldsymbol{y})$,
$\boldsymbol{x}$, $\boldsymbol{y}\in\Xi_{N}$, is defined by
\begin{equation}
\label{capacity}
c(\boldsymbol{x},\boldsymbol{y})\;=\;
\sum_{\boldsymbol{z}\in\widehat{\Xi}_{N}}
c_{\boldsymbol{z}}(\boldsymbol{x},\boldsymbol{y})\;.
\end{equation}
The symmetric conductance $c^{s}(\boldsymbol{x},\boldsymbol{y})$ is
defined by
\begin{equation*}
c^{s}(\boldsymbol{x},\boldsymbol{y}) \;=\; 
\frac{1}{2}\,\{c(\boldsymbol{x},\boldsymbol{y})
\;+\; c(\boldsymbol{y},\boldsymbol{x})\}\;.
\end{equation*}
Note here that $c(\bs x , \bs y)=0$ if $\bs y - \bs x\not\in \{\bs
z^N_{j+1} -\bs z^N_j : 0\le j < L\}$.

Let $E_{N}$ be the set of oriented edges defined by
\begin{equation}
\label{00E}
E_{N} \;=\; \{(\boldsymbol{x},\boldsymbol{y})\in\Xi_{N}\times\Xi_{N}:
\,c_{s}(\boldsymbol{x},\boldsymbol{y})>0\}\;.
\end{equation}
Clearly, $E_{N}$ is the collection of all oriented edges of the cycles
$\gamma_{\boldsymbol{x}}^{N}$, $\boldsymbol{x}\in\widehat{\Xi}_{N}$.  An
anti-symmetric function $\phi:E_{N}\to\mathbb{R}$ is called a
$\textit{flow}$. The $\textit{divergence}$ of a flow $\phi$ at
$\boldsymbol{x}\in\Xi_{N}$ is defined as
\begin{equation}
\label{08}
(\mbox{div}\,\phi)(\boldsymbol{x}) \;=\;
\sum_{\boldsymbol{y}:(\boldsymbol{x},\boldsymbol{y})\in E_{N}}
\phi (\boldsymbol{x},\boldsymbol{y})\;,
\end{equation}
while its divergence on a set $\mc{A}\subset\Xi_N$ is given by
\begin{equation*}
(\text{div }\phi)(\mc{A}) \;=\; \sum_{\boldsymbol{x}\in \mc{A}}
(\text{div}\,\phi)(\boldsymbol{x})\;. 
\end{equation*}
The flow $\phi$ is said to be \textit{divergence-free at}
$\boldsymbol{x}$ if $(\mbox{div}\,\phi)(\boldsymbol{x})=0$.

Denote by $\mathcal{F}_{N}$ the set of flows endowed with the scalar
product given by
\begin{equation*}
\left\langle \phi,\psi\right\rangle _{\mathcal{F}_{N}} \;=\; 
\frac{1}{2}\sum_{(\boldsymbol{x},\boldsymbol{y})\in E_{N}}
\frac{\phi(\boldsymbol{x},\boldsymbol{y})\,
\psi(\boldsymbol{x},\boldsymbol{y})}
{c^{s}(\boldsymbol{x},\boldsymbol{y})}\;,\quad
\text{and let} \quad \left\Vert \phi\right\Vert _{\mathcal{F}_{N}}^{2} 
\;=\; \left\langle \phi,\phi\right\rangle _{\mathcal{F}_{N}}\;. 
\end{equation*}
From now on, we omit $\mc{F}_N$ from the notation above and we write
$\langle \phi,\psi\rangle$, $\Vert \phi\Vert$ for $\langle
\phi,\psi\rangle _{\mathcal{F}_{N}}$, $\Vert \phi\Vert
_{\mathcal{F}_{N}}$, respectively.

\smallskip\noindent{\bf Dirichlet and Thomson Principles.}
For a function $f:\Xi_{N}\rightarrow\mathbb{R}$, define the flows
$\Phi_{f}$, $\Phi_{f}^{*}$ and $\Psi_{f}$ by
\begin{equation}
\label{21} 
\begin{aligned}
& \Phi_{f}(\boldsymbol{x},\boldsymbol{y}) \;=\; 
f(\boldsymbol{x})\, c(\boldsymbol{x},\boldsymbol{y}) \,-\,
f(\boldsymbol{y})\, c(\boldsymbol{y},\boldsymbol{x})\;, \\
& \quad \Phi_{f}^{*}(\boldsymbol{x},\boldsymbol{y}) \;=\; 
f(\boldsymbol{x})\, c(\boldsymbol{y},\boldsymbol{x}) \,-\,
f(\boldsymbol{y})\, c(\boldsymbol{x},\boldsymbol{y})\;,\\
& \qquad \Psi_{f}(\boldsymbol{x},\boldsymbol{y}) \;=\; 
c^{s}(\boldsymbol{x},\boldsymbol{y})\, (f(\boldsymbol{x})-f(\boldsymbol{y}))\;.  
\end{aligned}
\end{equation}
It follows from the definition of these flows that for all functions
$f:\Xi_{N}\to\bb R$, $g:\Xi_{N}\to\bb R$,
\begin{equation}
\label{22}
\langle \Psi_{f} , \Phi_g \rangle  \;=\;
\< (- \mc L_N) f , g \>_{\mu_N} \;, \qquad 
\langle \Psi_{f} , \Phi^*_g \rangle \;=\;
\< (- \mc L^*_N) f , g \>_{\mu_N}\;. 
\end{equation}

Fix two disjoint subsets $\mc{A},\,\mc{B}$ of $\Xi_{N}$ and two real numbers
$a$, $b$. Denote by $\mathfrak{C}_{a,b}(\mc{A},\mc{B})$ the set of functions
which are equal to $a$ on $\mc{A}$ and $b$ on $\mc{B}$:
\begin{equation}
\label{eq:C_ab}
\mathfrak{C}_{a,b}(\mc{A},\mc{B}) \;=\; \left\{ f:\Xi_{N}\rightarrow\mathbb{R}:
f|_{\mc{A}}\equiv a,\,f|_{\mc{B}}\equiv b\right\} \;,
\end{equation}
and let $\mathfrak{U}_{a}(\mc{A},\mc{B})$ be the set of flows from $\mc{A}$ to $\mc{B}$ with strength $a$:
\begin{equation}
\label{eq:U_ab}
\begin{aligned}
\mathfrak{U}_{a}(\mc{A},\mc{B})\;=\;
\bigl\{\phi\in\mathcal{F}_{N}:(\mbox{div }\phi)(\mc{A})
=a & =-(\mbox{div }\phi)(\mc{B})\,,\\
 & (\mbox{div}\,\phi)(\boldsymbol{x})=0\,,\,
\boldsymbol{x}\in(\mc{A}\cup \mc{B})^{c}\bigl\}\;.
\end{aligned}
\end{equation}
In particular, $\mathfrak{U}_{1}(\mc{A},\mc{B})$ is the set of unitary
flows from $\mc{A}$ to $\mc{B}$.

Recall from \eqref{32} that we represent by $H_\mc{A}$ the hitting
time of a subset $\mc{A}$ of $\Xi_N$.  Denote by
$V_{\mc{A},\mc{B}}:\Xi_{N}\rightarrow[0,\,1]$ the \textit{equilibrium
  potential} between two disjoint subsets $\mc{A}$, $\mc{B}$ of
$\Xi_N$:
\begin{equation*}
V_{\mc{A},\mc{B}}(\boldsymbol{x}) \;=\; 
\mathbb{P}_{\boldsymbol{x}}^{N}[H_{\mc{A}}<H_{\mc{B}}]\;.
\end{equation*}
Let $V_{\mc{A},\mc{B}}^{*}$ be the equilibrium potential corresponding
to the adjoint dynamics. The proof of next theorem can be found in
\cite{GL}.

\begin{theorem}[Dirichlet principle]
\label{s07}
For any disjoint and non-empty subsets $\mc{A}$, $\mc{B}$ of $\Xi_{N}$,
\begin{equation*}
\mbox{\textup{cap}}_{N}(\mc{A},\mc{B}) \;=\; 
\inf_{f\in\mathfrak{C}_{1,0}(\mc{A},\mc{B})}\,
\inf_{\phi\in\mathfrak{U}_{0}(\mc{A},\mc{B})}
\left\Vert \Phi_{f}-\phi\right\Vert ^{2}\;.
\end{equation*}
Furthermore, the unique optimizers of the variational problem 
are given by 
\begin{equation*}
f_{0}\,=\,\frac{1}{2}(V_{\mc{A},\mc{B}}+V_{\mc{A},\mc{B}}^{*})\;\;\mbox{ and }\;\;
\phi_{0}\,=\,\frac{1}{2}(\Phi_{V_{\mc{A},\mc{B}}^{*}}-\Phi_{V_{\mc{A},\mc{B}}}^{*})\;.
\end{equation*}
\end{theorem}

Next theorem is due to Slowik \cite{Slo}.

\begin{theorem}[Thomson principle]
\label{s10}
For any disjoint and non-empty subsets $\mc{A}$, $\mc{B}$ of $\Xi_{N}$,  
\begin{equation*}
\mbox{\textup{cap}}_{N}(\mc{A},\mc{B}) \;=\; 
\sup_{g\in\mathfrak{C}_{0,0}(\mc{A},\mc{B})}\,
\sup_{\psi\in\mathfrak{U}_{1}(\mc{A},\mc{B})}
\frac{1}{\left\Vert \Phi_{g}-\psi\right\Vert ^{2}}\;\cdot
\end{equation*}
Furthermore, the unique optimizers of the variational problem are
given by
\begin{equation*}
g_{0}\,=\,\frac{V_{\mc{A},\mc{B}}^{*}-V_{\mc{A},\mc{B}}}{2\,\mbox{\textup{cap}}_{N}(\mc{A},\mc{B})}
\;\;\mbox{ and }\;\;\psi_{0}\,=\,\frac{\Phi_{V_{\mc{A},\mc{B}}^{*}}+
\Phi_{V_{\mc{A},\mc{B}}}^{*}}{2\,\mbox{\textup{cap}}_{N}(\mc{A},\mc{B})}\;\cdot
\end{equation*}
\end{theorem}

\smallskip\noindent{\bf Comments on the Proof.}  In view of Theorems
\ref{s07} and \ref{s10}, the proof of Theorem \ref{s133} consists in
finding functions $f_N$, $g_N$ and flows $\phi_N$, $\psi_N$ satisfying
the constraints of the variational problems of Theorems \ref{s07} and
\ref{s10} with $\mc{A}=\mc{E}_N (A),\,\mc{B}=\mc{E}_N (B)$, and such that
\begin{align*}
& \frac{Z_{N}}{(2\pi N)^{\frac{d}{2}-1}}e^{NH}
\left\Vert \Phi_{f_{N}}-\phi_{N}\right\Vert ^{2} \;\le\;
\big[1+o_{N}(1) \big] \textup{cap}_Y(A,B)\;, \\
&\quad \frac{Z_{N}}{(2\pi N)^{\frac{d}{2}-1}}\, e^{NH}
\frac{1}{\left\Vert \Phi_{g_{N}}-\psi_{N}\right\Vert ^{2}}
\;\ge\; \big[1+o_{N}(1) \big] \textup{cap}_Y(A,B)\;.
\end{align*}
The crucial point of the argument is the definition of these functions
and flows close to the saddle points where the equilibrium potential
between two wells exhibits a non-trivial behavior, changing abruptly
from $0$ to $1$.

The main difficulty of the proof of Theorem \ref{mr1} consist in
computing the mean jump rates. While in the reversible case, the mean
jump rates are expressed in terms of capacities, in the non-reversible
setting they appear as the value on a metastable set of the
equilibrium potential between two other metastable sets. To estimate
this value is delicate because, in contrast with capacities, there is
no variational formula for the value at one point of an equilibrium
potential. 

Theorem \ref{ek} is a straightforward consequence of Theorem
\ref{s133} and of the fact that the equilibrium potential is close to
a constant on each well.

\smallskip\noindent{\bf Summary of the article.} In Section
\ref{sec3}, we construct a function in a mesoscopic neighborhood of a
saddle point between two local minima of the potential $F$ which
approximates in this neighborhood the equilibrium potential between
the local minima. Denote this function by $V_N$ and the local minima
by $\bs m_1$, $\bs m_2$. In Section \ref{sec:flow}, we present a flow,
denoted here by $\Phi_N$, which approximates the flow $\Phi^*_{V_N}$
and which is divergence-free on $\{\bs m_1, \bs m_2\}^c$. The
functions $V_N$ and the flows $\Phi_N$ are the building blocks with
which we produce, in the next sections, the approximating functions
and flows described above in the summary of the proof. In Section
\ref{sec6}, we use these functions and flows to prove Theorem
\ref{s133} in the case where $A$ is a singleton and $B=S\setminus A$.
In Section \ref{sec7}, we extend the analysis of the previous section
to the general case and prove Theorem \ref{s133}. One of the steps
 consists in determining the value of the equilibrium potential
between $\mc E_N(A)$ and $\mc E_N(B)$ in the other wells $\mc E^i_N$,
$i\not\in A\cup B$. In Section \ref{sec8}, we compute the asymptotic
mean jump rates which describe the scaling limit of the random walk
$X_N(t)$. As we stressed above, this analysis requires the estimation
of the value at $\mc E^i_N$ of the equilibrium potential between $\mc
E_N(A)$ and $\mc E_N(B)$. We present in that section a general
strategy to obtain a sharp estimate for this quantity. In Section
\ref{sec9}, we prove the metastable behavior of $X_N(t)$ in all
time-scales $\beta^{(m)}_N$ by showing that all conditions required in
the main result of \cite{BL2,Lan1} are fulfilled. Finally, in the
appendix, we present a generalization of Sylvester's law of inertia.

\section{The equilibrium potential around saddle points} 
\label{sec3}

In this section, we introduce a function $V_N$ which approximates the
equilibrium potential $V_{\mathcal{E}_N^i,\mathcal{E}_N^j}$ around a
saddle point $\bs \sigma \in \mf S_{i,j}$. To fix ideas, let
$(i,j)=(1,2)$ and assume that $\boldsymbol{\sigma}$ is the origin
$\boldsymbol{0}$, so that $H=F(\bs 0)$ is the height of the saddle
point $\bs 0$.  Throughout the remaining part of this article, $C$ and
$C_0$ represent constants which do not depend on $N$ and whose value
may change from line to line.

\subsection{The geometry around the saddle point}
\label{sub:geometry}

Recall from \eqref{04} the definition of the matrix $b$ and from
\eqref{45b} that we represent by $\bb A^\dagger$, $\bs w^\dagger$ the
transpose of the matrix $\bb A$, vector $\bs w$, respectively.  The
Jacobian of $b$ at the origin, denoted by $\bb M$, is given by
\begin{equation}
\label{36} 
\mathbb{M}=\mathbb{A}\mathbb{H} \quad\text{where}\quad 
\mathbb{A} \;=\;
\sum_{i=0}^{L-1}(\boldsymbol{z}_{i}-\boldsymbol{z}_{i+1})\, 
\boldsymbol{z}_{i}^{\dagger}\quad \mbox{and}\quad
\mathbb{H}=(\mbox{Hess}\, F) (\boldsymbol{0})\,.
\end{equation}
Remark that 
\begin{equation*}
\frac{1}{2}(\mathbb{A}+\mathbb{A}^{\dagger})\;=\; 
\frac 12 \sum_{j=0}^{L-1}(\boldsymbol{z}_{i}-\boldsymbol{z}_{i+1})
\, (\boldsymbol{z}_{i}-\boldsymbol{z}_{i+1})^{\dagger}
\end{equation*}
is a positive-definite matrix because, by assumption, the vectors
$\bs{z}_{i+1}-\bs{z}_i$ generate $\bb Z^d$. It follows from this last
observation, from the assumption that $\bb H$ has only one negative
eigenvalue and from Lemma \ref{lem: property of M} that $\mathbb{M}$
has only one negative eigenvalue. Denote this eigenvalue by $-\mu$ and
denote by $\boldsymbol{v}$ the eigenvector of $\bb M^\dagger = \bb H
\bb A^\dagger$ associated to the eigenvalue $-\mu$.

Let
\begin{equation}
\label{06}
\alpha \;=\; \frac{\mu}{\boldsymbol{v}^{\dagger}\mathbb{A} \boldsymbol{v}}  
\end{equation}
and note that $\alpha>0$ because $\mathbb{A}$ is a positive-definite
matrix. Moreover,
\begin{equation}
\label{lem:pre_vH^-1v}
\boldsymbol{v}^{\dagger}\mathbb{H}^{-1}\boldsymbol{v} \;=\; - \frac
1\alpha
\end{equation}
because
\begin{equation*}
\boldsymbol{v}^{\dagger}\mathbb{H}^{-1}\boldsymbol{v} \;=\; 
\boldsymbol{v}^{\dagger}\mathbb{A}^{\dagger} 
(\mathbb{H}\mathbb{A}^{\dagger})^{-1}
\boldsymbol{v}\,=\,\boldsymbol{v}^{\dagger}\mathbb{A}^{\dagger}
\left(-\frac{1}{\mu}\boldsymbol{v}\right)\,=\,-\frac{1}{\alpha}\;\cdot 
\end{equation*}

Denote by $-\lambda_{1}, \lambda_{2}, \dots, \lambda_{d}$ the
eigenvalues of $\mathbb{H}$, and by $\boldsymbol{u}_{1},
\boldsymbol{u}_{2}, \dots, \boldsymbol{u}_{d}$ the corresponding
eigenvectors, where $-\lambda_{1}$ is the unique negative eigenvalue
of $\mathbb{H}$.  Denote by $v_1, \dots, v_d$ the coordinates of $\bs
v$ in the basis $\bs u_1, \dots, \bs u_d$:
\begin{equation}
\label{eq: rel u v}
\boldsymbol{v} \;=\; \sum_{i=1}^{d}v_{i}\, \boldsymbol{u}_{i}\;.
\end{equation}
With this notation, \eqref{lem:pre_vH^-1v} can be rewritten as
\begin{equation}
\label{02}
\frac{v_{1}^{2}}{\lambda_{1}} \;=\;
\sum_{k=2}^{d}\frac{v_{k}^{2}}{\lambda_{k}}
\;+\; \frac{1}{\alpha}\;\cdot
\end{equation}
In particular, as $\alpha>0$, $v_1\not =0$. This proves that the
vectors $\bs v$ and $\bs u_1$ are not orthogonal. Assume, without loss
of generality, that $v_{1} = \bs v \cdot \bs u_1 := \bs v^\dagger \bs
u_1 >0$.

\begin{lemma}
\label{s17}
The $d\times d$ matrix
$\mathbb{H}+\alpha\boldsymbol{v}\boldsymbol{v}^{\dagger}$ is
non-negative definite and
$\det(\mathbb{H}+\alpha\boldsymbol{v}\boldsymbol{v}^{\dagger})=0$.
The matrix $\mathbb{H}+2\alpha\boldsymbol{v}\boldsymbol{v}^{\dagger}$
is positive definite and
$\det(\mathbb{H}+2\alpha\boldsymbol{v}\boldsymbol{v}^{\dagger})=-\det(\mathbb{H})$.
\end{lemma}

\begin{proof}
We first prove that
$\mathbb{H}+\alpha\boldsymbol{v}\boldsymbol{v}^{\dagger}$ is
non-negative definite:
\begin{equation*}
\boldsymbol{x}^{\dagger}\mathbb{H}\boldsymbol{x}
+\alpha(\boldsymbol{x}\cdot\boldsymbol{v})^{2} \;
\ge \; 0\;, \quad \forall\, \boldsymbol{x}\in\mathbb{R}^{d}\;.
\end{equation*}

We consider two cases. Suppose first that $\sum_{2\le k\le
  d}v^2_{k}=0$. Under this hypothesis, by \eqref{02}, $v^2_1 =
\lambda_1/\alpha$ and, writing $\boldsymbol{x}$ as
$\sum_{i=1}^{d}x_{i}\boldsymbol{u}_{i}$, the previous sum becomes
\begin{equation*}
-\lambda_{1}x_{1}^{2} \;+\;
\sum_{k=2}^{d}\lambda_{k}x_{k}^{2} \;+\;
\lambda_1 x^2 _{1} \;\ge\; 0 \; .
\end{equation*}

Suppose, on the other hand, that $\sum_{2\le k\le d}v^2_{k}>0$.
Writing again $\boldsymbol{x}$ as
$\sum_{i=1}^{d}x_{i}\boldsymbol{u}_{i}$, we obtain that
\begin{equation*}
\boldsymbol{x}^{\dagger}\mathbb{H}\boldsymbol{x}
+\alpha(\boldsymbol{x}\cdot\boldsymbol{v})^{2} \;=\;
-\lambda_{1}x_{1}^{2} \;+\;
\sum_{k=2}^{d}\lambda_{k}x_{k}^{2} \;+\;
\alpha\Big(\sum_{i=1}^{d}x_{i}v_{i}\Big)^{2} \; .
\end{equation*}
By \eqref{02}, this expression is convex in $x_1$.
By optimizing this sum over $x_{1}$ and by using \eqref{02}, we show
that it is bounded below by
\begin{equation*}
\sum_{k=2}^{d}\lambda_{k}x_{k}^{2}\;-\;
\frac{\left(\sum_{k=2}^{d}x_{k}v_{k}\right)^{2}}
{\sum_{k=2}^{d} (v_{k}^{2}/\lambda_{k})}\;\cdot
\end{equation*}
The denominator is well-defined by assumption.  By the Cauchy-Schwarz
inequality, this difference is non-negative, which proves the first
assertion of the lemma.

We turn to the determinants. Recall the well-known formula:
\begin{equation}
\label{eq: rank one upd}
\det(\mathbb{B}+\boldsymbol{u}\boldsymbol{w}^{\dagger})
\;=\; (1+\boldsymbol{w}^{\dagger}\mathbb{B}^{-1}\boldsymbol{u})
\, \det\mathbb{B}\;,
\end{equation}
where $\mathbb{B}$ is any $d\times d$ non-singular matrix and
$\boldsymbol{u},\,\boldsymbol{w}$ are any $d$-dimensional
vectors. It follows from this identity that for any $r\in\mathbb{R}$ 
\begin{equation*}
\det(\mathbb{H}+r\alpha\boldsymbol{v}\boldsymbol{v}^{\dagger}) \;=\; 
(1+r\alpha\boldsymbol{v}^{\dagger}\mathbb{H}^{-1}\boldsymbol{v})\,
\det\mathbb{H}\,=\,(1-r)\, \det\mathbb{H}\;,
\end{equation*}
where the last equality is due to \eqref{lem:pre_vH^-1v}. This proves
that
$\det(\mathbb{H}+\alpha\boldsymbol{v}\boldsymbol{v}^{\dagger})=0$,
$\det(\mathbb{H}+2\alpha\boldsymbol{v}\boldsymbol{v}^{\dagger})=-\det\mathbb{H}$.

Finally, the positive definiteness of
$\mathbb{H}+2\alpha\boldsymbol{v}\boldsymbol{v}^{\dagger}$ follows by
the non-negative definiteness of $\mathbb{H}+\alpha \bs v\bs
v^{\dagger}$, $\alpha \bs v\bs v^{\dagger}$ and from the fact that
$\det(\mathbb{H} + 2\alpha\boldsymbol{v} \boldsymbol{v}^{\dagger}) =
-\det\mathbb{H}>0$.
\end{proof}

For vectors $\bs w_1, \dots, \bs w_k$ in $\bb R^d$, denote by $\<\bs
w_1, \dots, \bs w_k\>$ the linear space generated by these vectors.

\begin{lemma}
\label{s21}
The null space of $\mathbb{H}+ \alpha\boldsymbol{v}
\boldsymbol{v}^{\dagger}$ is one-dimensional and is given by
$\left\langle \mathbb{H}^{-1}\boldsymbol{v}\right\rangle $.
\end{lemma}

\begin{proof}
Suppose that $\boldsymbol{w}\in\mathbb{R}^{d}$ satisfies
$(\mathbb{H}+\alpha\boldsymbol{v}\boldsymbol{v}^{\dagger})\boldsymbol{w}=0$.
Since $\bb H$ is invertible, this equation can be rewritten as $\boldsymbol{w} =
- \alpha (\boldsymbol{v}\cdot\boldsymbol{w})\mathbb{H}^{-1}\boldsymbol{v}$ and
therefore $\boldsymbol{w}\in\left\langle
  \mathbb{H}^{-1}\boldsymbol{v}\right\rangle $.  

On the other hand, it follows from \eqref{lem:pre_vH^-1v} that any
vector $\boldsymbol{w}=a\mathbb{H}^{-1}\boldsymbol{v}$,
$a\in\mathbb{R}$, satisfies $(\mathbb{H}+ \alpha\boldsymbol{v}
\boldsymbol{v}^{\dagger}) \boldsymbol{w}=0$.
\end{proof}

\subsection{The neighborhood of a saddle point.} 
Consider the $(d-1)$-dimensional hyperplane $\left\langle
  \boldsymbol{u}_{1}\right\rangle ^{\perp}=\left\langle
  \boldsymbol{u}_{2}, \boldsymbol{u}_{3}, \dots,
  \boldsymbol{u}_{d}\right\rangle $.  Clearly, $\mathcal{E}_{N}^{1}$
and $\mathcal{E}_{N}^{2}$ are on different sides of $\left\langle
  \boldsymbol{u}_{1}\right\rangle ^{\perp}$. Since $\<\bs v , \bs
u_1\> >0$, without loss of generality, we may assume that
$\boldsymbol{u}_{1}$ and $\boldsymbol{v}$ are directed toward
$\mathcal{E}_{N}^{1}$, with respect to $\left\langle
  \boldsymbol{u}_{1}\right\rangle ^{\perp}$.  See Figure \ref{fig:41}.

\begin{figure}
\centering
\includegraphics[scale=0.24]{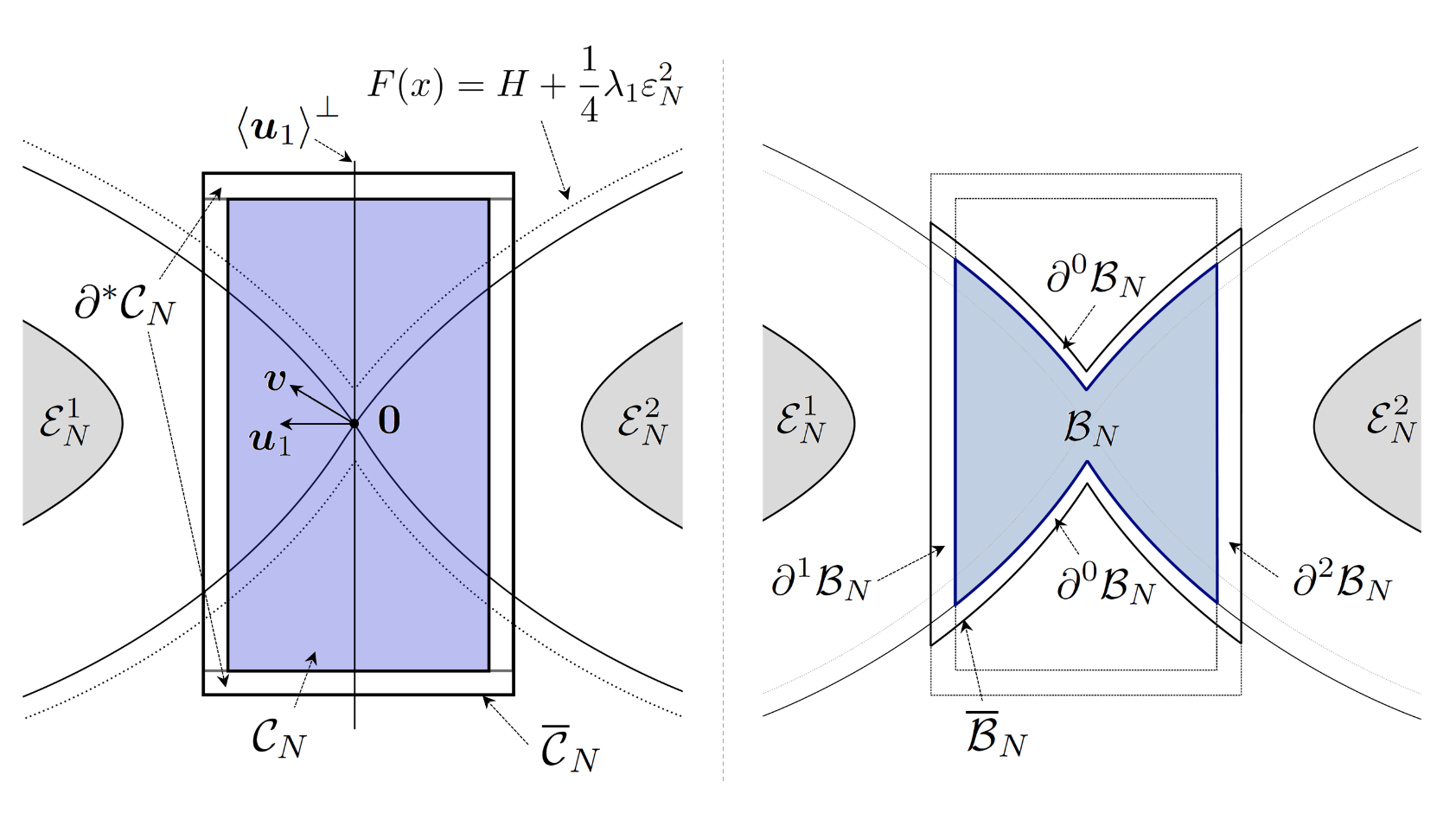}\protect
\caption{\label{fig:41} The neighborhoods $\mc C_N$ and $\mc B_N$ of a saddle point}
\end{figure}

For a subset $\mc G_N$ of $\Xi_N$, denote by $\mathring{\mc G}_{N}$,
$\overline{\mc G}_{N}$ and $\partial \mc G_{N}$ the core, the
closure, and the boundary of $\mc G_{N}$, respectively,
\begin{equation*}
\mathring{\mc G}_{N} \;=\; \{\boldsymbol{x}\in\widehat{\Xi}_{N}:
\,\gamma_{\boldsymbol{x}}^{N}\cap \mc G_{N}\neq \varnothing\}\;,\quad
\overline{\mc G}_{N} \;=\; \bigcup_{\boldsymbol{x}\in\mathring{\mc G}_{N}}
\gamma_{\boldsymbol{x}}^{N}\;,\quad
\partial \mc G_{N} \;=\; \overline{\mc G}_{N}\setminus \mc G_{N}\;.
\end{equation*}
One can easily check that $\mc G_N \subset \mathring{\mc G}_N \subset
\overline{\mc G}_N$, and that for two subsets $\mc G_N$, $\mc G_N'$ of
$\Xi_N$, $\partial (\mc G_N \cap \mc G_N') \subset \partial \mc G_N
\cup \partial\mc G_N'$.

In order to define the approximation of the equilibrium potential and
the related flow, as in the reversible case \cite{LMT}, we introduce a
mesoscopic neighborhood of the saddle point. 

Let
$\varepsilon_{N}\downarrow0$ be a sequence such that
\begin{align}
 & N\varepsilon_{N}^{2}\uparrow\infty\,\,
\mbox{and \,\ensuremath{N\varepsilon_{N}^{3}\downarrow0}\,\,as}
\,\,N\uparrow\infty\;,\label{vare1}\\
 & \forall a>0,\, \exp\{-aN\varepsilon_{N}^{2}\}
\downarrow 0\mbox{ as }N\uparrow \infty
\mbox{ faster than any polynomials of \ensuremath{N}.}\label{vare2}
\end{align}
One can take, for instance, $\varepsilon_{N}=N^{-\frac{2}{5}}$. 

Denote by $C_N$, $\mc C_N$, the mesoscopic neighborhoods of the origin
given by
\begin{equation}
\label{C_N}
C_{N} \;=\; \Big\{ \sum_{i=1}^{d} x_{i} \boldsymbol{u}_{i} : \,|x_{1}|\le
\varepsilon_{N} \,,\, |x_{i}|\le\sqrt{\frac{ 2\lambda_{1}}
{\lambda_{i}}} \, \varepsilon_{N} \,,\, 2\le i\le d \Big\} 
\;, \;\; \mc C_{N}\;=\; C_{N}\cap\Xi_{N}\;.
\end{equation}
Let $\partial^{*} \mc C_{N}$ be the piece of the boundary $\partial
\,\mc C_{N}$ given by
\begin{equation*}
\partial^{*}\mc C_{N} \;=\; \bigcup_{i=2}^{d} \Big\{
\boldsymbol{x}\in\partial \, \mc C_{N}:
\,|\boldsymbol{x}\cdot\boldsymbol{u}_{i}|>
\sqrt{\frac{2 \lambda_{1}}{\lambda_{i}}}\varepsilon_{N}\Big\} \;.
\end{equation*}
We claim that 
\begin{equation}
\label{26}
\inf_{\boldsymbol{x}\in\partial^{*} \mc C_{N}}F(\boldsymbol{x})
\;\ge\; H+ \big[ 1 + o_N(1) \big]\,
\frac{1}{2}\lambda_{1}\varepsilon_{N}^{2}\;.
\end{equation}
Indeed, it follows from the definition of $\mc C_N$ that
$|\boldsymbol{x}\cdot\boldsymbol{u}_{1}| \le
\varepsilon_{N}+O(N^{-1})$ for $\boldsymbol{x}\in\partial^{*} \mc
C_{N}$. Since $F(\bs{0})=H$, by the Taylor expansion of $F$ at
$\bs{0}$,
\begin{align*}
F(\bs{x})-H \;& =\; \frac{1}{2}\,
\boldsymbol{x}^{\dagger}\mathbb{H}\boldsymbol{x}
\;+\; o(\varepsilon_{N}^{2}) \\
& =\; - \frac{1}{2} \,
(\boldsymbol{x}\cdot\boldsymbol{u}_{1})^{2}\lambda_{1}
\;+\; \frac{1}{2} \,
\sum_{i=2}^{d}(\boldsymbol{x}\cdot\boldsymbol{u}_{i})^{2}
\lambda_{i} \;+\; o(\varepsilon_{N}^{2}) \\
& \ge\; \frac{1}{2} (-\lambda_{1}+ 2 \lambda_{1})
\varepsilon_{N}^{2} \;+\; o(\varepsilon_{N}^{2})\;,
\end{align*} 
as claimed. The previous Taylor expansion holds also for
$F_{N}(\boldsymbol{x})$ and for $\overline{F}_{N}(\boldsymbol{x})$
with exactly same form and these Taylor expansions will be frequently
used later.

Denote by $\mc B_N$ the discrete mesoscopic neighborhoods of the
origin given by
\begin{equation}
\label{25}
\mathcal{B}_{N} \;=\; \mc C_{N}\cap \Big\{ \boldsymbol{x}\in\Xi_N:\,
F(\boldsymbol{x})\le H+\frac{1}{4}\lambda_{1}\varepsilon_{N}^{2}\Big\} \;.
\end{equation}
Divide the boundary
$\partial\mathcal{B}_{N}$ in three pieces,
$\partial^{0}\mathcal{B}_{N},\,\partial^{1}\mathcal{B}_{N}$ and
$\partial^{2}\mathcal{B}_{N}$, as follows
\begin{align*}
& \partial^{0}\mathcal{B}_{N} \;=\; \big\{ 
\boldsymbol{x}\in\partial\mathcal{B}_{N}:
\,F(\boldsymbol{x})> H + (1/4)
\lambda_{1}\varepsilon_{N}^{2}\big\} \;,\\
&\quad
\partial^{1}\mathcal{B}_{N} \;=\; \big\{ \boldsymbol{x}
\in\partial\mathcal{B}_{N}\setminus\partial^{0}\mathcal{B}_{N}:
\,\boldsymbol{x}\cdot\boldsymbol{u}_{1} > \varepsilon_{N} \big\} \;,\\
&\qquad \partial^{2}\mathcal{B}_{N} \;=\; \big\{ 
\boldsymbol{x}\in\partial\mathcal{B}_{N}\setminus\partial^{0}
\mathcal{B}_{N}:\,\boldsymbol{x}\cdot\boldsymbol{u}_{1} < -\varepsilon_{N}
\big\} \;.
\end{align*}
Note that $\partial^{i}\mathcal{B}_{N}$, $i=1,\,2$, is the portion of
$\partial\mathcal{B}_{N}$ close to the metastable set
$\mathcal{E}_{N}^{i}$.  This decomposition is visualized in Figure
\ref{fig:41}. Note also that it follows from the definitions of
$\mathcal{B}_{N}$ and $\partial^{0}\mathcal{B}_{N}$, that 
\begin{equation}
\label{27}
F(\boldsymbol{x}) \;=\; H+ \frac{1}{4}\lambda_{1}\varepsilon_{N}^{2}
\;+\; O(N^{-1}) \quad\text{for}\quad
\boldsymbol{x}\in\partial^{0}\mathcal{B}_{N} \;.
\end{equation}
Of course, the same estimate is valid for $F_{N}$ or
$\overline{F}_{N}$.

We claim that $\partial\mathcal{B}_{N} = \cup_{0\le i\le
  2} \partial^i\mc B_{N}$. Indeed, one of the inclusions follows by
definition. To prove the other one, fix $\bs x\in \partial\mc B_N$ and
assume that $\bs x \not \in \partial^0\mc B_{N}$, i.e., that $F(\bs x)
\le H + (1/4) \lambda_{1} \varepsilon_{N}^{2}$. We have to show that
$\bs x \in \partial^1\mc B_{N} \cup \partial^2\mc B_{N}$. Clearly,
$\partial\mc B_N \subset \overline{\mc B}_N \subset \overline{\mc C}_N
= \mc C_N \cup \partial \mc C_N$ so that $\bs x\in \mc C_N
\cup \partial \mc C_N$. Assume that $\bs{x}\in \mc C_N$. Then, since
$\bs{x}\notin \mc{B}_N$, we have $\bs{x}\in \mc C_N \setminus
\mc{B}_N$ and thus $F(\bs{x})> (1/4)\lambda_1 \varepsilon_N^2$, by
definition of $\mc{B}_N$, which is a contradiction. This proves that
$\bs x \in \partial \mc C_N$.  It remains to recall the estimate
\eqref{26} to conclude that $|\<\bs x, \bs u_1\>| > \varepsilon_N$,
so that $\bs x\in \partial^1\mathcal{B}_{N}
\cup \partial^2\mathcal{B}_{N}$.

\begin{lemma}
\label{s25}
We have that 
\begin{equation*}
\sum_{\boldsymbol{x}\in \mathring{\mc B}_{N}}
e^{-\frac{N}{2}\boldsymbol{x}^{\dagger}(\mathbb{H}+2\alpha\boldsymbol{v}
\boldsymbol{v}^{\dagger})\boldsymbol{x}} 
\;=\; \big[1+o_{N}(1)\big] \, \frac{(2\pi
  N)^{\frac{d}{2}}}{\sqrt{-\det\mathbb{H}}}\;\cdot 
\end{equation*}
\end{lemma}

\begin{proof}
Let $\mathbb{K}$ be a positive-definite matrix, and let $\mc A_N$ be a
sequence of subsets of $\Xi_N$ such that
\begin{equation*}
\Xi_N \cap \Big\{ \sum_{i=1}^{d} x_{i}\bs u_{i} : |x_{i}|\le a_i 
\varepsilon_N\,,\,1\le i\le d \Big\} 
\;\subset\;
 \mc A_N
 \;\subset\; \mc C_N
\end{equation*}
for some $a_i>0$, $1\le i\le d$ and for all large enough $N$. It
follows from the proof of \cite[Assertion 3.B]{LMT} that
\begin{equation*}
\sum_{\boldsymbol{x}\in \mc A_{N}}e^{-\frac{N}{2}\boldsymbol{x}^{\dagger}
\mathbb{K}\boldsymbol{x}} \;=\; \big[1+o_{N}(1)\big] \, 
\frac{(2\pi N)^{\frac{d}{2}}}{\sqrt{\det\mathbb{K}}}\;.
\end{equation*}

The assertion of the lemma follows from this identity, from
Lemma \ref{s17}, and from the fact that
\begin{equation*}
\Xi_N \cap \Big\{ \sum_{i=1}^{d}a_{i}\boldsymbol{u}_{i} : 
|a_{i}|\le\sqrt{\frac{\lambda_1}{3(d-1)\lambda_i}} 
\varepsilon_N\,,\,1\le i\le d \Big\}  
\subset \;\mc B_N \subset \mathring{\mc B}_N\;, 
\end{equation*}
where the first inclusion can be easily proven by the Taylor expansion. 
\end{proof}

A similar estimation for $\mathbb{H}+ \alpha\boldsymbol{v}
\boldsymbol{v}^{\dagger}$ is needed. Since, by Lemma
\ref{s21}, the rank of the matrix $\mathbb{H}+\alpha
\boldsymbol{v} \boldsymbol{v}^{\dagger}$ is $d-1$, denote by $\ms
P_{t}$, $t\in\mathbb{R}$, the $(d-1)$-dimensional hyperplane given by
\begin{equation*}
\ms P_{t} \;=\; t\boldsymbol{u}_{1}+
\left\langle \boldsymbol{u}_{2}, \,\boldsymbol{u}_{3},\, \dots,\,
  \boldsymbol{u}_{d} \right\rangle \;.
\end{equation*}
For any $\boldsymbol{w}\in\mathbb{R}^{d}$, denote by $\ms
P_t(\boldsymbol{w})$ the region located between the hyperplanes $\ms
P_t$ and $\ms P_t+\boldsymbol{w}$:
\begin{equation}
\label{35}
\ms P_t(\boldsymbol{w}) \;=\; \{ \bs x\in\bb R^d : t\le \bs x \cdot
\bs u_1 < t+ (\bs w \cdot \bs u_1) \}\;,   
\end{equation}
provided $\bs w \cdot \bs u_1\ge 0$, with an analogous definition in
the case $\bs w \cdot \bs u_1 < 0$.  Note that $\ms
P_t(\boldsymbol{w})$ includes $\ms P_t$ but excludes $\ms
P_t+\boldsymbol{w}$.

Let $\mc A_N(a)$, $a>0$, be the set defined by
\begin{equation}
\label{f05}
\mc A_N (a) \;=\; \Big\{\bs x\in \Xi_N : \sum_{i=2}^d \lambda_j 
(\bs x \cdot \bs u_i)^2 \le (1+a) \lambda_1 \varepsilon^2_N \Big\}\;.
\end{equation}

\begin{lemma}
\label{s14}
Let $H_N(t,\bs{z},a)$, $t\in\mathbb{R}$, $\bs{z}\in\mathbb{R}^d$, $a>0$,
be given by
\begin{equation*}
H_N(t,\bs{z},a)\;=\;\frac{1}{(2\pi N)^{\frac{d-1}{2}}}
\sum_{\boldsymbol{x}\in \ms P_t(N^{-1} \boldsymbol{z})\cap {\mc A}_{N}(a)}
e^{-\frac{N}{2}\boldsymbol{x}^{\dagger}(\mathbb{H}+
\alpha\boldsymbol{v}\boldsymbol{v}^{\dagger})\boldsymbol{x}}\;. 
\end{equation*}
Then, for any $\bs{z}\in\mathbb{Z}^d$, $a>0$, $r>0$ and sequence
$\{t_N: N\ge 1\}$, such that $|t_N|\le\varepsilon_{N}+ r/N$,
\begin{equation}
\label{s141}
H_N(t_N,\bs{z}, a)\;=\; \big[1+o_{N}(1)\big] \,
\frac{|\boldsymbol{u}_{1}\cdot\boldsymbol{z}|}
{|\boldsymbol{u}_{1}\cdot\boldsymbol{v}|}
\sqrt{\frac{\lambda_{1}}{\alpha\prod_{k=2}^{d}\lambda_{k}}}\;. 
\end{equation}
\end{lemma}

\begin{proof}
Fix $a>0$ and $r>0$.  We may assume that $\bs{z} \cdot \bs u_1 \neq0$
since, if this is not the case, $H_N(t,\bs{z},a)=0$ for all
$t\in\mathbb{R}$ so that \eqref{s141} is trivial. We begin by proving
\eqref{s141} for $\boldsymbol{z} \in \{\pm \mf e_1, \dots, \pm \mf
e_d\}$, where $(\mf e_1, \dots, \mf e_d)$ stands for the canonical
basis of $\bb R^d$. It is enough to consider the case
$\boldsymbol{z}=\mathfrak{e}_{i}$, the proof for
$\boldsymbol{z}=-\mathfrak{e}_{i}$ being analogous. As explained
above, we may also assume that $\mf e_i \cdot \bs u_1 \neq0$.

For each $\boldsymbol{x}\in \ms P_{t_N}(N^{-1}\mathfrak{e}_{i})\cap
{\mc A}_{N}(a)$, there exists $\delta(\boldsymbol{x})\in(-1/N,\,1/N)$
such that $T(\boldsymbol{x})= \boldsymbol{x}+ \delta(\boldsymbol{x})
\mathfrak{e}_{i}\in \ms P_{t_N}$, so that
\begin{align*}
 & \frac{1}{(2\pi N)^{\frac{d-1}{2}}}\,
\sum_{\boldsymbol{x}\in\ms
  P_{t_N}(N^{-1}\mathfrak{e}_{i})\cap{\mc A}_{N}(a)}
e^{-\frac{N}{2}\boldsymbol{x}^{\dagger}(\mathbb{H}+
\alpha\boldsymbol{v}\boldsymbol{v}^{\dagger})\boldsymbol{x}}\\
 & =\;\frac{1+o_{N}(1)}{(2\pi N)^{\frac{d-1}{2}}}
\,\sum_{\boldsymbol{x}\in\ms
  P_{t_N}(N^{-1}\mathfrak{e}_{i})\cap{\mc A}_{N}(a)}
e^{-\frac{N}{2}T(\boldsymbol{x})^{\dagger}(\mathbb{H}+
\alpha\boldsymbol{v}\boldsymbol{v}^{\dagger})T(\boldsymbol{x})}\;.
\end{align*}
Replacing $\sqrt{N}T(\bs{x})$ by $\bs y$, and approximating the sum
appearing on the right hand side by a Riemann integral on the
hyperplane $\ms P_{t_N}$, the previous sum becomes
\begin{equation}
\label{07}
\big[1+o_{N}(1)\big]\,
\frac{|\boldsymbol{u}_{1}\cdot \mf e_i|}{(2\pi)^{\frac{d-1}{2}}}
\int_{\sqrt{N}(\ms P_{t_N}\cap A_{N})}
e^{-\frac{1}{2}\boldsymbol{y}^{\dagger}(\mathbb{H}+
\alpha\boldsymbol{v}\boldsymbol{v}^{\dagger})\boldsymbol{y}}
\, dS(\boldsymbol{y})\;,
\end{equation}
where $dS$ represents the $(d-1)$-dimensional surface integral. In
this formula, $A_N$ is the set given by
\begin{equation*}
A_N \;=\; \Big\{\bs x\in \Xi : \sum_{i=2}^d \lambda_j 
(\bs x \cdot \bs u_i)^2 \le (1+a) \lambda_1 \varepsilon^2_N \Big\}\;,
\end{equation*}
and $|\boldsymbol{u}_{1}\cdot \mf e_i|$ appeared to take into
account the tilt of the hyperplane $\ms P_{t_N}$. 

Fix $\bs x \in \sqrt{N}(\ms P_{t_N}\cap A_N)$ and write $\bs x=
\sqrt{N} \big( t_N \boldsymbol{u}_1 + \sum_{2\le j\le d}
x_{k}\boldsymbol{u}_{k} \big)$. It follows from \eqref{02}, that
\begin{align*}
&\boldsymbol{x}^{\dagger}(\mathbb{H} \;+\;
\alpha\boldsymbol{v}\boldsymbol{v}^{\dagger}) \boldsymbol{x} \\
&= N\bigg[ \; \sum_{k=2}^{d} \lambda_{k} \,x_{k}^{2} \;+\;
\alpha  t_N^2   \lambda_1 \sum_{k=2}^{d} \frac{v_{k}^{2}}{\lambda_{k} } \;+\;
2 \alpha t_N   v_1  \sum_{k=2}^{d} v_{k}\, x_{k}
\;+\; \alpha \Big( \sum_{k=2}^{d} v_{k}\, x_{k}\Big)^{2} \bigg]\;.
\end{align*}
Using again \eqref{02}, we show that the previous expression is equal
to
\begin{equation*}
N\bigg[ \sum_{k=2}^{d} \lambda_{k} \, \big( x_k + \theta_{k})^2 \;+\; \alpha\,
\Big\{\sum_{j=2}^d (x_j+ \theta_{j}) v_j \Big\}^2 \bigg] \;,
\end{equation*}
where $\theta_{k}= t_N (v_k \lambda_1)/(\lambda_k v_1)$. Hence, by the
change of variable, \eqref{07} can be written as
\begin{equation}
\label{601}
\big[1+o_{N}(1)\big]\,
\frac{|\boldsymbol{u}_{1}\cdot \mf e_i|}{(2\pi)^{\frac{d-1}{2}}}
\int_{\sqrt{N}A'_N}
e^{-\frac{1}{2}\boldsymbol{y}^{\dagger}(\mathbb{L}+
\alpha\boldsymbol{p}\boldsymbol{p}^{\dagger})\boldsymbol{y}}
\, dy_{2} \, dy_{3}\, \cdots \, dy_{d}\;,
\end{equation}
where $\mathbb{L}=\text{diag }(\lambda_{2},\dots,\lambda_{d})$,
$\boldsymbol{p}=(v_{2},\dots,v_{d})$ and
\begin{equation*}
A'_{N} \;=\; \Big\{ (y_{2}, \dots, y_{d})\in \mathbb{R}^{d-1}:
t_N \bs{u}_1 + \sum_{k=2}^d (y_k -\theta_k) \bs{u}_k \in A_N
\Big\} \;.
\end{equation*}

The point $\bs{w}_N = t_N \bs u_1 - \sum_{k=2}^d \theta_{k} \bs u_k$
belongs to the interior of $A_N$ because, by definition of $\theta_k$
and by \eqref{02},
\begin{equation*}
\sum_{j=2}^d \lambda_j \theta^2_j \;=\; t^2_N \sum_{j=2}^d
\lambda_j \Big(\frac{v_j \lambda_1 }{\lambda_j v_1}\Big)^2 
\;=\; t^2_N \frac{\lambda^2_1}{v^2_1} 
\Big(\frac{v_1^2}{\lambda_1} - \frac 1\alpha \Big) \;<\; 
t^2_N \lambda_1 \;\le\; [1+o_N(1)]\, \lambda_1\, \varepsilon^2_{N}
\;. 
\end{equation*}
The set $A'_{N}$ contains therefore a $(d-1)$-dimensional ball
centered at the origin and of radius $\delta \varepsilon_{N}$ for some
$\delta>0$. In particular, since $\varepsilon_{N} \sqrt{N}
\uparrow\infty$, the expression \eqref{601} is equal to
\begin{equation*}
\big[1+o_{N}(1)\big]\,
|\boldsymbol{u}_{1}\cdot \mf e_i|  \,
\frac{1}{\sqrt{\det(\mathbb{L}+\alpha\boldsymbol{p}^{\dagger}\boldsymbol{p})}}
\;.
\end{equation*}
To complete the proof, it remains to recall from (\ref{eq: rank one
  upd}) that 
\begin{equation*}
\det(\mathbb{L}+\alpha\boldsymbol{p}^{\dagger}\boldsymbol{p}) \;=\; 
(1+\alpha\boldsymbol{p}^{\dagger}\mathbb{L}^{-1}\boldsymbol{p})\det\mathbb{L}
\;=\; \alpha\, \Big(\frac{1}{\alpha}+\sum_{k=2}^{d}\frac{v_{k}^{2}}
{\lambda_{k}}\Big)\, \prod_{j=2}^{d}\lambda_{j} \;=\;
\alpha\frac{v_{1}^{2}}{\lambda_{1}}\prod_{j=2}^{d}\lambda_{j}\;. 
\end{equation*}
This proves the lemma in the case $\bs z\in \{\pm \mf e_1, \dots, \pm
\mf e_d\}$.

In order to extend the result \eqref{s141} to general $\bs{z}\in\mathbb{Z}^d$, we
proceed by induction on $|\bs{z}| = \sum_{1\le i \le d} |z_i|$. The
case $|\bs z|=1$ has been established above. For $|\bs{z}|>1$, we can
decompose $\bs z$ as $\bs{z}=\bs{w}_1 +\bs{w}_2$ where $\bs{w}_1, \bs{w}_2 \in
\mathbb{Z}^d$ and $0<|\bs{w}_1|$, $|\bs{w}_2|<|\bs{z}|$. First, if the
signs of $\bs{u}_1\cdot \bs{w}_1$ and $\bs{u}_1\cdot\bs{w}_2$ are
same, then $\ms{P}_{t_N}(N^{-1}(\bs{w}_1+\bs{w}_2))$ can be decomposed
into two disjoint sets $\ms{P}_{t_N}(N^{-1}\bs{w}_1)$ and
$\ms{P}_{t_N+N^{-1}\bs{w}_1}(N^{-1}\bs{w}_2)$. By the induction
hypothesis, \eqref{s141} holds for these two sets and therefore we can
verify that \eqref{s141} holds for $\bs{z}=\bs{w}_1 +\bs{w}_2$ as
well. On the other hand, if the signs of $\bs{u}_1\cdot \bs{w}_1$ and
$\bs{u}_1\cdot\bs{w}_2$ are different, since by assumption $\bs u_1
\cdot \bs z \neq 0$, assume, without loss of generality, that
$|\bs{u}_1\cdot \bs{w}_1|>|\bs{u}_1\cdot \bs{w}_2|$. Then,
$\ms{P}_{t_N} ( N^{-1}\bs{w}_1)$ can be decomposed into two disjoint
sets $\ms{P}_{t_N}(N^{-1}(\bs{w}_1+\bs{w}_2))$ and
$\ms{P}_{t_N+N^{-1}(\bs{w}_1+\bs{w}_2)}(-N^{-1}\bs{w}_2)$. Again by
the induction hypothesis we can check that \eqref{s141} is valid for
$\bs{z}$ as well. This completes the proof.
\end{proof}

\begin{corollary}
\label{s26}
There exists a finite constant $C_0$ such that
\begin{equation*}
\frac{1}{N \varepsilon_{N}} \, \frac{1}{N^{\frac{d-1}{2}}}
\sum_{\boldsymbol{x}\in{\mc B}_{N}}
e^{-\frac{N}{2}\boldsymbol{x}^{\dagger}(\mathbb{H}
+\alpha\boldsymbol{v}\boldsymbol{v}^{\dagger})\boldsymbol{x}}\;\le\;
C_0\;.   
\end{equation*}
\end{corollary}

\begin{proof}
Choose $a$ large enough for $\mc A_N(a)$ to contain $\mc B_N$.
Divide the set ${\mc B}_{N}$ into $O(N\varepsilon_{N})$ slices of the form
$\ms P_t(N^{-1}\mathfrak{e}_{i})$ for some $\mathfrak{e}_{i}$ such
that $\mathfrak{e}_{i}\cdot\boldsymbol{u}_{1}\neq 0$. The results
follows from Lemma \ref{s14}, by observing that all the estimates were
uniform on $|t_N|\le \varepsilon_N + r/N$.
\end{proof}

\subsection{The equilibrium potential near a saddle point.}

We now derive an approximation of the equilibrium
potential on the box $\overline{\mathcal{B}}_{N}$. By definition of the
generator, for a smooth function $f:\Xi\rightarrow\mathbb{R}$,
\begin{equation*}
\mathcal{L}_{N}f(\boldsymbol{x}) \;=\; \sum_{i=0}^{L-1}
e^{-N\left[\overline{F}_{N}(\boldsymbol{x}-\boldsymbol{z}_{i}^{N})-
F_{N}(\boldsymbol{x})\right]}
\left[f(\boldsymbol{x}+(\boldsymbol{z}_{i+1}^{N}
-\boldsymbol{z}_{i}^{N}))-f(\boldsymbol{x})\right]\;.
\end{equation*}
Performing a second-order Taylor expansion, and recalling that
$|\boldsymbol{x}| =O(\varepsilon_{N})$ for $\boldsymbol{x}
\in\overline{\mathcal{B}}_{N}$, we obtain that $\mathcal{L}_{N}
f(\boldsymbol{x}) = (1+\varepsilon_{N}) \widetilde {\mathcal{L}}_{N}
f(\boldsymbol{x})$, where $\widetilde {\mathcal{L}}_{N}$ is the
second-order differential operator given by
\begin{equation*}
\widetilde{\mathcal{L}}_{N}f(\boldsymbol{x}) \;=\; 
\frac{1}{2N^{2}} \, \sum_{i=0}^{L-1}\boldsymbol{w}_{i}^{\dagger}
[D^{2}f(\boldsymbol{x})] \boldsymbol{w}_{i} \;-\; \frac{1}{N}\,
\mathbb{M}\boldsymbol{x}\cdot\nabla f(\boldsymbol{x})\;,
\end{equation*}
where $\boldsymbol{w}_{i}=\boldsymbol{z}_{i+1}-\boldsymbol{z}_{i}$,
$0\le i<L$. It is not difficult to check that the function
$V_{N}:\overline{\mathcal{B}}_{N} \rightarrow(0,\,1)$ defined by
\begin{equation}
\label{06b}
V_{N}(\boldsymbol{x}) \;=\; \int_{-\infty}^{\sqrt{\alpha
    N}(\boldsymbol{x}\cdot\boldsymbol{v})}
\frac{1}{\sqrt{2\pi}}e^{-\frac{1}{2}y^{2}}dy\;,
\end{equation}
solves the equation $\widetilde{\mathcal{L}}_{N}V_{N}=0$ on
$\mathcal{B}_{N}$.  The function $V_{N}$ is therefore the natural
candidate for the approximation of the equilibrium potential
$V_{\mathcal{E}_{N}^{1}, \mathcal{E}_{N}^{2}}$ on $\mathcal{B}_{N}$.

The next result states that $\mathcal{L}_{N}V_{N}(\boldsymbol{x})$ is
small in the set $\mathcal{B}_{N}$, as expected from its definition.

\begin{lemma}
\label{s01}
There exists a finite constant $C_0<\infty$, independent of $N$,
such that
\begin{equation*}
|\mathcal{L}_{N}V_{N}(\boldsymbol{x})| \;\le\; 
C_0\, \frac{\varepsilon_{N}^{2}}{\sqrt{N}} \, 
\exp\Big\{ - \frac 12\, \alpha N(\boldsymbol{x}
\cdot\boldsymbol{v})^{2}\Big\} 
\end{equation*}
for all $\boldsymbol{x}\in\mathcal{B}_{N}$.
\end{lemma}

\begin{proof}
A straightforward calculation gives that there exists a finite
constant $C_0$ such that
\begin{equation*}
\big|\, \partial^k V_{N}(\boldsymbol{x})\,\big| \;\le\;
C_0 \, \sqrt{N}  \, (\varepsilon_N N)^{k-1} \, \exp\Big\{ - \frac 12\,
\alpha N(\boldsymbol{x}\cdot\boldsymbol{v})^{2}\Big\} 
\end{equation*}
for all $1\le k\le 3$, $\bs x\in \overline{\mathcal{B}}_N$, where $\partial^k
V_{N}$ represents any partial derivative of $V_N$ of order $k$.  On
the other hand, by the Taylor expansion, $\partial_j F(\boldsymbol{x})
\;=\; (\bb H \, \bs x)_j + O(\varepsilon^2_{N})$ so that
\begin{equation*}
-N\, [\overline{F}_{N}(\boldsymbol{x}-\boldsymbol{z}_{i}^{N})
-F_{N}(\boldsymbol{x})] \;=\;
(\boldsymbol{z}_{i}-\bar{\boldsymbol{z}})\cdot
\bb H\, \boldsymbol{x} \,+\, O(\varepsilon^2_{N})\;.
\end{equation*}
It follows from these estimates and from the fact that
$\widetilde{\mathcal{L}}_{N} V_{N}$ vanishes on $\mc{B}_N$ that
\begin{equation*}
\big|\, (\mc L_N V_{N}) (\boldsymbol{x})\,\big| \;=\;
\big|\, (\mc L_N V_{N}) (\boldsymbol{x}) - 
(\widetilde{\mc L}_N V_{N}) (\boldsymbol{x})\,\big| \;\le\;
C_0 \, N^{-1/2} \varepsilon_{N}^{2} \, e^{- (1/2) \alpha N
(\boldsymbol{x}\cdot\boldsymbol{v})^{2}}\;,
\end{equation*}
as claimed.
\end{proof}

The next lemma asserts that the value of the function $V_{N}$ at
the boundary of $\mc B_N$ is close to the one of the equilibrium
potential. For $r>0$, let 
\begin{align*}
&\partial_r^1 \mc B_N \;=\; \big\{\bs{x}\in\overline{\mc B}_N: 
\textup{dist}(\bs{x},\ms{P}_{\varepsilon_N})\le N^{-1}r
\big\}\;,\\
&\qquad \partial_r^2 \mc B_N\; =\; \big\{\bs{x}\in\overline{\mc B}_N: 
\textup{dist}(\bs{x},\ms{P}_{-\varepsilon_N})\le N^{-1}r
\big\}\;,
\end{align*}
where $\textup{dist}(\cdot,\cdot)$ denotes the usual Euclidean
distance. 

\begin{lemma}
\label{s18}
For every $r>0$, there exist constants $0<c_0<C_0$, independent of
$r$ and $N$, and $N_0=N_0(r)$ such that
\begin{align*}
& e^{-N[\overline{F}_{N}(\boldsymbol{x})-H]}\left[1-V_{N}(\boldsymbol{x})\right]^{2}
\;\le\; C_0\, e^{-c_0 N\varepsilon_{N}^{2}}\;, \quad \forall\, 
\boldsymbol{x}\in\partial_r^1 \mc B_N \;, \\
&\qquad e^{-N[\overline{F}_{N}(\boldsymbol{x})-H]}
V_{N}(\boldsymbol{x})^{2}
\;\le\;  C_0\, e^{-c_0 N\varepsilon_{N}^{2}}\;,\quad
\forall\, \bs{x}\in\partial_r^2 \mc B_N \;,
\end{align*}
for all $N\ge N_0$.
\end{lemma}

\begin{proof}
We prove the first estimate, the second one being identical. By
definition of $V_{N}$ and by a Taylor expansion of $F_N$, it suffices to
show that there exist $c_0>0$ such that
\begin{equation}
\label{05}
e^{-\frac{N}{2}\boldsymbol{x}^{\dagger}\mathbb{H}\boldsymbol{x}}
\Big( \int_{\sqrt{\alpha
    N}(\boldsymbol{x}\cdot\boldsymbol{v})}^{\infty}
\frac{1}{\sqrt{2\pi}}e^{-\frac{1}{2}y^{2}}dy\Big)^{2}
\;\le\; e^{-c_0 N\varepsilon_{N}^{2}}
\end{equation}
for all $\bs x\in \partial_r^{1}\mathcal{B}_{N}$.

In view of \eqref{02}, let $0<\delta<v_1$ such that
\begin{equation*}
\Big(\frac{v_{1}^{2}}{\lambda_{1}}-\frac{1}{\alpha}\Big)\, 
(\lambda_{1}+\delta) \;<\; (v_{1}-\delta)^{2}\;.
\end{equation*}
We claim that there exists $N_0 = N_0(r)$ such that for all $N\ge N_0$
and for all $\boldsymbol{x} \in\partial_r^{1} \mathcal{B}_{N}$ either
$\boldsymbol{x}\cdot\boldsymbol{v} \ge \delta\varepsilon_{N}$ or
$\boldsymbol{x}^{\dagger} \mathbb{H}\boldsymbol{x} \ge
\delta\varepsilon_{N}^{2}$ holds.  

Indeed, fix $N\ge 1$, $\bs x \in \partial_r^{1}\mathcal{B}_{N}$ and
suppose that $\boldsymbol{x}\cdot\boldsymbol{v}<\delta\varepsilon_{N}$
and $\boldsymbol{x}^{\dagger} \mathbb{H} \boldsymbol{x}
<\delta\varepsilon_{N}^{2}$. Since $\boldsymbol{x}$ belongs to
$\partial_r^{1}\mathcal{B}_{N}$, on the basis $(\bs u_1, \dots, \bs
u_d)$ it can be expressed as
\begin{equation*}
\boldsymbol{x} \;=\; \varepsilon_{N} (1+r_N) \boldsymbol{u}_{1} \;+\;
\varepsilon_{N} \sum_{k=2}^{d}x_{k} \boldsymbol{u}_{k} \;,
\end{equation*}
where $\varepsilon_{N} |r_N| \le r / N$. Since $\bs v = \sum_i v_i \bs
u_i$, the two conditions on $\bs x$ can be written as 
\begin{equation*}
(1+r_N) v_{1}-\delta \;<\; -\sum_{k=2}^{d}x_{k}v_{k}
\quad\text{and}\quad 
\sum_{k=2}^{d}x_{k}^{2}\lambda_{k} \;<\; \lambda_{1}(1+r_N)^2 \;+\; \delta\;.
\end{equation*}
Since $\delta<v_1$, there exists $N_1=N_1(\delta, r)$ such that
$(1+r_N) v_{1}-\delta >0$ for all $N\ge N_1$. Hence, by taking the
square in the first inequality and by applying the Cauchy-Schwarz
inequality and recalling \eqref{02}, we have that
\begin{align*}
&\big[(1+r_N) v_{1}-\delta\big]^2 \;\le\;  
\Big( -\sum_{k=2}^{d}x_{k}v_{k} \Big)^{2}\;\le\; 
\sum_{k=2}^{d}\frac{v_{k}^{2}}{\lambda_{k}}\,
\sum_{k=2}^{d}x_{k}^{2}\lambda_{k}  \\
&\quad<\;
\big[\lambda_{1}(1+r_N)^2 \;+\; \delta \big]
\sum_{k=2}^{d}\frac{v_{k}^{2}}{\lambda_{k}} 
\;\le\; \big[\lambda_{1}(1+r_N)^2 \;+\; \delta \big]
\Big(\frac{v_{1}^{2}}{\lambda_{1}}-\frac{1}{\alpha}\Big)
\;.
\end{align*}
Since $r_N\to 0$, there exists $N_2=N_2(r)$ such that for all $N\ge
N_2$ this relation is a contradiction with the definition of
$\delta$, which proves the claim.

Assume first that $\boldsymbol{x} \cdot\boldsymbol{v}
\ge\delta\varepsilon_{N}$. In this case, since $N\varepsilon_{N}^2
\uparrow\infty$, for $N$ sufficiently large the left hand side of
\eqref{05} is bounded above by
\begin{equation*}
e^{-\frac{N}{2}\boldsymbol{x}^{\dagger}\mathbb{H}\boldsymbol{x}}
\Big(\frac{1}{\sqrt{\alpha N}(\boldsymbol{x}\cdot\boldsymbol{v})}
\frac{1}{\sqrt{2\pi}}e^{-\frac{1}{2}\alpha N
(\boldsymbol{x}\cdot\boldsymbol{v})^{2}}\Big)^{2} \;\le\; 
e^{-\frac{N}{2}\boldsymbol{x}^{\dagger}(\mathbb{H}+2\alpha\boldsymbol{v}
\boldsymbol{v}^{\dagger})\boldsymbol{x}}
\;\le\; e^{-\frac{\tau}{2}N|\boldsymbol{x}|^{2}}\;,
\end{equation*}
where $\tau$ is the smallest eigenvalue of the positive definite
matrix $\mathbb{H}+2\alpha\boldsymbol{v}\boldsymbol{v}^{\dagger}$. To
complete the proof of \eqref{05}, with $c_0=\tau/2$ and under the
hypothesis that $\boldsymbol{x} \cdot\boldsymbol{v}
\ge\delta\varepsilon_{N}$, it remains to recall that
$|\boldsymbol{x}|\ge\varepsilon_{N}- r/N$ for
$\boldsymbol{x}\in\partial_r^{1}\mathcal{B}_{N}$.

Suppose now that $\boldsymbol{x}^{\dagger}
\mathbb{H}\boldsymbol{x}\ge\delta\varepsilon_{N}^{2}$. In this case,
the left hand side of \eqref{05} is bounded above by $ \exp\{- (N/2)
\boldsymbol{x}^{\dagger}\mathbb{H} \boldsymbol{x}\} \le \exp\{ -(N/2)
\delta \varepsilon_{N}^{2}\}$, which completes the proof of the lemma.
\end{proof}

Denote by $\mathcal{D}_{N}(f;U)$ the Dirichlet form of a function $f:
\Xi_{N} \to \bb R$ restricted to a subset
$U\subset\widehat{\Xi}_{N}$:
\begin{equation*}
\mathcal{D}_{N}(f;U) \;=\; \sum_{\bs{x}\in
  U} \mathcal{D}_{N,\boldsymbol{x}}(f) \;.
\end{equation*}
Let $\{\kappa_{N}: N\ge 1\}$ be the sequence
\begin{equation}
\label{13}
\kappa_{N} \;=\; \frac 1{Z_{N}}\, (2\pi N)^{\frac{d}{2}-1}\, e^{-NH}\;,
\end{equation}
where $H$ represents the height of the saddle points.  We also recall from \eqref{45} that
\begin{equation*}
\omega_{\bs{0}}=e^{-G(\bs{0})}\frac{\mu}{\sqrt{-\det\mathbb{H}}}\;.
\end{equation*}

\begin{proposition}
\label{s02}
We have that
\begin{equation*}
\mathcal{D}_{N}(V_{N};\,\mathring{\mathcal{B}}_{N}) 
\; =\; \big[ 1+ o_N(1) \big]\, \kappa_N \,\omega_{\bs{0}}\;.
\end{equation*}
\end{proposition}

\begin{proof}
By (\ref{eq: Decomposition of Diri Form -2}), we can write Dirichlet
form as 
\begin{equation*}
\mathcal{D}_{N}(V_{N};\,\mathring{\mathcal{B}}_{N}) 
\;=\; \frac{1}{2}\sum_{\boldsymbol{x}\in\mathring{\mathcal{B}}_{N}}
Z_{N}^{-1}e^{-N\overline{F}_{N}(\boldsymbol{x})}\sum_{i=0}^{L-1}
[V_{N}(\boldsymbol{x}+\boldsymbol{z}_{i+1}^{N})-
V_{N}(\boldsymbol{x}+\boldsymbol{z}_{i}^{N})]^{2}.
\end{equation*}
In view of the definition \eqref{06b} of $V_N$ and by the Taylor expansion, 
\begin{equation*}
V_{N}(\boldsymbol{x}+\boldsymbol{z}_{i+1}^{N})-
V_{N}(\boldsymbol{x}+\boldsymbol{z}_{i}^{N}) \;=\; 
\big[1+o_{N}(1)\big]\, \sqrt{\frac{\alpha}{2\pi N}}\,
\big[(\boldsymbol{z}_{i+1}-\boldsymbol{z}_{i}) \cdot
\boldsymbol{v}\big]
\, e^{-\frac{1}{2}\alpha N(\boldsymbol{x}\cdot\boldsymbol{v})^{2}}
\end{equation*}
where the error term is $o_N(1)$ uniformly in $x\in \mc{B}_N$ and $0\le i<L$. 
In particular, the right hand side of the penultimate formula can be
rewritten as
\begin{equation*}
\big[1+o_{N}(1)\big]  \, 
\frac{\alpha \,e^{-NH-G(\bs{0})}}{4\pi N Z_{N}} \sum_{i=0}^{L-1}
\big[(\boldsymbol{z}_{i+1}-\boldsymbol{z}_{i})\cdot\boldsymbol{v}\big]^{2}
\sum_{x\in\mathring{\mathcal{B}}_{N}}
e^{-N\left\{ \frac 1 2 \, \boldsymbol{x}^{\dagger}
\mathbb{H}\boldsymbol{x}+ \alpha(\boldsymbol{x}\cdot\boldsymbol{v})^{2}\right\}}\;.  
\end{equation*}
To complete the proof, it remains to use the relation
\begin{equation*}
\sum_{i=0}^{L-1}\left[(\boldsymbol{z}_{i+1}-\boldsymbol{z}_{i})
\cdot\boldsymbol{v}\right]^{2} \;=\;
\boldsymbol{v}^{\dagger}[\mathbb{A} + \bb A^\dagger] \boldsymbol{v}\;=\;
2\boldsymbol{v}^{\dagger}\mathbb{A}\boldsymbol{v}\;=\;\frac{2\mu}{\alpha}\;,
\end{equation*}
where the last identity follows from the definition \eqref{06} of
$\alpha$, and to recall the statement of Lemma \ref{s25}.
\end{proof}

\subsection{Adjoint dynamics}
\label{ad}

We have presented in this section an approximation $V_N$ of the
equilibrium potential $V_{\mc E^1_N,\mc E^2_N}$. All the arguments
presented in this section can be carried to the adjoint process,
providing an approximation, denoted by $V^*_N$, of the equilibrium
potential $V^*_{\mc E^1_N,\mc E^2_N}$.

Indeed, denote by $\mathbb{M}^{*}$ the Jacobian of the adjoint drift
$b^*$: $\mathbb{M}^{*}: =Db^{*} (\boldsymbol{\bs \sigma})$. The
Jacobian can be written as $\mathbb{M}^{*}= \mathbb{A}^{\dagger}
\mathbb{H}= \mathbb{H}^{-1} \mathbb{M}^{\dagger} \mathbb{H}$, and thus
the eigenvalues of $(\mathbb{M}^{*})^{\dagger}$ coincide with the ones
of $\bb M$. In particular, the unique negative eigenvalue of $\bb
M^*$, denoted by $-\mu^*$, is equal to $-\mu$. Let $\bs v^*$ be the
associated eigenvector.

We may define $\alpha^*$ as $\alpha$ has been defined in \eqref{02} by
replacing $\bs v$ by $\bs v^*$ and $\bb A$ by $\bb
A^\dagger$. Clearly, all identities presented for $\alpha$ also hold
for $\alpha^*$ with $\bs v$ replaced by $\bs v^*$. Lemmata \ref{s17}
and \ref{s21} are also in force with the ad-hoc modifications.

The definition of the sets $\mc C_N$ and $\mc B_N$ depend only on the
Hessian of $F$ at the saddle point, and therefore coincide for the
adjoint dynamics. Lemma \ref{s25} to Corollary \ref{s26} holds for the
adjoint if we replace everywhere $\alpha$, $\bs v$ by $\alpha^*$, $\bs
v^*$, respectively.

Finally, if we define the function $V^*_{N}:\overline{\mathcal{B}}_{N}
\rightarrow (0,1)$ by
\begin{equation*}
V^*_{N}(\boldsymbol{x}) \;=\; \int_{-\infty}^{\sqrt{\alpha^*
    N}(\boldsymbol{x}\cdot\boldsymbol{v}^*)}
\frac{1}{\sqrt{2\pi}}e^{-\frac{1}{2}y^{2}}dy\;,
\end{equation*}
we can prove all statements presented from Lemma \ref{s01} to
Proposition \ref{s02} with $\mc L^*_N$, $V^*_N$, $\alpha^*$, $\bs v^*$
replacing $\mc L_N$, $V_N$, $\alpha$, $\bs v$, respectively.

\section{Flows at saddle points} 
\label{sec:flow}

In the previous section, for a fixed saddle point $\bs \sigma \in\mf
S_{1,2}$, we introduced the functions $V_N$, $V^*_N$ which approximate
the equilibrium potential $V_{\mc E^1_N, \mc E^2_N}$, $V^*_{\mc E^1_N,
  \mc E^2_N}$, respectively, in a mesoscopic neighborhood
$\overline{\mc B}_N$ of $\bs{\sigma}$.  In this section we present
flows which approximate the flows $\Phi^*_{V_N}$, $\Phi_{V^*_N}$ in
$\mc B_N$.  These flows, indexed by the saddle points, are the
building blocks on which we construct, in the next sections, the flows
approximating $\Phi^*_{V_{\mc E_N(A), \mc E_N(B)}}$, $\Phi_{V^*_{\mc
    E_N(A), \mc E_N(B)}}$, $A$, $B\subset S$, $A\cap B= \varnothing$.

We introduce below, in \eqref{40}, the flow $\Phi_N$ and we present in
Section \ref{sub51} its main properties. In Section \ref{sub54} we
present the flow $\Phi^*_N$.

For $\bs z\in \widehat\Xi_N$ and a function
$f:\Xi_{N}\rightarrow\mathbb{R}$, define a flow
$\Phi_{f,\boldsymbol{z}}^{*}$, supported on
$\gamma_{\boldsymbol{z}}^{N}$, by
\begin{equation}
\label{34}
\Phi_{f,\boldsymbol{z}}^{*} (\boldsymbol{x},\boldsymbol{y})
\;=\;f(\boldsymbol{x})
c_{\boldsymbol{z}}(\boldsymbol{y},\boldsymbol{x})
-f(\boldsymbol{y})c_{\boldsymbol{z}}(\boldsymbol{x},\boldsymbol{y})\;,
\quad \bs x\,,\; \bs y \in\Xi_N\;.
\end{equation}
Fix a saddle point $\bs \sigma\in \mf S_{1,2}$ and let $\Phi_{N}$ be
the flow defined by
\begin{equation}
\label{40}
\Phi_{N}\;=\;\sum_{\boldsymbol{z}\in\mathring{\mathcal{B}}_{N}}
\Phi_{V_{N},\boldsymbol{z}}^{*}\;,
\end{equation}
where $V_N$ is the approximation of the equilibrium potential
introduced in the previous section. Recall from Section 2 that we
denote by $\bs{m}_i$ one of the global minima of $F$ on $W_i$ and by
$[\bs{m}_i]_N =:\bs{m}_i^N$ the discrete approximation of $\bs{m}_i$.

\begin{theorem}
\label{s24} 
Assume without loss of generality that $\bs \sigma$ belongs to $\mf
S_{1,2}$.  There exist a flow $\widetilde{\Phi}_{N}
\in\mathcal{F}_{N}$ which is divergence-free on
$\{\bs{m}_1^N ,\bs{m}_2^N \}^{c}$, and such that
\begin{gather*}
\big(\textup{div}\,
  \widetilde{\Phi}_{N}\big)(\bs{m}_1^N) 
\;=\;  - \big(\textup{div}\,
  \widetilde{\Phi}_{N}\big)(\bs{m}_2^N)
\;=\;\kappa_N \, \omega_{\bs{0}}\;, \\
\Vert \widetilde{\Phi}_{N} -  \Phi_N \Vert^2 \;=\; \kappa_N\,o_N(1)\;.
\end{gather*}
\end{theorem}

The proof of Theorem \ref{s24} is given in Section \ref{sub53}.  As in
the previous section, we assume below that $\bs{\sigma}=\bs{0}$.

\subsection{The divergence of $\Phi_N$.} 
\label{sub51}

In this subsection we examine the divergence of the flow $\Phi_N$.

\begin{lemma}
\label{lf1}
The flow $\Phi_N$ is divergence-free on $\Xi_{N}\setminus
\overline{\mathcal{B}}_{N}$, and for every $\bs x$ in $\mc B_N$
\begin{equation*}
(\textup{div }\Phi_{N})(\boldsymbol{x})\;=\;-\,
\frac 1{Z_N}\, e^{-NF_{N}(\boldsymbol{x})}
(\mathcal{L}_{N}V_{N}) (\boldsymbol{x}) \;.
\end{equation*}
\end{lemma}

\begin{proof}
Fix a point $\bs z\in \mathring{\mathcal{B}}_{N}$.  Since the
conductance $c_{\bs z}(\bs x, \bs y)$ vanishes unless $(\bs x, \bs y)
= (\bs z + \bs z^N_i, \bs z + \bs z^N_{i+1})$ for some $0\le i<L$,
\begin{equation*}
\Phi_{V_N,\bs z}^{*} (\boldsymbol{x},\boldsymbol{y})
\;=\; 
\begin{cases}
-V_N(\bs{y}) \, c_{\bs z} (\bs x, \bs y)
 & \text{ if $(\bs x, \bs y) = (\bs z + \bs z^N_i, \bs z + \bs
   z^N_{i+1})$}\;, \\
V_N(\bs y) \, c_{\bs z} (\bs x, \bs y)
 & \text{ if $(\bs x, \bs y) = (\bs z + \bs z^N_{i+1}, \bs z + \bs
   z^N_{i})$}\;, \\
0 & \text{otherwise}\;.
\end{cases}
\end{equation*}
In particular, $\Phi_{V_N,\bs z}^{*} (\bs x, \bs y) = 0$ if $\bs x$ does
not belong to the cycle $\gamma^N_{\bs z}$ so that
\begin{equation*}
(\textup{div }\Phi_{V_N,\boldsymbol{z}}^{*})(\boldsymbol{x}) \;=\; 0
\quad\text{if } \bs x \not\in \gamma^N_{\bs z}\;.
\end{equation*}

By the additivity of the divergence functional, since $\bs x$ belongs
only to the cycles $\gamma^N_{\bs x-\bs z^N_i}$, $0\le i<L$, for every
$\bs x\in\Xi_N$,
\begin{equation*}
(\textup{div }\Phi_{N})(\boldsymbol{x}) \;=\;
\sum_{\boldsymbol{z}\in\mathring{\mathcal{B}}_{N}}
(\textup{div }\Phi_{V_{N},\boldsymbol{z}}^{*})(\boldsymbol{x})
\;=\; \sum_{i=0}^{L-1} (\textup{div }\Phi_{V_{N},\boldsymbol{x}-\boldsymbol{z}_{i}^{N}}^{*})
(\boldsymbol{x}) \, \mathbf{1}\{\boldsymbol{x}-\boldsymbol{z}_{i}^{N}
\in\mathring{\mathcal{B}}_{N}\}\;.
\end{equation*}
Therefore, by definition of $\overline{\mc B}_N$, $(\textup{div
}\Phi_{N})(\boldsymbol{x})=0$ for all $\boldsymbol{x}\in \Xi_{N}
\setminus \overline{\mathcal{B}}_{N}$, which is the firs assertion of
the lemma.

Fix $\boldsymbol{x}\in \overline{\mc B}_N$, $0\le i < L$ and assume
that $\bs x - \bs z^N_i$ belongs to $\mathring{\mc B}_N$. It follows
from the first formula of the proof and from the explicit formula
\eqref{eq:3} for the conductance that
\begin{align*}
&(\textup{div }\Phi_{V_{N},\boldsymbol{x}-\boldsymbol{z}_{i}^{N}}^{*})
(\boldsymbol{x}) \\
&\quad =\; 
\Phi_{V_{N},\boldsymbol{x}-\boldsymbol{z}_{i}^{N}}^{*} 
(\boldsymbol{x},\boldsymbol{x}-\boldsymbol{z}_{i}^{N}+\boldsymbol{z}_{i+1}^{N})
\;-\; \Phi_{V_{N},\boldsymbol{x}-\boldsymbol{z}_{i}^{N}}^{*}
(\boldsymbol{x}-\boldsymbol{z}_{i}^{N}+\boldsymbol{z}_{i-1}^{N},\boldsymbol{x})
\\
&\quad 
=\;  \frac 1{Z_N}\,e^{-N\overline{F}_{N}(\boldsymbol{x}-\boldsymbol{z}_{i}^{N})}
\left[-V_{N}(\boldsymbol{x}-\boldsymbol{z}_{i}^{N}+
\boldsymbol{z}_{i+1}^{N})+V_{N}(\boldsymbol{x})\right]\;.
\end{align*}
Thus, summing over $i$ and in view of the penultimate displayed
equation, for every $\boldsymbol{x}\in \overline{\mc B}_N$,
\begin{align}
\label{eq:div(Phi_N)}
&(\textup{div }\Phi_{N})(\boldsymbol{x})\\
&=\; \frac 1{Z_N}
\sum_{i=0}^{L-1}
e^{-N\overline{F}_{N}(\boldsymbol{x}-\boldsymbol{z}_{i}^{N})}
\left[V_{N}(\boldsymbol{x})-V_{N}(\boldsymbol{x}+\boldsymbol{z}_{i+1}^{N}
-\boldsymbol{z}_{i}^{N})\right]\, 
\mathbf{1}\{\boldsymbol{x}-\boldsymbol{z}_{i}^{N}\in\mathring{\mathcal{B}}_{N}\}\;.
\nonumber
\end{align}
If $\bs x$ belongs to $\mc B_N$, $\bs x - \bs z^N_i \in \mathring{\mc
  B}_N$ for all $0\le i<L$, and we may remove the indicator in the
previous formula. This completes the proof of the lemma.
\end{proof}

By definition of the flow $\Phi_N$, and in view of \eqref{34},
since the function $V_N$ is absolutely bounded by $1$, $|\Phi_{N}
(\boldsymbol{x}, \bs y)| \le 2 c^s (\bs x, \bs y)$. Hence, by
the explicit form of the symmetric conductance, and since $|F_N(\bs w)
- \overline{F}_N(\bs z)|\le C_0/N$ if $\bs w$ belongs to the cycle
$\gamma^N_{\bs z}$,
\begin{equation}
\label{37}
\big |\Phi_{N} (\boldsymbol{x}, \bs y) \big | \;\le\;
C_0 \, \frac 1{Z_N}\, e^{-N F_N(\bs x)} \quad\text{and}\quad
\big |\Phi_{N} (\boldsymbol{x}, \bs y) \big | \;\le\;
C_0 \, \frac 1{Z_N}\, e^{-N F_N(\bs y)} 
\end{equation}
for a finite constant $C_0$ independent of $N$.

The next lemma of this section asserts that the divergence of
$\Phi_{N}$ on $\mathcal{B}_{N}$ is small. This result is in accordance
with the fact that $\Phi_{V_{\mathcal{E}_N^1, \mathcal{E}_N^2}}^{*}$
is divergence-free on $\mathcal{B}_N$.

\begin{lemma}
\label{s19}
We have that $\sum_{\boldsymbol{x}\in\mathcal{B}_{N}}
\left| (\mbox{\textup{div} }\Phi_{N})(\boldsymbol{x})\right| =
\kappa_N \, o_N(1)$. 
\end{lemma}

\begin{proof}
By Lemma \ref{s01}, by the second assertion of Lemma \ref{lf1}, and by
a Taylor expansion of $F_{N}$, there exists a finite constant $C_0$
such that for all $\boldsymbol{x}\in\mathcal{B}_N$,
\begin{equation}
\label{31}
|(\textup{\mbox{div}}\,\Phi_{N})(\boldsymbol{x})| \;\le\; 
C_0 e^{-NH}\, \frac{\varepsilon_{N}^{2}}{\sqrt{N}} \,\frac 1{Z_{N}}\, 
e^{-\frac{N}{2}\boldsymbol{x}^{\dagger}(\mathbb{H}+\alpha\boldsymbol{v}
\boldsymbol{v}^{\dagger})\boldsymbol{x}}\;.
\end{equation}
Hence, the sum appearing in the statement of the lemma is bounded
above by
\begin{equation*}
C_0\, \kappa_N \, \varepsilon_{N}^{2} \frac {1}{N^{(d-1)/2}}
\sum_{\boldsymbol{x}\in\mathcal{B}_{N}}
e^{-\frac{N}{2}\boldsymbol{x}^{\dagger}(\mathbb{H}
+\alpha\boldsymbol{v}\boldsymbol{v}^{\dagger})\boldsymbol{x}} \;. 
\end{equation*}
By Corollary \ref{s26}, this expression is less than or
equal to $C_0 \kappa_N N \varepsilon_{N}^{3}$ for some finite constant $C_0$
independent of $N$.  By (\ref{vare1}), $N
\varepsilon_{N}^{3}=o_N(1)$ and the proof of the proposition is completed.
\end{proof}

\begin{proposition}
\label{s03}
We have that 
\begin{align*}
& \sum_{\boldsymbol{x}\in\partial^{1}\mathcal{B}_{N}}
(\textup{\mbox{div}}\,\Phi_{N})(\boldsymbol{x})
\;=\; \big[1+ o_{N}(1)\big]\, \kappa_N\,\omega_{\bs{0}} \;, \\
&\quad \sum_{\boldsymbol{x}\in\partial^{2}\mathcal{B}_{N}}
(\textup{\mbox{div}}\,\Phi_{N})(\boldsymbol{x}) \;=\; 
-\, \big[1+ o_{N}(1)\big]\, \kappa_N \,\omega_{\bs{0}}\;.
\end{align*}
\end{proposition}

\begin{proof}
We prove the first estimate, the arguments for the second one being
analogous.  By (\ref{eq:div(Phi_N)}) and by a change of variables, the
first sum appearing in the statement of the proposition is equal to
\begin{align*}
& \frac 1{Z_N}\sum_{\boldsymbol{x}\in\partial^{1}\mathcal{B}_{N}}
 \, \sum_{i=0}^{L-1}
e^{-N\overline{F}_{N}(\boldsymbol{x}-\boldsymbol{z}_{i}^{N})}
\left[V_{N}(\boldsymbol{x})-V_{N}(\boldsymbol{x}+\boldsymbol{z}_{i+1}^{N}
-\boldsymbol{z}_{i}^{N})\right] \, 
\mathbf{1}\{\boldsymbol{x}-\boldsymbol{z}_{i}^{N}\in\mathring{\mathcal{B}}_{N}\} \\
&\qquad =  \; \frac 1{Z_N} \sum_{i=0}^{L-1}
\sum_{\substack{\boldsymbol{x}+\boldsymbol{z}_{i}^{N}\in
\partial^{1}\mathcal{B}_{N}}} 
e^{-N\overline{F}_{N}(\boldsymbol{x})}\,
\left[V_{N}(\boldsymbol{x}+\boldsymbol{z}_{i}^{N})-
V_{N}(\boldsymbol{x}+\boldsymbol{z}_{i+1}^{N})\right]\,
\mathbf{1}\{\boldsymbol{x}\in\mathring{\mathcal{B}}_{N}\}\;.
\end{align*}
We may rewrite the last sum as
\begin{equation}
\label{f02}
\frac 1{Z_N} 
\sum_{\boldsymbol{x}\in\mathring{\mathcal{B}}_{N}}
e^{-N\overline{F}_{N}(\boldsymbol{x})}\, 
\sum_{i:\boldsymbol{x}+\boldsymbol{z}_{i}^{N}
\in\partial^{1}\mathcal{B}_{N}}
\left[V_{N}(\boldsymbol{x}+\boldsymbol{z}_{i}^{N})-
V_{N}(\boldsymbol{x}+\boldsymbol{z}_{i+1}^{N})\right] \;,
\end{equation}
where the second sum is carried over the indices $i$ which satisfy the
conditions appearing below the sum.  

Recall the definition of the set $\mc A_N(a)$, $a>0$, introduced in
\eqref{f05} and set $\mc A_N = \mc A_N(1/4)$.  By Taylor expansion,
every point $\bs x \in \mc A^c_N$ such that $\bs x\cdot \bs u_1 =
[1+o_N(1)] \varepsilon_N$,
\begin{equation}
\label{f04}
F_N(\bs x) - H \;\ge \; [1+o_N(1)]\, \frac 18 \, \lambda_1
\varepsilon^2_N\;. 
\end{equation}

We claim that we may restrict the sum appearing in \eqref{f02} to
points $\bs x$ in $\mathring{\mc B}_N \cap \mc A_N$. Indeed, fix $\bs
x$ in $\mathring{\mc B}_N \setminus \mc A_N$. Since
$\boldsymbol{x}+\boldsymbol{z}_{i}^{N}$ belongs to
$\partial^{1}\mathcal{B}_{N}$, $\bs x\cdot \bs u_1 = [1+o_N(1)]
\varepsilon_N$. Hence, by the previous paragraph, $F_N(\bs x) - H
\;\ge \; [1+o_N(1)]\, (1/8) \, \lambda_1 \varepsilon^2_N$, and the
sum \eqref{f02} carried over points $\bs x\in\mathring{\mc B}_N
\setminus \mc A_N$ is of order $\kappa_N \, o_N(1)$ because $V_N$ is
bounded by $1$.

Consider the sum \eqref{f02} carried over points $\bs x$ in
$\mathring{\mc B}_N \cap \mc A_N$.  By Taylor's expansion, this
expression is equal to
\begin{equation}
\label{eq:reat}
\big[1+o_{N}(1)\big] \, \frac 1{N\, Z_N}
\sum_{\boldsymbol{x}\in\mathring{\mathcal{B}}_{N}\cap\mc A_N}
e^{-N F_{N}(\boldsymbol{x})}\sum_{i:\boldsymbol{x}+\boldsymbol{z}_{i}^{N}
\in\partial^{1}\mathcal{B}_{N}}
(\boldsymbol{z}_{i}-\boldsymbol{z}_{i+1})\cdot\nabla V_{N}(\boldsymbol{x})\;.
\end{equation}
Note that $\overline{F}_{N}(\boldsymbol{x})$ has been replaced by
$F_{N}(\boldsymbol{x})$.  Changing variables and taking advantage of
the cancellations, the sum over $i$ can be rewritten as
\begin{equation}
\label{f03}
\sum_{\substack{\boldsymbol{x}+\boldsymbol{z}_{i}^{N}\in\partial^{1}\mathcal{B}_{N}\\
\boldsymbol{x}+\boldsymbol{z}_{i-1}^{N}\notin\partial^{1}\mathcal{B}_{N}}}
\boldsymbol{z}_{i}\cdot\nabla V_{N}(\boldsymbol{x}) \;-\;
\sum_{\substack{\boldsymbol{x}+\boldsymbol{z}_{i}^{N}\notin\partial^{1}\mathcal{B}_{N}\\
\boldsymbol{x}+\boldsymbol{z}_{i-1}^{N}\in\partial^{1}\mathcal{B}_{N}}}
\boldsymbol{z}_{i}\cdot\nabla V_{N}(\boldsymbol{x})\;,
\end{equation}
where both sums are carried over the indices $i$ which satisfy the
conditions appearing below the sums. 

In view of \eqref{f03}, fix a point $\bs x$ in $\mathring{\mc
  B}_N\cap\mc A_N$ such that $\bs x + \bs z^N_j \in \partial^1 \mc
B_N$ for some $0\le j<L$ and that $\bs x + \bs z^N_k\not
\in \partial^1 \mc B_N$ for some $0\le k<L$. We claim that $\bs x +
\bs z^N_k \in \mc B_N$.

Indeed, since $\bs x\in \mathring{\mc B}_N$, $\bs x + \bs z^N_k\in
\overline{\mc B}_N$. Therefore, to prove that $\bs x + \bs z^N_k \in
\mc B_N$ it is enough to show that $\bs x + \bs z^N_k
\not\in \partial{\mc B}_N$. By the paragraph succeeding \eqref{27},
$\partial{\mc B}_N = \partial^0{\mc B}_N \cup \partial^1{\mc B}_N
\cup \partial^2{\mc B}_N$. 

Since $\bs x$ belongs to $\mathring{\mc B}_N\cap\mc A_N$ and $\bs x +
\bs z^N_j \in \partial^{1} \mc B_N$ for some $0\le j<L$, as in
\eqref{f04}, a Taylor expansion shows that $F(\bs x + \bs z^N_k) - H
\le [1+o_N(1)] (1/8) \lambda_1 \varepsilon^2_N$. In particular, by
\eqref{27}, $\bs x + \bs z^N_k$ does not belong to $\partial^0{\mc
  B}_N$. The point $\bs x + \bs z^N_k$ does not belong to
$\partial^1{\mc B}_N$ by assumption, and can not belong to
$\partial^2{\mc B}_N$ because $\bs x + \bs z^N_j$ belongs to
$\partial^1{\mc B}_N$. This proves the claim

\begin{figure}
\centering
\includegraphics[scale=0.21]{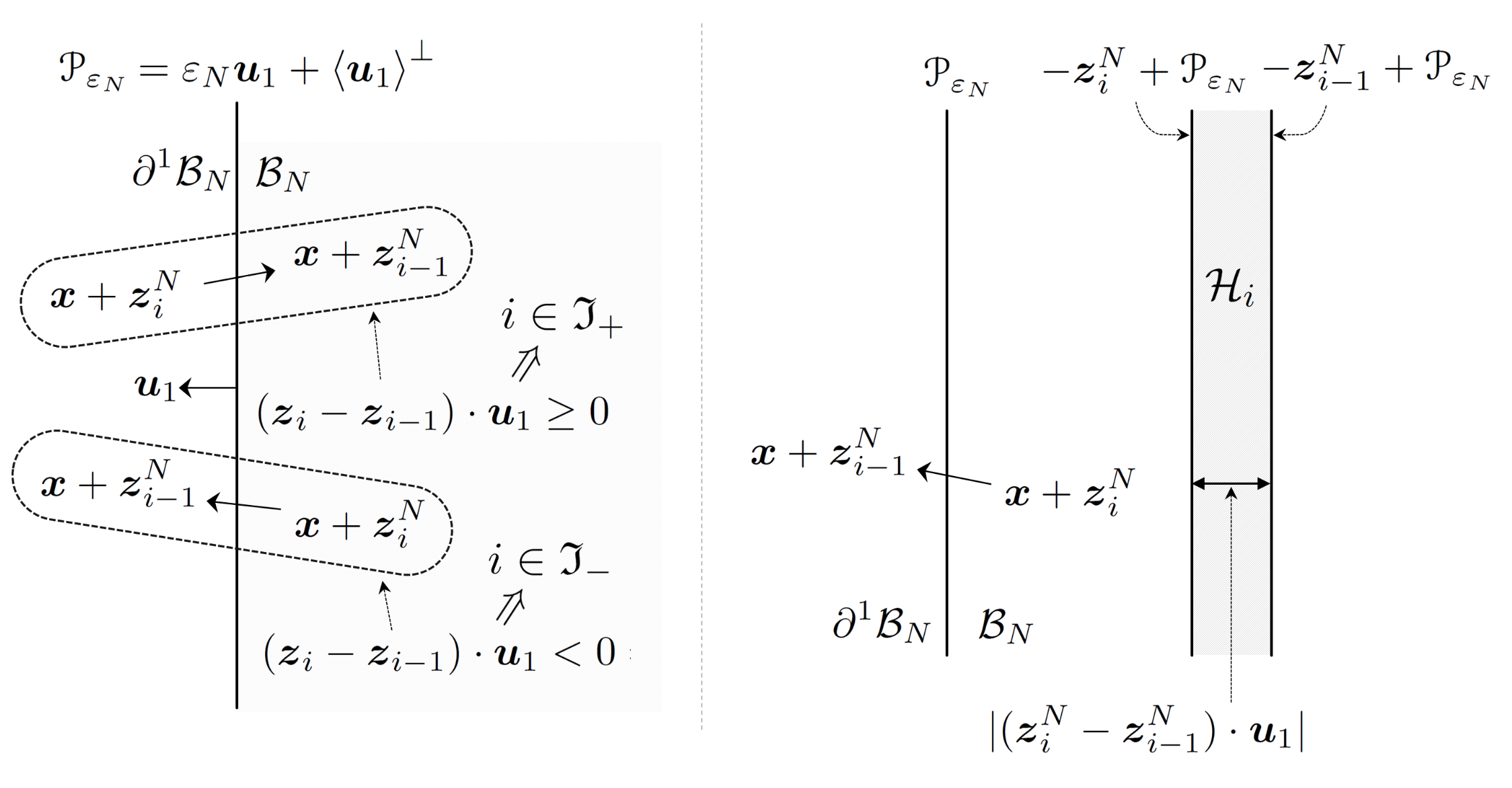}
\protect\caption{\label{fig:divergence}
The sets $\mathcal{C}_{i},\,\mathfrak{I}_{+}$
and $\mathfrak{I}_{-}$} 
\end{figure}

It follows from the previous claim that we may replace in \eqref{f03}
the conditions $\bs x+\bs z^N_j \notin\partial^{1}\mathcal{B}_{N}$ by
the condition $\bs x+\bs z^N_j \in\mathcal{B}_{N}$.  Let
\begin{equation*}
\mathfrak{I}_{+} \;=\;
\{i:(\boldsymbol{z}_{i}-\boldsymbol{z}_{i-1})\cdot \bs u_{1} >
0\}\;,\quad
\mathfrak{I}_{-} \;=\;
\{i:(\boldsymbol{z}_{i}-\boldsymbol{z}_{i-1})\cdot \bs u_{1}<0\}\;.
\end{equation*}
If $\boldsymbol{x}+ \boldsymbol{z}_{i}^{N}
\in\partial^{1}\mathcal{B}_{N}$ and $\boldsymbol{x}+
\boldsymbol{z}_{i-1}^{N} \in\mathcal{B}_{N}$, then
$(\boldsymbol{x}+\boldsymbol{z}_{i}^{N}) \cdot \bs u_1 >
\varepsilon_N$ and $(\boldsymbol{x}+\boldsymbol{z}_{i-1}^{N}) \cdot
\bs u_1 \le \varepsilon_N$ so that $i\in \mf I_+$. The first sum in
\eqref{f03} may be restricted to indices $i$ in $\mf
I_+$. Analogously, the second sum in \eqref{f03} may be restricted to
indices $i$ in $\mf I_-$. Therefore, in view of the explicit
expression \eqref{06b} of the function $V_N$, we can rewrite
(\ref{eq:reat}) as
\begin{equation*}
[1+o_{N}(1)] \, 
\sqrt{\frac{\alpha N}{2\pi}} \, \frac 1{N\, Z_N} 
\sum_{i\in \mf I_+ \cup \mf I_-}
\mf s(i)\, (\boldsymbol{z}_{i}\cdot\boldsymbol{v})
\sum_{\boldsymbol{x}\in\mathcal{H}_{i}} e^{-N F_{N}(\boldsymbol{x})}
e^{-\frac{1}{2}\alpha N(\boldsymbol{x}
\cdot\boldsymbol{v})^{2}}\;,
\end{equation*}
where $\mf s(i) = 1$ if $i\in\mf I_+$ and $\mf s(i) = -1$ if $i\in\mf
I_-$. In this formula, the set $\mc H_i$, $i\in \mf I_+ \cup \mf I_-$,
is given by
\begin{equation*}
\mathcal{H}_{i} \; =\; \begin{cases}
\{\boldsymbol{x}\in\mathring{\mathcal{B}}_{N} \cap \mc A_N:\,
\boldsymbol{x}+\boldsymbol{z}_{i-1}^{N}\in
\mathcal{B}_{N},\,\boldsymbol{x}+\boldsymbol{z}_{i}^{N}\in
\partial^{1}\mathcal{B}_{N}\}\, & \mbox{if }i\in\mathfrak{I}_{+}\;, \\
\{\boldsymbol{x}\in\mathring{\mathcal{B}}_{N} \cap \mc A_N:\,\boldsymbol{x}+
\boldsymbol{z}_{i-1}^{N}\in\partial^{1}\mathcal{B}_{N},\,
\boldsymbol{x}+\boldsymbol{z}_{i}^{N}\in\mathcal{B}_{N}\}\, 
& \mbox{if }i\in\mathfrak{I}_{-}\;.
\end{cases}
\end{equation*}
By Taylor expansion, the previous sum is equal to
\begin{equation*}
[1+o_{N}(1)] \, \frac{\sqrt{\alpha}\, e^{-G(\bs{0})}}{(2\pi
  N)^{(d-1)/2}} \,\kappa_N
\sum_{i\in \mf I_+ \cup \mf I_-} \mf s(i)\, 
(\boldsymbol{z}_{i}\cdot\boldsymbol{v})
\sum_{\boldsymbol{x}\in\mathcal{H}_{i}} 
e^{- (N/2) \boldsymbol{x}^{\dagger}(\mathbb{H}+\alpha vv^{\dagger})\boldsymbol{x}}\;. 
\end{equation*}

Fix $i\in \mf I_+$, the argument for $i\in \mf I_-$ being
analogous. We claim that the set $\mc H_i$ can be rewritten as
\begin{equation*}
\mathcal{H}_{i} \; =\; 
\{\boldsymbol{x}\in\mathring{\mathcal{B}}_{N} \cap \mc A_N:\,
\varepsilon_N - \boldsymbol{z}_{i}^{N} \cdot \bs u_1  
< \boldsymbol{x} \cdot \bs u_1 \le \varepsilon_N -
\boldsymbol{z}_{i-1}^{N} \cdot \bs u_1 \}\;.
\end{equation*}
Since all points $\bs y$ in $\mc B_N$ are such that $\bs y \cdot \bs
u_1 \le \varepsilon_N$ and since all points $\bs w$ in
$\partial^{1}\mc B_N$ are such that $\bs w \cdot \bs u_1 >
\varepsilon_N$, the set $\mc H_i$ is clearly contained in the set
appearing on the right hand side of the previous equality.  On the
other hand, if a point $\bs y$ belongs to this latter set, $\bs y
\cdot \bs u_1 = \varepsilon_N + O(1/N)$. Hence, since $\bs y$ belongs
to $\mc A_N$, $F(\bs y) -H \le [1+o_N(1)] (1/8) \lambda_1
\varepsilon^2_N$. A point $\bs y$ in $\overline{\mc B}_N \cap \mc A_N$
satisfying these two properties is in $\mc B_N$ if $\bs y \cdot\bs u_1
\le \varepsilon_N$ and is in $\partial^{1}\mc B_N$ otherwise. This
proves the claim

Finally, we claim that $\mathcal{H}_{i} = \widehat{\mathcal{H}}_{i}$,
where $\widehat{\mathcal{H}}_{i}$ is the set given by
\begin{equation*}
\widehat{\mathcal{H}}_{i} \; =\; \{\boldsymbol{x}\in \mc A_N:\,
\varepsilon_N - \boldsymbol{z}_{i}^{N} \cdot \bs u_1  
< \boldsymbol{x} \cdot \bs u_1 \le \varepsilon_N -
\boldsymbol{z}_{i-1}^{N} \cdot \bs u_1 \}\;.
\end{equation*}
We have to show that $\widehat{\mathcal{H}}_{i} \subset
\mathcal{H}_{i}$, equivalently, that any point in
$\widehat{\mathcal{H}}_{i}$ belongs to $\mathring{\mc B}_N$. Fix $\bs
y \in \widehat{\mathcal{H}}_{i}$. Since $\bs y\in\mc A_N$, $\sum_{2\le
  j\le d} \lambda_j (\bs y \cdot \bs u_j)^2 \le (5/4) \lambda_1
\varepsilon^2_N$ and, by definition of $\widehat{\mathcal{H}}_{i}$,
$(\bs y + \boldsymbol{z}_{i-1}^{N}) \cdot \bs u_1 \le
\varepsilon_N$. These two conditions imply that $\bs y +
\boldsymbol{z}_{i-1}^{N} \in \mc B_N$ so that $\bs y\in
\mathring{\mathcal{B}}_{N}$ as claimed.

Recall the definition of the hyperplanes $\ms P_t$ introduced above
equation \eqref{35}.  The set $\widehat{\mathcal{H}}_{i}$ consists of
the points in $\mc A_{N}$ which lies between the hyperplanes
$-\boldsymbol{z}_{i}^{N} + \ms P_{\varepsilon_{N}}$ and
$-\boldsymbol{z}_{i-1}^{N} + \ms P_{\varepsilon_{N}}$ (cf. Figure
\ref{fig:divergence}):
\begin{equation*}
\widehat{\mc H}_i \;=\; \mc A_N \,\cap\, \ms P_{\varepsilon_N - \bs
  z^N_i \cdot \bs u_1}(\bs z_i^N -\bs z_{i-1}^N)\;.
\end{equation*}
Therefore, by Lemma \ref{s14}, for $i\in \mf I_+$,
\begin{align*}
&\frac{1}{(2\pi N)^{(d-1)/2}} \sum_{\boldsymbol{x}\in\mathcal{H}_{i}}
e^{-(N/2) \boldsymbol{x}^{\dagger}(\mathbb{H}+ \alpha\boldsymbol{v} 
\boldsymbol{v}^{\dagger})\boldsymbol{x}} \\
&\qquad \;=\; 
\big[1+o_{N}(1)\big] \,
\frac{(\boldsymbol{z}_{i}-\boldsymbol{z}_{i-1})
\cdot\boldsymbol{u}_{1}}{v_{1}} \, \frac 1{\sqrt{\alpha}} 
\, \sqrt{\frac{\lambda_{1}}{\prod_{k=2}^{d}\lambda_{k}}}\;\cdot
\end{align*}

Repeating the same argument for $i\in\mf I_-$, and since
$(\boldsymbol{z}_{i}-\boldsymbol{z}_{i-1})\cdot\boldsymbol{u}_{1} = 0$
for $i\not\in \mf I_- \cup \mf I_+$, we obtain that \eqref{eq:reat} is
equal to
\begin{equation*}
\big[1+o_{N}(1) \big]\, \frac{\kappa_N e^{-G(\bs{0})}}{v_{1}} \,
\sqrt{\frac{\lambda_{1}}{\prod_{k=2}^{d}\lambda_{k}}} \,
\sum_{i=1}^{L}(\boldsymbol{z}_{i}\cdot\boldsymbol{v}) \,
(\boldsymbol{z}_{i}-\boldsymbol{z}_{i-1})\cdot\boldsymbol{u}_{1}\;. 
\end{equation*}
By the definition \eqref{36} of the matrix $\bb A$ and since $\bs u_1$
(resp. $\bs v$) is an eigenvector of $\bb H$ (resp. $\bb H \bb
A^{\dagger}$) with eigenvalue $-\lambda_1$ (resp. $-\mu$), the last
sum can be written as
\begin{equation*}
\sum_{i=1}^{L} \boldsymbol{u}_{1}^{\dagger}
(\boldsymbol{z}_{i}-\boldsymbol{z}_{i-1})\, 
\boldsymbol{z}_{i}^{\dagger}\boldsymbol{v}
\;=\; \boldsymbol{u}_{1}^{\dagger}\mathbb{A}^{\dagger}\boldsymbol{v}
\;=\; -\, \frac{1}{\lambda_{1}}\,
\boldsymbol{u}_{1}^{\dagger}\mathbb{H}
\mathbb{A}^{\dagger}\boldsymbol{v}
\;=\;\frac{\mu}{\lambda_{1}}\boldsymbol{u}_{1}^{\dagger}
\boldsymbol{v}\;=\;\frac{\mu}{\lambda_{1}}\, v_1 \;.
\end{equation*}
Hence, the penultimate formula becomes 
\begin{equation*}
\big[1+o_{N}(1) \big]\, \kappa_N e^{-G(\bs{0})}
\frac{\mu}{\sqrt{\prod_{k=1}^{d}\lambda_{k}}}
\;=\; \big[1+o_N (1)\big] \,\kappa_N\,\omega_{\bs{0}}\;,
\end{equation*}
which completes the proof of the proposition. 
\end{proof}

The divergences along the boundary $\partial^{0}\mathcal{B}_{N}$ is
negligible.

\begin{lemma}
\label{s11}
We have that $\sum_{\boldsymbol{x}\in \partial^0 \mathcal{B}_{N}} \big
|\, (\text{\rm div } \Phi_{N}) (\boldsymbol{x})\, \big | =\kappa_N \,
o_N(1)$.
\end{lemma}

\begin{proof}
In view of \eqref{37} and since $F_N(\bs x) \ge H+(1/4) \lambda_1
\varepsilon^2_N$ on $\partial^0 \mathcal{B}_{N}$, on this set we have
that
\begin{equation*}
| (\text{\rm div } \Phi_{N}) (\boldsymbol{x}) | \le C_0 Z_N^{-1}e^{-NH}
\exp\{- (1/4) \lambda_1 N \varepsilon^2_N\}\;.
\end{equation*}
To complete the proof of
the proposition it remains to recall (\ref{vare2}).
\end{proof}

\subsection{Divergence-free flow.} 
\label{sub52}

In this subsection, we transfer the divergence of the flow $\Phi_{N}$
to $\bs{m}_1^N$ and $\bs{m}_2^N$.

\begin{definition}
\label{def1}
A sequence $\{\mathfrak{f}_N : N\ge 1\}$,
$\mathfrak{f}_N\in\mathcal{F}_N$, of flows is said to be negligible if
$\Vert\mathfrak{f}_N\Vert^2 \,=\, \kappa_N\, o_N(1)$.
\end{definition}

The following proposition is a weaker version of Theorem \ref{s24}.

\begin{proposition}
\label{s15}
There exists a flow $\widetilde{\Phi}_{N}$ which is divergence-free on
$\{\bs{m}_1^N, \bs{m}_2^N\}^{c}$, and such that
\begin{equation}
\label{eq: div of crr}
\big(\textup{div}\,
  \widetilde{\Phi}_{N}\big)
(\bs{m}_1^N) \;=\; -\big(\textup{div}\,
  \widetilde{\Phi}_{N}\big)
(\bs{m}_2^N)\;=\; \big[ 1+o_{N}(1)\big] \,\kappa_N \, \omega_{\bs{0}}\;.
\end{equation}
Moreover, $\widetilde{\Phi}_{N} - \Phi_{N}$ is negligible. 
\end{proposition}

The proof of this proposition relies on the displacement of the
divergence of a flow along \textit{good paths}.  Fix a constant
$C_{0}>0$. A path $\Gamma= (\boldsymbol{x}_{0}, \boldsymbol{x}_{1},
\dots, \boldsymbol{x}_{M}) \subset\Xi_{N}$ with no self-intersections
is called a good path connecting $\boldsymbol{x}_{0}$ to
$\boldsymbol{x}_{M}$ if
\begin{enumerate}
\item[{\bf (P1)}] Each edge $(\bs{x}_i, \bs{x}_{i+1})$, $0\le i<M$,
  belongs to $E_N$, where $E_N$ has been introduced in \eqref{00E}.
\item[{\bf (P2)}] $F_{N}(\boldsymbol{x}_{k}) \le
  F_{N}(\boldsymbol{x}_{0}) + C_{0}/N$ for all $0\le k\le M$.
\end{enumerate}

Let $\Gamma= (\boldsymbol{x}_{0}, \boldsymbol{x}_{1}, \dots,
\boldsymbol{x}_{M})$ be a good path and let $r\in\bb R$. 
Denote by $\chi_{\Gamma, r}$ the flow defined by
\begin{equation}
\label{eq:chi_Ga}
\chi_{\Gamma,r}(\boldsymbol{x}_{k},\boldsymbol{x}_{k+1}) \;=\; r
\;=\; -\, \chi_{\Gamma,r}(\boldsymbol{x}_{k+1},\boldsymbol{x}_{k})\;,
\quad 0\le k < M \;,
\end{equation}
and $\chi_{\Gamma,r}(\boldsymbol{x},\boldsymbol{y})=0$, otherwise.
The divergence of the flow $\chi_{\Gamma,r}$ vanishes at all points
except at $\bs x_0$ and at $\bs x_M$, where $(\text{div }
\chi_{\Gamma,r})(\bs x_0) = r = - (\text{div } \chi_{\Gamma,r})(\bs
x_M)$. Therefore, by linearity, if $\psi$ is a flow, the divergence of
the flow $\psi + \chi_{\Gamma,r}$ coincides with the one of $\psi$ at
all points except at $\bs x_0$ and $\bs x_M$, where it is modified by
$\pm r$.

\begin{lemma}
\label{s16}
Consider two disjoint subsets $\mc{A}$, $\mc{B}$ of $\Xi_{N}$, and a flow
$\psi\in\mathcal{F}_{N}$. Suppose that for each $\boldsymbol{a}\in \mc A$
there exists a good path $\Gamma_{\boldsymbol{a}}$ connecting
$\boldsymbol{a}$ to a point in $\mc B$. Denote by $L_{N}$ the maximal
length of these good paths, and assume that each edge of $E_{N}$ is an
edge of at most $M_{N}$ of these paths. Then, there exists a flow,
denoted by $\chi_{\mc A}$, which is divergence-free on $(\mc A\cup \mc B)^{c}$,
and such that
\begin{align*}
& (\textup{div }\chi_{\mc A})(\boldsymbol{a}) \;=\; 
-\, (\textup{div }\psi)(\boldsymbol{a})\;, \quad  \bs{a}\in \mc A\;, \\
&\quad |(\textup{div }\chi_{A})(\boldsymbol{b})| \;\le\; 
\sum_{\boldsymbol{a}\in \mc A}|(\textup{div }\psi)(\boldsymbol{a})|\;,
\quad \bs{b}\in \mc B\;.
\end{align*}
In particular, the flow $\psi+\chi_{\mc{A}}$ is divergence-free on $\mc{A}$.
Moreover, there exists a finite constant $C$, independent of $N$, such that
\begin{equation*}
\left\Vert \chi_{\mc{A}}\right\Vert ^{2} \;\le\; C L_{N}
M_{N}Z_{N}\sum_{\boldsymbol{a}\in \mc A}e^{NF_{N}(\boldsymbol{a})}
\left[(\textup{div }\psi)(\boldsymbol{a})\right]^{2}\;.
\end{equation*}
\end{lemma}

\begin{proof}
For each $\boldsymbol{a}\in \mc{A}$, let $\chi_{\boldsymbol{a}}$ be the
flow $\chi_{\Gamma,r}$ constructed in \eqref{eq:chi_Ga}, where
$\Gamma$ is a good path which connects $\boldsymbol{a}$ to a point in
$\mc{B}$, and $r= - (\text{div } \psi) (\bs a)$. Let $\chi_\mc{A} = \sum_{\bs
  a\in \mc{A}} \chi_{\bs a}$. The assertions concerning the divergence of
$\chi_\mc{A}$ follows from the definition of the flows $\chi_{\Gamma,r}$.

We turn to the proof of the last assertion of the lemma. Fix $\bs a\in
\mc{A}$, and let $ (\boldsymbol{a}=\boldsymbol{x}_{0}, \boldsymbol{x}_{1},
\dots,\boldsymbol{x}_{k} =\boldsymbol{b})$ be the good path which
connects $\bs a$ to $\mc{B}$. By property {\bf{(P2)}} of good paths
and by the fact that $k\le L_{N}$,
\begin{align*}
\left\Vert \chi_{\boldsymbol{a}}\right\Vert ^{2}  \;&\le\;
C\sum_{i=0}^{k-1} Z_{N }e^{NF_{N}(\boldsymbol{x}_{i})}
\left[(\textup{div }\psi)(\boldsymbol{a})\right]^{2}\nonumber \\
& \le\;  C \, L_{N} \, Z_{N} \, e^{NF_{N}(\boldsymbol{a})}
\left[(\textup{div }\psi)(\boldsymbol{a})\right]^{2}\;.
\end{align*}
Assume that an edge $(\boldsymbol{x}, \boldsymbol{y})\in E_{N}$ is
used by the good paths $\Gamma_{\boldsymbol{a}_{1}},
\Gamma_{\boldsymbol{a}_{2}}, \dots,\Gamma_{\boldsymbol{a}_{m}}$. By
the Cauchy-Schwarz inequality and since $m\le M_{N}$,
\begin{equation*}
\big[\chi_{\mc{A}}(\boldsymbol{x},\boldsymbol{y}) \big]^{2} \;=\;  
\Big( \sum_{i=1}^{m} \chi_{\boldsymbol{a}_{i}}
(\boldsymbol{x},\boldsymbol{y}) \Big)^{2}
\;\le\;M_{N}\sum_{i=1}^{m} \big[\chi_{\boldsymbol{a}_{i}} 
(\boldsymbol{x},\boldsymbol{y})\big]^{2}
\;=\;M_{N}\sum_{\bs a\in \mc{A}} \big[\chi_{\boldsymbol{a}} 
(\boldsymbol{x},\boldsymbol{y})\big]^{2}\;.
\end{equation*}
By dividing this inequality by $c^s(\bs x, \bs y)$ and summing over
all edges, we obtain that
\begin{equation*}
\left\Vert \chi_{\mc{A}}\right\Vert ^{2} \;\le\; 
M_{N}\sum_{\boldsymbol{a}\in \mc{A}}\left\Vert \chi_{\boldsymbol{a}}\right\Vert ^{2}\;.
\end{equation*}
Putting together the two previous estimates, we complete the proof of
the second assertion of the lemma.
\end{proof}

\begin{figure}
\centering
\includegraphics[scale=0.23]{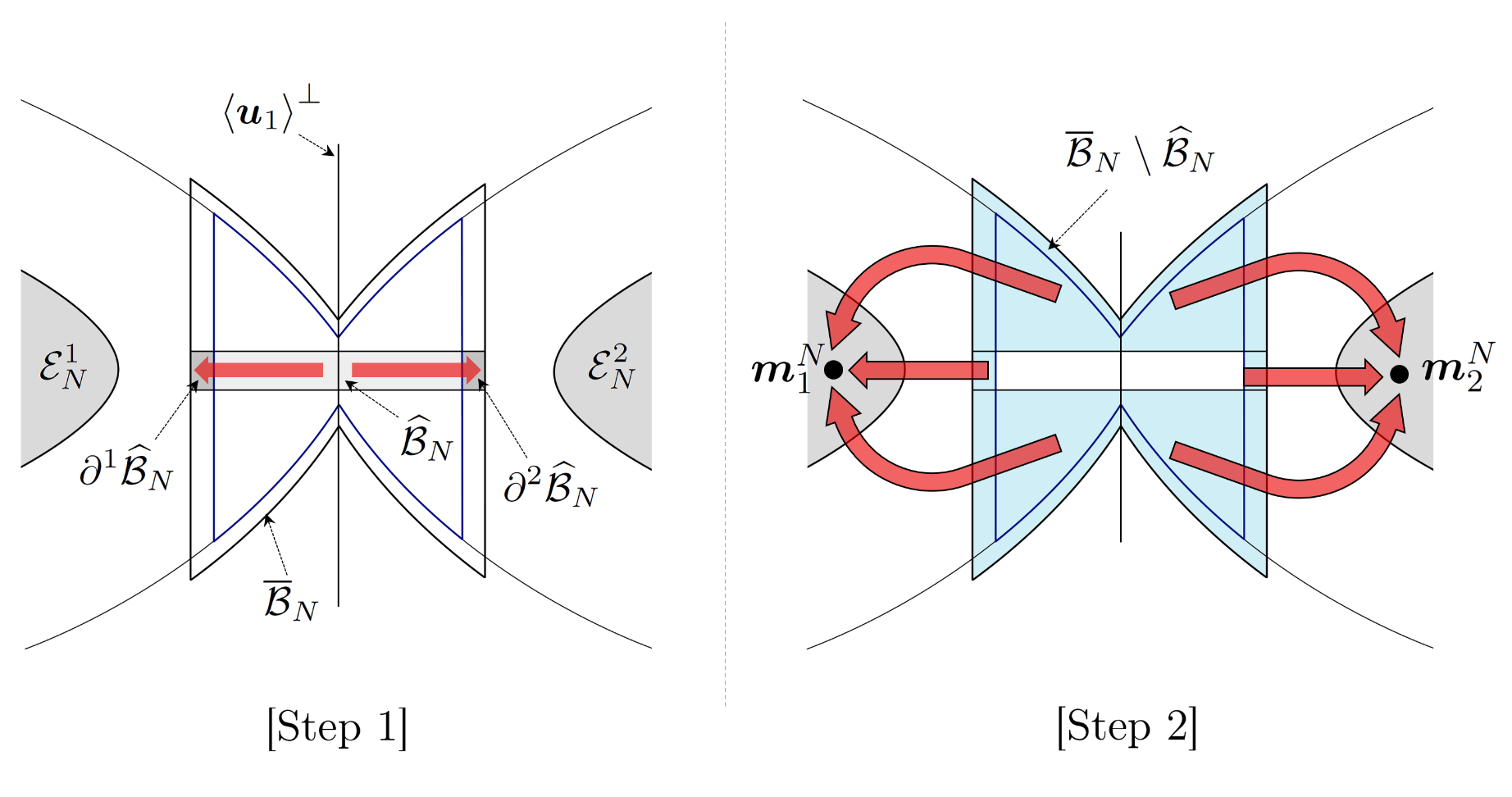}\protect
\caption{\label{fig:cleaning} Transferring the divergence of the flow
  $\Phi_N$ from $\mathcal{\overline{B}}_{N}$ to $\mc E^1_N \cup \mc E^2_N$}
\end{figure} 

\begin{proof}[Proof of Proposition \ref{s15}]
Let
\begin{equation}
\label{eq: hat CN}
\widehat{\mc C}_{N} \;=\; \Big\{ \boldsymbol{x}\in\Xi_{N}:
\,|\boldsymbol{x}\cdot\boldsymbol{u}_{k}|
<\sqrt{\frac{\lambda_{1}}{2(d-1)\lambda_{k}}}
\varepsilon_{N}\,\,\mbox{for all\,\,}2\le k\le d
\Big\} 
\end{equation}
and let $\widehat{\mathcal{B}}_{N}= \widehat{\mc C}_{N}
\cap\mathcal{B}_{N}$, $\partial^{i}\widehat{\mathcal{B}}_{N}
=\widehat{\mc C}_{N} \cap \partial^{i}\mathcal{B}_{N}$, $i=1$, $2$.

The proof is divided in two steps. We first push the divergence of
$\widehat{\mathcal{B}}_{N}$ into $\partial^{1}
\widehat{\mathcal{B}}_{N} \cup\partial^{2} \widehat{\mathcal{B}}_{N}$.
Then, we send all the divergences in $\overline{\mathcal{B}}_{N}
\setminus\widehat{\mathcal{B}}_{N}$ to minima $\bs m_1^N$ and $\bs
m_2^N$ of $\mc E_N^1$ and $\mc E_N^2$, respectively. This procedure is
visualized in Figure \ref{fig:cleaning}.

\noindent\textit{Step 1. Transfer of the divergences of $\widehat{\mathcal{B}}_{N}$
into $\partial^{1} \widehat{\mathcal{B}}_{N} \cup \partial^{2}
\widehat{\mathcal{B}}_{N}$.}  
We start by introducing the good paths connecting points in
$\widehat{\mathcal{B}}_{N}$ to points in $\partial^i
\widehat{\mathcal{B}}_{N}$, $i=1,\,2$.

Fix $\boldsymbol{x}\in \widehat{\mathcal{B}}_{N}$ and assume, without
loss of generality, that $\boldsymbol{x}\cdot\boldsymbol{u}_{1}\ge 0$.
Consider the line $l_{\boldsymbol{x}}= \{\boldsymbol{x}+
t\boldsymbol{u}_{1}: 0\le t\le\varepsilon_{N}\}$, and let
$\Gamma_{\boldsymbol{x}}= (\boldsymbol{x}= \boldsymbol{x}_{0},
\boldsymbol{x}_{1}, \dots, \boldsymbol{x}_{m}\in\partial^{1}
\widehat{\mathcal{B}}_{N})$ be a path such that
\begin{enumerate}
\item[(a)] The edge $(\boldsymbol{x}_{i},\boldsymbol{x}_{i+1})$
  belongs to $E_N$ for all $0\le i<m$;
\item[(b)] $\textup{dist}(\boldsymbol{x}_{i},\,l_{\boldsymbol{x}})\le C_0/N$ for
  all $0\le i<m$; 
\item[(c)] $m\le C_0 N \varepsilon_{N}$ 
\end{enumerate}
for some finite constant $C_0$ independent of $N$.  Since
$\boldsymbol{u_{1}}$ is a decreasing direction of $F$, there exists a
finite constant $C_0$ such that $F_{N}(\boldsymbol{y})\le
F_{N}(\boldsymbol{x})+C_0/N$ for all $\boldsymbol{y}\in
l_{\boldsymbol{x}}$. Hence, in view of (a) and (b) ,
$\Gamma_{\boldsymbol{x}}$ is a good path for all
$\bs{x}\in\widehat{\mc{B}}_N$.

Denote by $\chi_N$ the flow $\chi_{\mc{A}}$ given by Lemma \ref{s16}
for $\mc A:=\widehat{\mc{B}}_{N}$, $\mc B:=\partial^{1}
\widehat{\mc{B}}_{N} \cup \partial^{2} \widehat{\mc{B}}_{N}$,
$\psi=\Phi_{N}$, and good paths
$\{\Gamma_{\bs{x}}:\bs{x}\in\widehat{\mc{B}}_N\}$.  By construction,
$L_{N} \le C_0 N\varepsilon_{N}$. We can observe from (b) that a path
$\Gamma_{\boldsymbol{x}}$,
$\boldsymbol{x}\in\widehat{\mathcal{B}}_{N}$, visit $\boldsymbol{y}$
only if $\textup{dist}(\boldsymbol{y},\,l_{\boldsymbol{x}})\le C_{0}/N$.  From
this observation, it is clear that there are at most
$C_{0}N\varepsilon_{N}$ paths visiting $\boldsymbol{y}$. This implies
$M_{N}\le C_{0}N\varepsilon_{N}$. Therefore, by Lemma \ref{s16},
$\widehat{\Phi}_{N} = \Phi_{N} + \chi_{N}$ is divergence-free on
$\widehat{\mc{B}}_{N}$ and
\begin{equation*}
\big\Vert \chi_{N}\big\Vert ^{2} \;\le\; C_0 \, N^{2} 
\, \varepsilon_{N}^{2} \, Z_{N} \sum_{\boldsymbol{x}\in\widehat{\mc{B}}_{N}} 
e^{NF_{N}(\boldsymbol{x})} \big[(\textup{div
}\Phi_{N})(\boldsymbol{x})\big]^{2}\;. 
\end{equation*}
By \eqref{31}, by a Taylor expansion of $F_{N}(\boldsymbol{x})$, and
by Lemma \ref{s25}, this expression is less than or equal
to
\begin{equation*}
C_0 Z_{N}^{-1} e^{-NH} N  \varepsilon_{N}^{6} 
\sum_{\boldsymbol{x}\in\mc{B}_{N}}
e^{-\frac{N}{2}\boldsymbol{x}^{\dagger}(\mathbb{H}+
2\alpha\boldsymbol{v}\boldsymbol{v}^{\dagger})\boldsymbol{x}}
\;\le\;  C_0 Z_{N}^{-1} e^{-NH} N^{\frac{d}{2}+1}\, \varepsilon_{N}^{6}\;.
\end{equation*}
Therefore, by (\ref{vare1}), 
\begin{equation}
\label{eq: contro1}
\left\Vert \chi_{N}\right\Vert^{2} 
\;\le\; C_0 \kappa_N N^{2}\, \varepsilon_{N}^{6}
\;=\;  \kappa_N  o_{N}(1)\;,
\end{equation}
and hence $\chi_N$ is a negligible flow.

\noindent\textit{Step 2. Transfer of the divergences of
  $\overline{\mc{B}}_{N} \setminus \widehat{\mc{B}}_{N}$ to local minima of 
  $W_1$ and $W_2$.}
Fix, without loss of generality, $\boldsymbol{x} \in
\overline{\mc{B}}_{N} \setminus\widehat{\mc{B}}_{N}$ such that
$\boldsymbol{x}\cdot \bs u_{1}\ge 0$. 

We first claim that there exists $a>0$ such that
\begin{equation}
\label{eq: est4}
e^{N[F_{N}(\boldsymbol{x})+H]}\big[ (\textup{div
}\Phi_{N})(\boldsymbol{x}) \big]^{2} \; \le \; 
\frac 1{Z_{N}^2}\, e^{-a\varepsilon_{N}^{2}N}
\end{equation}
for all $\boldsymbol{x}\in \overline{\mc{B}}_{N} \setminus
\widehat{\mc{B}}_{N}$.  We consider three cases,
$\boldsymbol{x}\in\mc{B}_{N}\setminus \widehat{\mc{B}}_{N}$,
$\boldsymbol{x} \in\partial^{0} \mc{B}_{N}$ and $\boldsymbol{x}
\in\partial^{1} \mc{B}_{N}$.  

Assume first that $\boldsymbol{x}\in
\mc{B}_{N} \setminus \widehat{\mc{B}}_{N} = \mc{B}_{N} \setminus
\widehat{\mc C}_N$ and recall from Lemma \ref{s17} that
$\mathbb{H}+2\alpha \bs{v}\bs{v}^{\dagger}$ is a positive-definite matrix. By
\eqref{31}, and by a Taylor expansion of $F_{N}(\boldsymbol{x})$,
\begin{equation*}
e^{N[F_{N}(\boldsymbol{x})+H]} \big[(\textup{div }\Phi_{N}) (\boldsymbol{x})\big]^{2} 
\;\le\; C_0 \frac{\varepsilon_{N}^{4} }{N\, Z_{N}^{2}} \,
e^{-\frac{N}{2}\boldsymbol{x}^{\dagger}(\mathbb{H}+2\alpha
\boldsymbol{v}\boldsymbol{v}^{\dagger})\boldsymbol{x}}
\;\le\; C_0\, \frac 1{Z_{N}^2} \,
e^{-\frac{N}{2}\tau|\boldsymbol{x}|^{2}} \;,
\end{equation*}
where $\tau>0$ is the smallest eigenvalue of $\mathbb{H}+2\alpha
vv^{\dagger}$. To complete the proof of (\ref{eq: est4}) for points
$\boldsymbol{x}$ in the set $\mc{B}_{N} \setminus
\widehat{\mc{B}}_{N}$, it remains to observe from (\ref{eq: hat CN})
that for $\bs x$ in $\mc{B}_{N} \setminus \widehat{\mc{B}}_{N}$,
\begin{equation*}
|\boldsymbol{x}|^{2} \;\ge\; \varepsilon_{N}^{2}
\min_{2\le k\le d}\Big\{ \frac{\lambda_{1}}{2(d-1)\lambda_{k}}\Big\}
\;. 
\end{equation*}

Fix $\bs x\in\partial^{0}\mc{B}_{N}$. By \eqref{37}, $|(\textup{div
}\Phi_{N}) (\boldsymbol{x})| \le C_0 (1/Z_N) \exp\{-N F_N(\bs x)\}$.
In particular, (\ref{eq: est4}) follows from the fact that
$F(\boldsymbol{x}) \ge H+ (1/4) \lambda_{1} \varepsilon_{N}^{2}$ for $\bs
x$ in $\partial^{0}\mc{B}_{N}$.

Finally, fix $\bs x \in \partial^{1} \mc{B}_{N}$. By \eqref{eq:div(Phi_N)},
\begin{equation*}
(\textup{div }\Phi_{N})(\boldsymbol{x}) \;=\; \frac 1{Z_N} \sum_{i}
e^{-N\overline{F}_{N}(\boldsymbol{x}-\boldsymbol{z}_{i}^{N})} \big\{ 
V_{N}(\boldsymbol{x})- V_{N}(\boldsymbol{x}+\boldsymbol{z}_{i+1}^{N} -
\boldsymbol{z}_{i}^{N})\big\}\;,
\end{equation*}
where the sum is performed over all indices $i$ for which
$\boldsymbol{x} - \boldsymbol{z}_{i}^{N}$ belongs to
$\mathring{\mc{B}}_{N}$. Expressing the difference
$V_{N}(\boldsymbol{x})- V_{N}(\boldsymbol{x}+\boldsymbol{z}_{i+1}^{N}
- \boldsymbol{z}_{i}^{N})$ as $[1-V_{N}(\boldsymbol{x})] +
[V_{N}(\boldsymbol{x}+\boldsymbol{z}_{i+1}^{N} -
\boldsymbol{z}_{i}^{N})-1]$, by Young's inequality,
\begin{equation*}
\big[ (\textup{div }\Phi_{N})(\boldsymbol{x}) \big]^2 \;\le\;
\frac {C_0} {Z^2_N} \sum_{i}
e^{-2N\overline{F}_{N}(\boldsymbol{x}-\boldsymbol{z}_{i}^{N})} \Big\{ 
[1-V_{N}(\boldsymbol{x})]^2 +[1-V_{N}(\boldsymbol{x}+\boldsymbol{z}_{i+1}^{N} -
\boldsymbol{z}_{i}^{N})]^2 \Big\}\;.
\end{equation*}
Since the sum is carried over indices $i$ for which $\boldsymbol{x} -
\boldsymbol{z}_{i}^{N}$ belongs to $\mathring{\mc{B}}_{N}$, and since
$\exp\{-N\overline{F}_{N} (\boldsymbol{x} - \boldsymbol{z}_{i}^{N})\}
\le C_0 \exp\{-NF_N (\boldsymbol{x})\}$, by Lemma \ref{s18}, 
\begin{equation*}
\big[ (\textup{div }\Phi_{N})(\boldsymbol{x}) \big]^2 \;\le\;
\frac {C_0} {Z^2_N} \, e^{-N[ F_{N}(\boldsymbol{x})+H]} \, e^{-a\varepsilon_{N}^{2}N}\;, 
\end{equation*}
for some $a>0$, which completes the proof of assertion \eqref{eq:
  est4}.
  
We next claim that there exists $a>0$ such that
\begin{equation}
\label{eq: est5}
e^{N[F_{N}(\boldsymbol{x})+H]} \big[ (\textup{div
}\chi_{N})(\boldsymbol{x})\big]^{2} \;\le\; C_0\, 
\frac 1{Z_{N}^2}\, e^{-a\varepsilon_{N}^{2}N}
\end{equation}
for all $\boldsymbol{x}\in \partial^1 \widehat{\mc{B}}_{N}
\cup \partial^2 \widehat{\mc{B}}_{N}$.

Fix a point $\boldsymbol{x} \in \partial^1 \widehat{\mc{B}}_{N}
\cup \partial^2 \widehat{\mc{B}}_{N}$. By Lemmata \ref{s19} and
\ref{s16},
\begin{equation}
\label{4331}
|(\textup{div }\chi_{N})(\boldsymbol{x})| \;\le\; 
\sum_{z\in\widehat{\mc{B}}_{N}}|(\textup{div
}\Phi_{N})(\boldsymbol{z})|
\;=\;\kappa_N o_{N}(1)\;.
\end{equation}
By (\ref{eq: hat CN}), each point $\boldsymbol{x}$ in $\partial^1
\widehat{\mc{B}}_{N} \cup \partial^2 \widehat{\mc{B}}_{N}$ can be
written as
\begin{equation*}
\boldsymbol{x}  \;=\; \Big( \pm\big[1+o_N(1)\big]\boldsymbol{u}_{1}
+ \sum_{k=2}^{d}x_{k}\boldsymbol{u}_{k}\Big)\, \varepsilon_{N}\;,
\end{equation*}
where $x_{k}^2 \le \lambda_{1}/[2(d-1)\lambda_{k}]$, $2\le k\le d$.
Therefore, by Taylor's expansion,
\begin{align*}
F_{N}(\boldsymbol{x}) -H \; &=\; \frac 12 \, \boldsymbol{x}^{\dagger}
\mathbb{H}\boldsymbol{x} \,+\, O(\varepsilon^3_N)
\;=\; \frac 12 \, \Big(-\lambda_{1} + 
\sum_{k=2}^{d}x_{k}^{2}\, \lambda_{k}\Big)
\varepsilon_{N}^{2}\,+\,O(\varepsilon^3_N) \\
&\le \; \frac 12 \, \Big( -\lambda_{1} + \sum_{k=2}^{d}
\frac{\lambda_{1}}{2(d-1)}\,\Big) \, \varepsilon_{N}^{2} \,+\, O(\varepsilon^3_N)
\; =\; -\,  \frac{\lambda_{1}}{4}\varepsilon_{N}^{2}
\,+\, O(\varepsilon_{N}^3)\;.  
\end{align*}
By the previous estimates, since $N\epsilon^3_N\downarrow 0$ and since
$\exp\{-a N\varepsilon_{N}^{2}\}$ vanishes faster than any polynomial,
\begin{equation*}
e^{N[F_{N}(\boldsymbol{x})+H]} \, \big[(\textup{div }\chi_{N})
(\boldsymbol{x})\big]^{2} \;\le\; C_0 \,\kappa_N^2\, 
e^{2NH}\, e^{-(\lambda_{1}/4) N\varepsilon_{N}^{2}} 
\;<\; C_0 \,  \frac 1{Z_{N}^2}\, e^{-(\lambda_{1}/8)
  N\varepsilon_{N}^{2}}\;, 
\end{equation*}
which completes the proof of assertion \eqref{eq: est5}.

By definition of the flow $\widehat{\Phi}_{N}$ and by 
(\ref{eq: est4}) and (\ref{eq: est5}), there exists $a>0$ such that
\begin{equation}
\label{e424}
e^{N[F_{N}(\boldsymbol{x})+H]} \big[ (\textup{div }\widehat{\Phi}_{N})
(\boldsymbol{x})\big]^{2} \;\le\; \frac 1{Z_{N}^2}\, 
e^{-a\varepsilon_{N}^{2}N}\;\;,\;
\forall \bs{x}\in\overline{\mathcal{B}}_N
\setminus\widehat{\mathcal{B}}_N\;.
\end{equation}

We are now in a position to move the divergence of $\widehat{\Phi}_N$
from $\overline{\mc{B}}_{N} \setminus \widehat{\mc{B}}_{N}$ to the 
minima of $W_1$ and $W_2$.

Fix $\bs x\in \overline{\mc{B}}_{N} \setminus \widehat{\mc{B}}_{N}$
and let $\bs x(t)$ be the solution of the ODE 
\begin{equation*}
\dot{\boldsymbol{x}}(t) \;=\; -\nabla F(\boldsymbol{x}(t))\;,\;\; 
\boldsymbol{x}(0)\;=\;\boldsymbol{x}\;.
\end{equation*}
Since we assumed that $\bs x \cdot \bs u_1 \ge 0$, this path connects
$\boldsymbol{x}$ to a local minima of $W_1^{\epsilon}$. Let
$T=\inf\{\,t>0:\bs{x}(t)\in \overline{W}_1^\epsilon\,\}<\infty$, and
let $\theta_\bs{x}^{(1)}=\{\bs{x}(t):0\le t\le T\}$. Since
$\overline{W}_1^\epsilon$ is connected, there is a continuous path
$\theta_\bs{x}^{(2)}\subset \overline{W}_1^\epsilon $ connecting
$\bs{x}(T)$ to $\bs{m}_1$. Note that $F(\bs y) \le H - \epsilon$ for
all $\bs y \in \theta_\bs{x}^{(2)}$. Appending $\theta_\bs{x}^{(2)}$
to $\theta_\bs{x}^{(1)}$ at $\bs{x}(T)$, we obtain a continuous path
$\theta_\bs{x}$ connecting $\bs{x}$ and $\bs{m}_1$. Let
$\Gamma_\bs{x}=(\bs{x}=\bs{x}_0,\,\bs{x}_1,\,\cdots,\,\bs{x}_m=\bs{m}_1^N)$
be the path connecting $\bs{x}$ to $\bs{m}_1^N$, obtained by
discretizing $\theta_\bs{x}$, as in step 1. It is obvious that
$F(\bs{y})\le F(\bs{x})$ for all $\bs{y}\in\theta_\bs{x}$, so that
$\Gamma_{\bs{x}}$ is a good path. 

Apply Lemma \ref{s16} with $\mc
A=\overline{\mc{B}}_{N}\setminus\widehat{\mc{B}}_{N}$, $\mc
B=\{\bs{m}_1^N, \bs{m}_2^N\}$, $\psi=\widehat{\Phi}_{N}$ and good
paths $\{\Gamma_\bs{x}:\bs{x}\in\overline{\mc B}_N \setminus
\widehat{\mc B}_N\}$.  As $M_{N}\le C_0 N^d$ and $L_{N}\le C_0 N^{d}$,
if we represent by $\widetilde{\chi}_{N}$ the flow denoted by
$\chi_{\mc A}$ in Lemma \ref{s16}, by the assertion of this lemma and
by \eqref{e424},
\begin{align*}
&\left\Vert \widetilde{\chi}_{N}\right\Vert ^{2} \;\le\; 
C_0 \, N^{2d}\, Z_{N} \sum_{\boldsymbol{x}\in \overline{\mc{B}}_{N}
\setminus\widehat{\mc{B}}_{N}} e^{NF_{N}(\boldsymbol{x})}
\big[(\textup{div }\widehat{\Phi}_{N})(\boldsymbol{x})\big]^{2} \\
& \quad \le\; C_0 \,e^{-NH} \frac {N^{2d}}{Z_{N}} \, |\overline{\mc{B}}_{N}|\,
e^{-a\varepsilon_{N}^{2}N}\;=\;C_0\, \kappa_N \, N^{(3d/2)+1} \, 
e^{-a\varepsilon_{N}^{2}N}
\;=\;\kappa_N\, o_{N}(1)\;.
\end{align*}
Therefore, $\widetilde{\chi}_{N}$ is negligible. 

Let $R_{N}=\chi_{N}+\widetilde{\chi}_{N}$. By construction,
the flow $\Phi_{N}+R_{N}$ is divergence-free on
$\{\bs{m}_1^N ,\bs{m}_2^N \}^{c}$, and $R_N$ is
negligible since both of $\chi_N$ and $\widetilde{\chi}_N$ are
negligible.

It suffices to check \eqref{eq: div of crr} to complete the proof. By
the first assertion of Lemma \ref{lf1}, and by Lemmata \ref{s19} and
\ref{s11}, the divergence of the flow $\Phi_N$ is negligible on $\Xi_N
\setminus (\partial^{1}\mathcal{B}_{N}
\cup \partial^{2}\mathcal{B}_{N})$.  Since the path
$\Gamma_{\boldsymbol{x}}$ connects $\partial^{i}\mathcal{B}_{N}$ to
$\bs{m}_i^N$, $i=1$, $2$, the divergences on
$\partial^{i}\mathcal{B}_{N}$ is transferred to
$\bs{m}_i^N$. Therefore, by Proposition \ref{s03}, we can conclude
\eqref{eq: div of crr}.

\end{proof}

\subsection{Flows for the adjoint dynamics.} 
\label{sub54}

We introduce in this subsection the flows which approximate the flow
$\Phi_{V^*_N}$ in $\mc B_N$. For $\bs z\in \widehat\Xi_N$ and a
function $f:\Xi_{N}\rightarrow\mathbb{R}$, define a flow
$\Phi_{f,\boldsymbol{z}}$, supported on $\gamma_{\boldsymbol{z}}^{N}$,
by
\begin{equation*}
\Phi_{f,\boldsymbol{z}} (\boldsymbol{x},\boldsymbol{y})
\;=\;f(\boldsymbol{x})
c_{\boldsymbol{z}}(\boldsymbol{x},\boldsymbol{y})
-f(\boldsymbol{y})c_{\boldsymbol{z}}(\boldsymbol{y},\boldsymbol{x})\;,
\quad \bs x\,,\; \bs y \in\Xi_N\;.
\end{equation*}
Let $\Phi^*_{N}$ be the flow defined by
\begin{equation}
\label{af2}
\Phi^*_{N}\;=\;\sum_{\boldsymbol{z}\in\mathring{\mathcal{B}}_{N}}
\Phi_{V^*_{N},\boldsymbol{z}}\;,
\end{equation}
where $V^*_N$ is the approximation of the equilibrium potential
introduced in Subsection \ref{ad}.

All results presented in Subsections \ref{sub51}, \ref{sub52} for the
flow $\Phi_N$ are in force for $\Phi^*_N$, replacing $\mc L_N$, $V_N$
by $\mc L^*_N$, $V^*_N$, respectively. In particular, there exists 
a flow $\widetilde{\Phi}^*_{N}$ which is divergence-free on
$\{\bs{m}_1^N, \bs{m}_2^N\}^{c}$, such that
\begin{equation}
\label{af1}
\big(\textup{div}\,
  \widetilde{\Phi}^*_{N}\big)
(\bs{m}_1^N) \;=\; -\big(\textup{div}\,
  \widetilde{\Phi}^*_{N}\big)
(\bs{m}_2^N)\;=\; \big[ 1+o_{N}(1)\big] \,\kappa_N \, \omega_{\bs{0}}\;,
\end{equation}
and such that $\widetilde{\Phi}^*_{N} - \Phi^*_{N}$ is negligible. 

\subsection{\bf Final corrections on $\widetilde{\Phi}_{N}$.}  
\label{sub53}

In this subsection, we remove the $o_{N}(1)$ terms in (\ref{eq: div of
  crr}).  The following elementary lemma is useful in the forthcoming
computations.

\begin{lemma}
\label{lem:neg_lemma}
Suppose that the sequence of flows
$\mathfrak{f}_{N}\in\mathcal{F}_{N}$ is negligible. Then, for any
sequence of flows $\mathfrak{h}_{N}\in\mathcal{F}_{N}$,
\begin{equation*}
\left\Vert \mathfrak{h}_{N}+\mathfrak{f}_{N}\right\Vert ^{2}
\;\le\; \big[1+o_N(1)\big]\left\Vert \mathfrak{h}_{N}\right\Vert ^{2}
\;+\;  \kappa_{N}\,o_{N}(1)\;.
\end{equation*}
\end{lemma}

Next result provides a lower bound for the capacity $\mbox{cap}_{N}
(\{\bs{m}_1^N\}, \{\bs{m}_2^N\})$. 

\begin{lemma}
\label{s09}
We have that 
\begin{equation*}
\textup{cap}_{N}(\{\bs{m}_1^N\}, \{\bs{m}_2^N\}) \;\ge\; 
\big[1+o_N(1)\big]\,\kappa_N\,\omega_{\bs{0}}\;.
\end{equation*}
\end{lemma}

\begin{proof}
By Theorem \ref{s10}, it suffices to find
$g_{N}\in\mathfrak{C}_{0,0}(\{\bs{m}_1^N\}, \{\bs{m}_2^N\})$ and $\psi_{N}\in \mathfrak{U}_{1}
(\{\bs{m}_1^N\}, \{\bs{m}_2^N\})$ satisfying
\begin{equation}
\label{eq:lower main-1}
\kappa_N \omega_{\bs{0}}
\, \left\Vert \Phi_{g_{N}}-\psi_{N}\right\Vert ^{2} 
\;=\; 1+o_{N}(1)\;.
\end{equation}
Let $\psi_{N}$ and $g_N$ be the flow  and the function given by
\begin{align*}
\psi_{N} \;=\; \frac{1+c_{N}}{\kappa_N \omega_{\bs{0}}}\, 
\frac{\widetilde{\Phi}_{N}+\widetilde{\Phi}_{N}^{*}}{2}\;, \quad 
g_{N} \;=\; \frac{1+c_{N}}{\kappa_N \omega_{\bs{0}}}\,W_N\;,
\end{align*}
where $c_{N}=o_{N}(1)$ is the normalizing sequence which guarantees
that $\psi_{N}$ is a unitary flow, and where $W_N$ is the function
defined by $(V_N^*-V_N)/2$ on $\overline{\mathcal{B}}_N$ and $0$
otherwise. By definition, the flow $\psi_{N}$ belongs to
$\mathfrak{U}_{1}(\{\bs{m}_1^N\}, \{\bs{m}_2^N\})$ and the function
$g_{N}$ belongs to $\mathfrak{C}_{0,0} (\{\bs{m}_1^N\},
\{\bs{m}_2^N\})$. It remains to show that \eqref{eq:lower main-1}
holds.

By Proposition \ref{s15}, by \eqref{af1} and by Lemma \ref{lem:neg_lemma},
\begin{equation}
\label{eq:rem01}
\begin{aligned}
& \kappa_N \omega_{\bs{0}}
\left\Vert \Phi_{g_{N}}-\psi_{N}\right\Vert ^{2} \; 
=\; \frac{1+o_{N}(1)}{\kappa_N \omega_{\bs{0}}}\,
\Big\Vert \Phi_{W_N} - 
\frac{\widetilde{\Phi}_{N}+\widetilde{\Phi}_{N}^{*}}{2}\Big\Vert ^{2} \\
&\quad =\;  \frac{1+o_{N}(1)}{\kappa_N \omega_{\bs{0}}}\,
\Big\Vert \Phi_{W_N} -
\frac{\Phi_{N}+\Phi_{N}^{*}}{2}\Big\Vert ^{2}+o_N(1)\;.
\end{aligned}
\end{equation}
Define the outer boundary of $\mc{B}_N$ as 
\begin{equation*}
\overline{\partial}\mathcal{B}_N \;=\; \{\bs{z}\in\widehat{\Xi}_N
\setminus\mathring{\mathcal{B}}_N:\gamma_\bs{z}^N
\cap\overline{\mathcal B}_N\neq\varnothing\}\;.
\end{equation*}
With this notation, 
\begin{equation*}
\Phi_{W_N}\;=\;\sum_{\bs{z}\in\mathring{\mathcal{B}}_N}
\Phi_{(V_N^* -V_N)/2,\bs{z}} \,+ \,
\sum_{\bs{z}\in\overline{\partial}\mathcal{B}_N} \Phi_{W_N,\bs{z}}\;.
\end{equation*}
Therefore, by the definitions \eqref{40}, \eqref{af2} of the flows $\Phi_{N}$,
$\Phi_{N}^{*}$, 
\begin{equation*}
\Phi_{W_N} \;-\; \frac 12 \big( \Phi_{N}+\Phi_{N}^{*} \big) \;=\;  
\, -\; \sum_{\bs z\in\mathring{\mc
    B}_N} \Psi_{V_N, \bs z} \;+\; 
\sum_{\bs{z}\in\overline{\partial}\mathcal{B}_N} \Phi_{W_N,\bs{z}}\;,
\end{equation*}
where $\Psi_{f,\bs{z}}=[\Phi_{f,\bs{z}}+\Phi_{f,\bs{z}}^*]/2$. 
Denote the first term on the right hand side by $\Psi_{N}^{(1)}$ and
the second one by $\Psi_{N}^{(2)}$. By the definition of
$\Psi_{f,\bs{z}}$ and by Proposition \ref{s02},
\begin{equation*}
\big\Vert \Psi_{N}^{(1)} \big\Vert ^{2} \;=\; 
\mathcal{D}_N (V_N;\mathring{\mathcal{B}}_N)\;=\; 
\big[1+o_{N}(1)\big]\kappa_N\,\omega_{\bs{0}}\;.
\end{equation*}
Therefore, in view of \eqref{eq:rem01}, to conclude the proof of the
lemma it remains to show that $\Psi_{N}^{(2)}$ is a negligible flow.

By the definition of $\Phi_{f,z}$, 
\begin{align*}
\Psi_{N}^{(2)} (\bs x, \bs y) \;&=\;
\sum_{i=0}^{L-1} c_{\bs x - \bs z^N_i}(\bs x, \bs y)\, W_N(\bs x)\, 
\mb 1\{ \bs x - \bs z^N_i \in \overline{\partial}\mathcal{B}_{N} \,,\,
\bs y = \bs x + \bs z^N_{i+1} - \bs z^N_i\} \\
\;&-\;
\sum_{i=0}^{L-1} c_{\bs x - \bs z^N_i}(\bs y, \bs x)\, W_N(\bs y)\, 
\mb 1\{ \bs x - \bs z^N_i \in \overline{\partial}\mathcal{B}_{N} \,,\,
\bs y = \bs x + \bs z^N_{i-1} - \bs z^N_i\}\;.
\end{align*}
Therefore, by Schwarz inequality and since $c_{\bs z}(\bs w, \bs w')
\le 2 \, c^s(\bs w, \bs w')$ for all $\bs z$, $\bs w$, $\bs w'$,
\begin{equation*}
\Vert \Psi_{N}^{(2)} \Vert^2 \;\le\;
C_0 \sum_{i=0}^{L-1} \sum_{\boldsymbol{x}\in \overline{\partial}\mathcal{B}_{N}}
c_{\bs x}(\bs x + \bs z^N_i, \bs x + \bs z^N_{i+1})\, W_N(\bs x + \bs z^N_i)^2\;.
\end{equation*}
By definition of $W_N$ and by the bound \eqref{eq:3} on the
conductances, this expression is less than or equal to
\begin{equation*}
\frac{C_0}{Z_N} \sum_{i=0}^{L-1} \sum_{\boldsymbol{x}\in \overline{\partial}\mathcal{B}_{N}}
e^{-N F(\bs x)} \, [V_N(\bs x + \bs z^N_i) - V^*_N(\bs x + \bs
z^N_i)] ^2 \, \mb 1\{ \bs x + \bs z^N_i \in \overline{\mc B}_N\}\;.
\end{equation*}
The last indicator appeared because $W_N$ vanishes outside
$\overline{\mc B}_N$.  Since $\bs x$ belongs to
$\overline{\partial}\mathcal{B}_{N}$, $\bs x + \bs z^N_i \not\in \mc
B_N$ for $0\le i<L$. We may therefore replace the indicator appearing
in the previous formula by the indicator of the set $\{\bs x : \bs x +
\bs z^N_i \in \partial \mc B_N\}$. Therefore, by performing the change
of variables $\bs y = \bs x + \bs z^N_i $, we obtain that 
\begin{equation*}
\Vert \Psi_{N}^{(2)} \Vert^2 \;\le\;
\frac{C_0}{Z_N} \sum_{\boldsymbol{y}\in \partial\mathcal{B}_{N}}
e^{-N F (\bs y)} \, [V_N(\bs y) - V^*_N(\bs y)] ^2\;.
\end{equation*}

To show that this expression is of order $o_N(1) \kappa_N$, we
consider separately each part of the boundary
$\partial\mathcal{B}_{N}$.  For $\boldsymbol{y}
\in \partial^{0}\mathcal{B}_{N}$, since $|V_{N}^*-V_{N}|\le1$,
\begin{equation*}
e^{-N F(\boldsymbol{y})}[(V_{N}^{*}-V_{N})(\boldsymbol{y})]^{2}
\;\le\;C\,e^{-NH}\,e^{-\frac{1}{4}\lambda_{1}N\varepsilon_{N}^{2}}\;.
\end{equation*}
For $\boldsymbol{y}\in\partial^{1}\mathcal{B}_{N}$,
by the Cauchy-Schwarz inequality and by Lemma \ref{s18}, 
\begin{align*}
e^{-N F(\boldsymbol{y})}[(V_{N}^{*}-V_{N})(\boldsymbol{y})]^{2}
\; & \le\;  2e^{-N F (\boldsymbol{y})}
\left([1-V_{N}(\boldsymbol{y})]^{2}+[1-V_{N}^{*}(\boldsymbol{y})]^{2}\right)
\\
& \le  \; C_0 \,e^{-NH}\,e^{-c_{0}N\varepsilon_{N}^{2}}
\end{align*}
for some $c_0>0$.  An analogously argument applies to $\boldsymbol{y}
\in \partial^{2}\mathcal{B}_{N}$.  It follows from the last three
estimates that there exists $a>0$ such that
\begin{equation*}
\Vert \Psi_{N}^{(2)} \Vert^2 \;\le\; C \kappa_N
N^{-\frac{d}{2}-1}e^{-aN\varepsilon_N^2}\;. 
\end{equation*}
This shows that $\Psi_{N}^{(2)}$ is a negligible flows and completes
the proof of the lemma.
\end{proof}

The lower bound presented in the previous lemma is sharp if there are
only two metastable sets $\mathcal{E}_{N}^{1}$, $\mathcal{E}_{N}^{2}$
and only one saddle point between them. It is not sharp otherwise, as
shall be seen in the next section.

Let $\textup{cap}_N^s (\cdot,\cdot)$ be the capacity with respect to
the process generated by $\mc L_N^s$. Then, by Lemma \ref{lem: sector
  condition} and by \cite[Lemma 2.6]{GL}, for any disjoint subsets
$\mc A$, $\mc B$ of $\Xi_N$,
\begin{equation}
\label{sc1}
\textup{cap}_N(\mc A,\mc B)\le 4\,L^2 \,\textup{cap}_N^s (\mc A,\mc B)\;.
\end{equation}

\begin{lemma}
\label{s20}
For any sequence $\delta_{N}\downarrow 0$, there exists a negligible flow
$\varrho_{N}\in\mathcal{F}_{N}$ which is divergence-free on
$\{\bs{m}_1^N , \bs{m}_2^N\}^{c}$ and such that
\begin{equation*}
\left(\mbox{\textup{div}\,}\varrho_{N}\right)(\bs{m}_1^N) 
\;=\;  -\left(\mbox{\textup{div}\,}\varrho_{N}\right)
(\bs{m}_2^N)\;=\;\kappa_N\delta_{N}\;.
\end{equation*}
\end{lemma}

\begin{proof}
By the Thomson principle for the reversible process generated by
$\mathcal{L}_{N}^{s}$, there is a unit flow
$\widetilde{\varrho}_{N}\in\mathfrak{U}_{1}(\{\bs{m}_1^N\}, \{\bs{m}_2^N\})$
such that
\begin{equation*}
\left\Vert \widetilde{\varrho}_{N}\right\Vert ^{2} 
\;=\;
\frac{1}{\mbox{cap}_{N}^{s}(\{\bs{m}_1^N\}, \{\bs{m}_2^N\})}
\; \le\;
\frac{4\, L^{2}}{\mbox{cap}_{N}(\{\bs{m}_1^N\}, \{\bs{m}_2^N\})}
\;\le\; C\, \kappa_N^{-1}\;,
\end{equation*}
where we used \eqref{sc1} and Lemma \ref{s09} to
obtain the last two inequalities. Therefore, $\varrho_{N} \;=\;
\kappa_N \delta_{N} \widetilde{\varrho}_{N}$ fulfils all the
requirements of the lemma.
\end{proof}

By adding the flow $\varrho_N$ to the flow $\widetilde{\Phi}_{N}$
introduced in Proposition \ref{s15} we obtain a new flow whose
divergences at $\bs{m}_1^N$ and $\bs{m}_2^N$ are exactly equal to
$\pm\kappa_N\omega_{\bs{0}}$. This completes the proof of Theorem
\ref{s24}.

\section{Computation of Capacities}
\label{sec6}

We prove in this section a special case of Theorem \ref{s133}. More
precisely, we are concentrating on the case $A=\{i\}$,
$B=S\setminus\{i\}$, $i\in S$. In this case, there is no ambiguity in
the constructions of the approximations of equilibrium potential
$V_{\mc{E}_N (A), \mc{E}_N (B)}$ and the flow $\Phi_{V_{\mc{E}_N (A),
    \mc{E}_N (B)}}$ and it is very clear how to use the building
blocks that we obtained so far. In the next section, based on the
argument of the current section, we prove Theorem \ref{s133} for
general $A$ and $B$.
 
Define $\breve{\mathcal{E}}_{N}^{i}=\mc{E}_N(S\setminus\{i\})$. Since
it is obvious that $\textup{cap}_Y(\{i\},S\setminus\{i\})=\sum_{k=1}^M
\omega_{i,k}$, the following theorem is the main result of this
section.

\begin{theorem}
\label{s13}
For every $1\le i\le M$, 
\begin{equation}
\label{s13e}
\frac{Z_{N}}{(2\pi N)^{\frac{d}{2}-1}}\, e^{NH}
\textup{cap}_{N}(\mathcal{E}_{N}^{i},\breve{\mathcal{E}}_{N}^{i}) 
\;=\; \big[ 1 + o_{N}(1) \big]\,\sum_{k=1}^{M}\omega_{i,k} \;.
\end{equation}
\end{theorem}

The proof of this result is based on the construction of an
approximation, denoted by $h_N$, of the equilibrium potential $V_{\mc
  E^1_N, \breve{\mc E}^1_N}$, and of an approximation, denoted by
$\Upsilon_N$, of the flow $\Phi^*_{V_{\mc E^1_N, \breve{\mc
      E}^1_N}}$. Identical arguments, left to the reader, permit to
define approximations of the equilibrium potential $V^*_{\mc E^1_N,
  \breve{\mc E}^1_N}$ and of the flow $\Phi_{V^*_{\mc E^1_N,
    \breve{\mc E}^1_N}}$. We assume throughout this section, without
loss of generality, that $i=1$ in the statement of Theorem \ref{s13}.

Let $\boldsymbol{\sigma}$ be a saddle point in $\mathfrak{S}_{1,j}$,
$j\not = 1$. All sets, functions and flows introduced in the previous
sections are represented in this section with an extra upper index
$\bs \sigma$ to specify the saddle point. For example, we denote by
$\mathcal{B}_{N}^{\boldsymbol{\sigma}}$, $\mc{C}_N^{\bs{\sigma}}$ the
mesoscopic neighborhood of $\boldsymbol{\sigma}$ defined in \eqref{25}
and \eqref{C_N}.

For a saddle point $\boldsymbol{\sigma}\in \mathfrak{S}$, denote by
$-\lambda_{1}^ {\boldsymbol{\sigma}}$ the negative eigenvalue of
$(\mbox{Hess }F) (\boldsymbol{\sigma})$ and let $\lambda= 
\min_{\boldsymbol{\sigma} \in \mathfrak{S}} \lambda^{\bs
  \sigma}_{1}$,
\begin{equation*}
\Omega_{N} \;=\; \left\{ \boldsymbol{z}\in\Xi:\,
F(\boldsymbol{z})\,\le\,H + (1/4) \lambda\, \varepsilon_{N}^{2}\right\} \;.
\end{equation*}
In view of \eqref{26}, denote by $W_N^1,\,W_N^2,\,\dots,\,W_N^M$ the
connected components of $\Omega_N \setminus
\left(\bigcup_{\boldsymbol{\sigma} \in \mathfrak{S}}
  \mathcal{C}_{N}^{\boldsymbol{\sigma}}\right)$ such that
$W_i^\epsilon \subset W_N^i$, $1\le i\le M$, and let $\mc W_N^i =W_N^i
\cap \Xi_N$. Note that $\mc E_N^i \subset \mc W_N^i$ by
definition. Let $\mathcal{X}_{N}$ be the outer region:
\begin{equation*}
\mathcal{X}_{N} \;=\; \Xi_N \setminus \Big( \bigcup_{k=1}^{M}
\mathcal{W}_{N}^{k} \bigcup_{\boldsymbol{\sigma}\in\mathfrak{S}}
\mathcal{B}_{N}^{\sigma}\Big)\;.
\end{equation*}
We refer to Figure \ref{fig:global fig} for an illustration of these
sets.  

\begin{figure}
\centering
\includegraphics[scale=0.25]{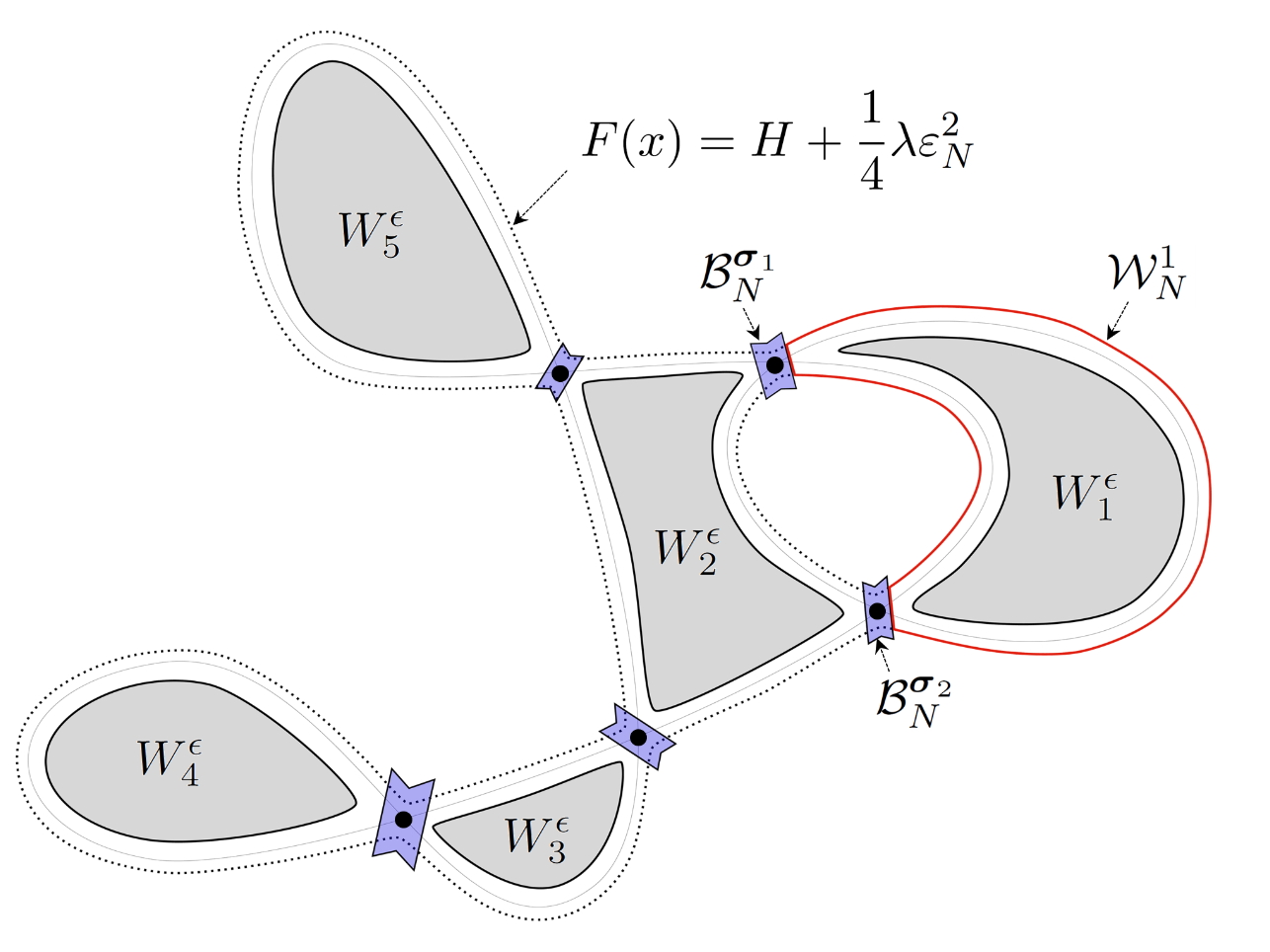}
\protect\caption{\label{fig:global fig} The sets
  $\mathcal{W}_{N}^{i}$, $\mathcal{B}_{N}^{\boldsymbol{\sigma}}$ and
  $\mathcal{X}_{N}$. Note that the sets $\mc B^{\bs \sigma}_N$ are
  thicker than the sets $\mc W^i_N$ because the formers are defined by
  the bound $F(\bs x) \le \lambda^{\bs \sigma}_1 \varepsilon^2_N/4$,
  while the latters are defined by the bound $F(\bs x) \le \lambda
  \varepsilon^2_N/4$.}
\end{figure}

Let $\mf S_1 = \bigcup_{2\le k\le M}\mathfrak{S}_{1,k}$, and denote by
$h_N$ the approximations of the equilibrium potential $V_{\mc E^1_N,
  \breve{\mc E}^1_N}$ given by
\begin{equation*}
h_{N}(\boldsymbol{x})\;=\;\begin{cases}
1 & \boldsymbol{x}\in\mathcal{W}_{N}^{1}\\
V^{\boldsymbol{\sigma}}_N(\boldsymbol{x}) & 
\boldsymbol{x}\in\overline{\mathcal{B}}_{N}^{\boldsymbol{\sigma}},
\,\boldsymbol{\sigma}\in\mathfrak{S}_{1}\\
0 & \mbox{otherwise,}
\end{cases}
\end{equation*}

\begin{lemma}
\label{s23}
Recall from \eqref{45} the definition of $\omega_{\bs \sigma}$. Then, 
\begin{equation*}
\mathcal{D}_{N}(h_{N}) \;=\; \big[ 1 + o_{N}(1) \big]\, \kappa_{N}\,
\sum_{\bs{\sigma}\in\mathfrak{S}_{1}}\omega_{\bs{\sigma}}\;.
\end{equation*}
\end{lemma}

\begin{proof}
Decompose $\mathcal{D}_{N}(h_{N})$ into
\begin{equation*}
\mathcal{D}_{N}(h_{N}) \;=\;
\sum_{\boldsymbol{\sigma}\in\mathfrak{S}_{1}}
\mathcal{D}_{N}(h_{N};\,
\mathring{\mathcal{B}}_{N}^{\boldsymbol{\sigma}})
\;+\; \mathcal{D}_{N}(h_{N};\,\mathcal{R}_{N})\;,
\end{equation*}
where $\mathcal{R}_{N}= \widehat\Xi_{N}\setminus
(\bigcup_{\boldsymbol{\sigma}\in\mathfrak{S}_{1}}
\mathring{\mathcal{B}}_{N}^{\boldsymbol{\sigma}})$. Since $h_N$
coincides with $V^{\bs \sigma}_N$ on
$\overline{\mathcal{B}}_N^{\bs{\sigma}}$, by Proposition \ref{s02},
for each $\boldsymbol{\sigma}$ in $\mathfrak{S}_{1}$,
\begin{equation*}
\mathcal{D}_{N}(h_{N};\,\mathring{\mathcal{B}}_{N}^{\boldsymbol{\sigma}})
\;=\; \big[1+o_{N}(1)\big]\, \kappa_{N}\, \omega_{\boldsymbol{\sigma}}\;.
\end{equation*}
To complete the proof of the lemma, it remains to show that
\begin{equation}
\label{eq:negs}
\mathcal{D}_{N}(h_{N};\mathcal{R}_{N}) \;=\; 
\kappa_{N}\,o_{N}(1) \;.
\end{equation}
The argument presented below to prove this assertion will be used
several times in the remaining part of this article. For this reason,
we present very carefully each step.

By (\ref{eq: Decomposition of Diri Form -2}) and \eqref{13},
$\mathcal{D}_{N}(h_{N};\mathcal{R}_{N})$ is equal to
\begin{equation}
\label{30}
\begin{aligned}
& \frac{1}{2 Z_N}\sum_{\boldsymbol{x}\in\mathcal{R}_{N}}
e^{-N\overline{F}_{N}(\boldsymbol{x})}
\sum_{i=0}^{L-1} \big[h_{N}(\boldsymbol{x}+\boldsymbol{z}_{i+1}^{N})
-h_{N}(\boldsymbol{x}+\boldsymbol{z}_{i}^{N})\big]^{2} \\
&\quad =\; \frac{\kappa_{N}}{2(2\pi N)^{\frac{d}{2}-1}}
\sum_{\boldsymbol{x}\in\mathcal{R}_{N}}
e^{N(H-\overline{F}_{N}(\boldsymbol{x}))}
\sum_{i=0}^{L-1}\left[h_{N}(\boldsymbol{x}+\boldsymbol{z}_{i+1}^{N})
-h_{N}(\boldsymbol{x}+\boldsymbol{z}_{i}^{N})\right]^{2} \;.
\end{aligned}
\end{equation}

In the previous sum, we may restrict our attention to the points $\bs
x\in \mc R_N$ such that $F_{N}(\boldsymbol{x}) \le H + (1/8) \lambda
\varepsilon^2_N$ because, as $h_N$ is bounded by $1$, the contribution
of the terms $\bs x\in \mc R_N$ such that $F_{N}(\boldsymbol{x}) \ge H
+ (1/8) \lambda \varepsilon^2_N$ is of order $ \kappa_N o_N(1)$ by
definition of $\varepsilon_N$.

To estimate the remaining sum, fix $\bs x\in\mc R_N$ and $0\le i<L$.
Assume first that $\bs x+ \bs z^N_i$ belongs to $\overline{\mc B}^{\bs
  \sigma}_N$ for some $\bs \sigma\in \mf S_1$. To fix ideas suppose
that $\bs \sigma \in \mf S_{1,2}$.  As $\bs x$ is an element of $\mc
R_N$, it does not belong to $\mathring{\mc B}^{\bs \sigma}_N$. Thus
$\bs x+\bs z^N_j\not\in \mc B ^{\bs \sigma}_N$ for all $0\le j<L$. In
particular, $\bs x+\bs z^N_i\not\in \mc B ^{\bs \sigma}_N$ so that
$\bs x+\bs z^N_i\in \partial \mc B ^{\bs \sigma}_N$. The point $\bs x+
\bs z^N_i$ can not belong to $\partial^0\mc B^{\bs \sigma}_N$ because,
by \eqref{27}, points $\bs y$ in this set are such that
$F_{N}(\boldsymbol{y}) \ge H+ (1/4) \lambda^{\bs \sigma}_1
\varepsilon^2_N - C_0/N$ and we already assumed that
$F_{N}(\boldsymbol{x}) \le H + (1/8) \lambda \varepsilon^2_N$.
Thus, $\bs x+ \bs z^N_i$ belongs to $\partial^1\mc B^{\bs \sigma}_N
\cup \partial^2\mc B^{\bs \sigma}_N$. 

Suppose, to fix ideas, that $\bs x+ \bs z^N_i$ belongs to
$\partial^1\mc B^{\bs \sigma}_N$.  The argument in the case $\bs x+
\bs z^N_i \in \partial^2\mc B^{\bs \sigma}_N$ is analogous. As
$F_{N}(\boldsymbol{x}) \le H + (1/8) \lambda \varepsilon^2_N$, $\bs
x + \bs {z}_{i+1}^{N}$ belongs to $\overline{\mc B}^{\bs \sigma}_N$ or
to $\mc W^1_N$.

Consider first the case where $\bs x + \bs {z}_{i+1}^{N}$ belongs to
$\overline{\mc B}^{\bs \sigma}_N$. Recall the arguments presented in
the penultimate paragraph which led to the conclusion that $\bs x+ \bs
z^N_i$ belongs to $\partial^1\mc B^{\bs \sigma}_N \cup \partial^2\mc
B^{\bs \sigma}_N$. Applied to $\bs x + \bs {z}_{i+1}^{N}$ this
reasoning permits to conclude that this point belongs to
$\partial^1\mc B^{\bs \sigma}_N \cup \partial^2\mc B^{\bs
  \sigma}_N$. Since, by assumption, $\bs x+ \bs z^N_i
\in \partial^1\mc B^{\bs \sigma}_N$, we also have that $\bs x+ \bs
z^N_{i+1} \in \partial^1\mc B^{\bs \sigma}_N$.

As $h_N = V^{\bs \sigma}_N$ on $\overline{\mc B}^{\bs \sigma}_N$,
$h_N(\bs x+ \bs z^N_i) = V^{\bs \sigma}_N (\bs x+ \bs z^N_i)$ and
$h_N(\bs x+ \bs z^N_{i+1}) = V^{\bs \sigma}_N (\bs x+ \bs
z^N_{i+1})$. Since both points belong to $\partial^1\mc B^{\bs
  \sigma}_N$, we may apply Lemma \ref{s18} to conclude that the
sum on the right hand side of \eqref{30} restricted to points $\bs x
\in \mc R_N$ and indices $0\le i<L$ satisfying the new set of
conditions is of order $\kappa_N o_N(1)$.

Consider now the case where $\bs x + \bs {z}_{i+1}^{N}$ belongs to
$\mc W^1_N$. In consequence, $h_N(\bs x + \bs {z}_{i+1}^{N}) =1$.
Since $\bs x+ \bs z^N_i\in \partial^1\mc B^{\bs \sigma}_N$ and
$h_N(\bs x+ \bs z^N_i) = V^{\bs \sigma}_N (\bs x+ \bs z^N_i)$, we may
also apply Lemma \ref{s18} to conclude that the
sum on the right hand side of \eqref{30} restricted to points $\bs x
\in \mc R_N$ and indices $0\le i<L$ satisfying all the above
conditions is of order $\kappa_N o_N(1)$.

By symmetry, the previous argument applies also to the case where $\bs
x+ \bs z^N_{i+1}$ belongs to $\overline{\mc B}^{\bs \sigma}_N$ for
some $\bs \sigma\in \mf S_1$. It remains therefore to consider the
case in which $\bs x+ \bs z^N_{i}$ and $\bs x+ \bs z^N_{i+1}$ do not
belong to $\cup_{\bs \sigma\in\mf S_1} \overline{\mc B}^{\bs
  \sigma}_N$, and $F_{N}(\boldsymbol{x}) \le H + (1/8) \lambda
\varepsilon^2_N$. For such points $h_N$ is equal to $0$ or $1$, and
the only possible contribution occurs if $h_N(\bs x+ \bs z^N_{i})=1$
and $h_N(\bs x+ \bs z^N_{i+1})=0$, or the contrary. These identities
imply that the point $\bs x$ is close to the boundary of $\mc W^1_N$,
but the only part of the boundary of $\mc W^1_N$ in which
$F_{N}(\boldsymbol{x}) \le H + (1/8) \lambda \varepsilon^2_N$ is the
one with $\cup_{\bs \sigma\in\mf S_1} \overline{\mc B}^{\bs \sigma}_N$
which has already been examined. This completes the proof of
\eqref{eq:negs} and the one of the lemma.
\end{proof}

The next step consists in defining a flow, denoted below by
$\Upsilon_N$, which approximates $\Phi^*_{h_N}$.  Denote by
$\mathring{\mathcal{W}}_{N}^{i}$ the core of $\mathcal{W}_{N}^{i}$.
Define the flows $\Upsilon_{N}\in\mathcal{F}_{N}$ as
\begin{equation*}
\Upsilon_{N} \;=\; \Upsilon_{N}^{1} \;+\;
\sum_{\boldsymbol{\sigma}\in\mathfrak{S}_{1}}
\widetilde{\Phi}^{\boldsymbol{\sigma}}_N\;,
\end{equation*}
where the flows $\widetilde{\Phi}^{\boldsymbol{\sigma}}_N$ have been
introduced in Theorem \ref{s24}, and where
\begin{equation*}
\Upsilon_{N}^{1}=\sum_{\boldsymbol{z}\in\mathring{\mathcal{W}}_{N}^{1}}
\Phi_{1,\boldsymbol{z}}^{*} \;.
\end{equation*}
The flow $\Phi_{1,\boldsymbol{z}}^{*}$ is obtained by taking $f$ as a
constant function $f\equiv 1$ at $\Phi_{f,\bs{z}}^*$. The flow
$\Upsilon_{N}^{1}$ has been added to take into account the fact that
$h_N=1$ on the set $\mc W^1_N$.

\begin{lemma}
\label{s22}
The flow $\Upsilon_{N}$ is divergence-free on
$(\mathcal{E}_{N})^{c}$ and satisfy
\begin{equation*}
(\textup{div}\,\Upsilon_{N})(\mathcal{E}_{N}^{1}) \;=\;
-\, (\textup{div}\,\Upsilon_{N})(\mathcal{\breve{\mathcal{E}}}_{N}^{1})
\;=\; \kappa_{N}\sum_{\boldsymbol{\sigma}\in\mathfrak{S}_{1}}
\omega_{\boldsymbol{\sigma}}\;.
\end{equation*}
Furthermore, $\Upsilon_{N}-\Phi_{h_{N}}^{*}$ is a negligible flow.
\end{lemma}

\begin{proof}
Since the conductance is constant over each cycle, the flow
$\Upsilon_{N}^{1}$ is divergence-free.  Furthermore, since the
divergence functional is additive, the first assertion of the
lemma follows from Theorem \ref{s24}.

By Theorem \ref{s24}, to prove that $\Upsilon_{N}-\Phi_{h_{N}}^{*}$ is
a negligible flow, it is enough to show that $\Upsilon^1_{N} +
\sum_{\boldsymbol{\sigma} \in\mathfrak{S}_{1}}
\Phi^{\boldsymbol{\sigma}}_N -\Phi_{h_{N}}^{*}$ is negligible. By
\eqref{40}, this difference is equal to
\begin{equation*}
\sum_{\bs z\in \mathring{\mc W}^1_N} \Phi^*_{1, \bs z} 
\;+\; \sum_{\boldsymbol{\sigma}  \in\mathfrak{S}_{1}} 
\sum_{\bs z\in \mathring{\mc B}^{\bs \sigma}_N} \Phi^*_{\bs z , V^{\bs \sigma}_N} 
\;-\; \sum_{\bs z\in \widehat{\Xi}_N} \Phi^*_{\bs z , h_N}\;.
\end{equation*}
Since $h_N = V^{\bs \sigma}_N$ on $\mc B^{\bs \sigma}_N$, $h_N=1$ on
$\mc W^1_N$, and $h_N=0$ on the complement of these sets, the unique
edges which survive in this difference belong to the boundary of these
sets. At the boundary of these sets, the function $F_N$ is bounded
below by $H+(1/4)\lambda\varepsilon^2_N - C/N$, and we may repeat the
arguments presented in the proof of Lemma \ref{s23} to show that
the flow appearing in the previous displayed formula is negligible. 
\end{proof}

\smallskip\noindent{\bf Proof of Theorem \ref{s13}.} We start by
proving the upper bound of the capacity. Define 
\begin{equation*}
f_{N}\;=\;\frac{1}{2}(h_{N}^{*}+h_{N}) 
\;,\quad  \phi_{N}\;=\;\frac{1}{2}(\Upsilon_{N}^{*}-\Upsilon_{N}) \;.
\end{equation*}
By definition of $h_N$, $h^*_N$, the function $f_N$ belongs to
$\mathfrak{C}_{1,0} (\mathcal{E}_{N}^{1},
\breve{\mathcal{E}}_{N}^{1})$.  By Lemma \ref{s22}, the flow
$\phi_{N}$ belongs to $\mathfrak{U}_{0} (\mathcal{E}_{N}^{1},
\breve{\mathcal{E}}_{N}^{1})$, and the flow
$\phi_{N}-\frac{1}{2}(\Phi_{h_{N}^{*}}-\Phi_{h_{N}}^{*})$ is
negligible. Therefore, by Lemma \ref{lem:neg_lemma},
\begin{equation*}
\left\Vert \Phi_{f_{N}} - \phi_{N}\right\Vert ^{2} \;\le\; 
\big[1+o_{N}(1)\big]\, \left\Vert \Phi_{\frac{1}{2} (h_{N}^{*}+h_{N})} -
\frac{1}{2}(\Phi_{h_{N}^{*}}-\Phi_{h_{N}}^{*}) \right\Vert ^{2}
\;+\;  \kappa_{N}\,o_{N}(1)\;.
\end{equation*}
The flow appearing in this formula is equal to $(1/2) \{\Phi_{h_{N}} +
\Phi_{h_{N}}^{*} \}$. Hence, by the explicit expression of these flows
and of the flow norm, the previous expression is equal to
\begin{equation}
\label{46}
\big[1+o_{N}(1)\big]\, \mathcal{D}_{N}(h_{N}) \;+\; \kappa_{N}\,o_{N}(1)
\;=\; \big[1+o_{N}(1)\big]\,  \kappa_{N} \sum_{\boldsymbol{\sigma}\in\mathfrak{S}_{1}}
\omega_{\boldsymbol{\sigma}} \;,
\end{equation}
where we applied Lemma \ref{s23} to derive the last identity.
To complete the proof of the upper bound, it remains to apply Theorem
\ref{s07} to obtain that
\begin{equation*}
\mbox{cap}_{N}(\mathcal{E}_{N}^{1},\breve{\mathcal{E}}_{N}^{1})
\;\le\; \big[1+o_{N}(1)\big]\, \kappa_{N} \sum_{\boldsymbol{\sigma}\in\mathfrak{S}_{1}}
\omega_{\boldsymbol{\sigma}}\;.
\end{equation*}

To prove the converse inequality, let $g_{N}$ and $\psi_{N}$ be given by
\begin{equation*}
g_{N}\;=\; \frac 1{\kappa_{N} \sum_{\boldsymbol{\sigma}\in\mathfrak{S}_{1}}
\omega_{\boldsymbol{\sigma}}} \, \frac{h_{N}^{*}-h_{N}}{2} \;, \quad 
\psi_{N}\;=\; \frac 1{\kappa_{N} \sum_{\boldsymbol{\sigma}\in\mathfrak{S}_{1}}
\omega_{\boldsymbol{\sigma}}} \, \frac{\Upsilon_{N}^{*}+\Upsilon_{N}}{2} \;,
\end{equation*}
As above, by definition of $h_N$, $h^*_N$, the function $g_N$ belongs to
$\mathfrak{C}_{0,0} (\mathcal{E}_{N}^{1},
\breve{\mathcal{E}}_{N}^{1})$, and by Lemma \ref{s22}, the flow
$\phi_{N}$ belongs to $\mathfrak{U}_{1} (\mathcal{E}_{N}^{1},
\breve{\mathcal{E}}_{N}^{1})$.

By Lemmata \ref{s23} and \ref{lem:neg_lemma},
\begin{equation*}
\Big(\kappa_{N} \sum_{\boldsymbol{\sigma}\in\mathfrak{S}_{1}}
\omega_{\boldsymbol{\sigma}} 
\Big)^2 \left\Vert \Phi_{g_{N}}-\psi_{N}\right\Vert ^{2}
\;\le\;\big[1+o_{N}(1)\big]\, \left\Vert \Phi_{\frac{h_{N}^{*}-h_{N}}{2}}
-\frac{\Phi_{h_{N}^{*}}+\Phi_{h_{N}}^{*}}{2}\right\Vert ^{2} \;+\;
\kappa_{N}\,o_{N}(1)  \;.
\end{equation*}
The flow appearing on the right hand side is equal to $-(1/2)
(\Phi_{h_{N}} + \Phi^*_{h_{N}})$. Hence, by \eqref{46}, the previous
expression is equal to $[1+o_{N}(1)] \kappa_{N}
\sum_{\boldsymbol{\sigma}\in\mathfrak{S}_{1}}
\omega_{\boldsymbol{\sigma}}$ so that
\begin{equation*}
\Big(\kappa_{N} \sum_{\boldsymbol{\sigma}\in\mathfrak{S}_{1}}
\omega_{\boldsymbol{\sigma}} 
\Big) \left\Vert \Phi_{g_{N}}-\psi_{N}\right\Vert ^{2}
\;\le\; 1\,+\, o_{N}(1) \;.
\end{equation*}
Finally, by Theorem \ref{s10}, 
\begin{equation*}
\mbox{cap}_{N}(\mathcal{E}_{N}^{1},\breve{\mathcal{E}}_{N}^{1})
\;\ge\; \big[1+o_{N}(1)\big]\, \kappa_{N} \sum_{\boldsymbol{\sigma}\in\mathfrak{S}_{1}}
\omega_{\boldsymbol{\sigma}}\;,
\end{equation*}
which completes the proof of the theorem.

\section{Proof of Theorem \ref{s133}}
\label{sec7}

Recall the definition of the Markov chain $Y(t)$ introduced just above
the statement of Theorem \ref{s133}.  The generator of $Y(t)$, denoted
by $\mc L_Y$, is given by
\begin{equation*}
(\mathcal{L}_{Y}f) (i)\;=\; \sum_{j\in S}
\frac{\omega_{i,j}}{\overline{\omega}_{i}} [f(j)-f(i)]
\end{equation*}
for each function $f:S\rightarrow\mathbb{R}$. The associated Dirichlet
form with respect to the equilibrium measure $\mu$ is given by
\begin{equation*}
\mathcal{D}_{Y}(f)\;=\;\left\langle f,\,-\mathcal{L}_{Y}f\right\rangle
_{\mu} \;=\;\frac{1}{2}\sum_{i,j\in S}\omega_{i,j}\, [f(j)-f(i)]^{2}\;.
\end{equation*}

For disjoint subsets $A,\,B$ of $S$, denote by $q_{A,B}$ the equilibrium
potential between $A$ and $B$: 
\begin{equation*}
q_{A,B}(k)\;=\;\mathbf{P}_{k}^{Y}[H_{A}<H_{B}]\;.
\end{equation*}
We review some well-known properties of the equilibrium potential
needed below. The first property is that the capacity between $A$ and
$B$ is given by the Dirichlet form of $q_{A,B}$: 
\begin{equation}
\textup{cap}_{Y}(A,B)\;=\;\mathcal{D}_{Y}(q_{A,B})\;.\label{qq1}
\end{equation}
Recall that the equilibrium potential can be characterized as the
solution of discrete elliptic equation 
\begin{equation}
\label{qq2} 
\begin{cases}
(\mathcal{L}_{Y}q_{A,B}) (k) \; = \;  0
& k\in S\setminus(A\cup B)\;,\\
q_{A,B}(k) \;=\; 1 & k\in A \;, \\
q_{A,B}(k) \;=\; 0 &  k\in B\;,
\end{cases}
\end{equation}
which reduces to a linear equation of dimension $M-|A\cup B|$. By
using \eqref{qq2}, we can rewrite \eqref{qq1} as
\begin{equation}
\label{qq3}
\textup{cap}_{Y}(A,B)\;=\;-\sum_{a\in A}\mu(a)\, 
(\mathcal{L}_{Y}q_{A,B})(a)\;=\;
\sum_{b\in B}\mu(b)\, (\mathcal{L}_{Y}q_{A,B}) (b)\;.
\end{equation}

\subsection{Approximation of the equilibrium potential $V_{\mc
    E_N(A), \mc E_N(B)}$}
\label{sub_pc}

Fix two disjoint subsets $A$, $B$ of $S$. The main difficulty of the
proof of Theorem \ref{s133} consists in constructing a good
approximation of the equilibrium potential between
$\mathcal{E}_{N}(A)$ and $\mathcal{E}_{N}(B)$. If $A\cup B\neq S$,
this construction requires a more refined analysis than the one
presented in the previous section, as we have to find the correct
value of the approximation of equilibrium potential at the sets
$\mathcal{E}_N^i$, $i \not\in A\cup B$. This value is related to the
equilibrium potential of the process $Y$ between $A$ and $B$, as
evidenced below. We present the arguments for $V_{\mc E_N(A), \mc
  E_N(B)}$, as the ones for $V_{\mc E_N(A), \mc E_N(B)}^*$ are
analogous.

For $\boldsymbol{\sigma}\in\mathfrak{S}_{i,j}$, $i<j$, denote by
$V_{N}^{\boldsymbol{\sigma}}$, the approximation, introduced in
Section \ref{sec3}, of $V_{\mc E_N^i, \mc E_N^j}$, the equilibrium
potential between $\mathcal{E}_{N}^{i}$ and $\mathcal{E}_{N}^{j}$, in
the mesoscopic neighborhood $\mc B^{\bs \sigma}_N$.
For a function $q: S\to \bb R$,
define the function $h_{N}^{q}:\Xi_{N}\rightarrow\mathbb{R}$ by
\begin{equation*}
h_{N}^{q}(\boldsymbol{x})\;=\;\begin{cases}
q(i) & \boldsymbol{x}\in\mathcal{W}_{N}^{i}\;,\\{}
[q(i)-q(j)]V_{N}^{\boldsymbol{\sigma}}(\boldsymbol{x})+q(j) & 
\boldsymbol{x}\in\overline{\mathcal{B}}_{N}^{\boldsymbol{\sigma}},
\,\boldsymbol{\sigma}\in\mathfrak{S}_{i,j},\,i<j\;,\\
0 & \mbox{otherwise\;.}
\end{cases}
\end{equation*}

\begin{lemma}
\label{con71}
For any function $q:S\rightarrow\mathbb{R}$, we have that
\begin{equation*}
\mathcal{D}_{N}(h_{N}^{q})\;=\;\big[1+o_{N}(1)\big]\, \kappa_{N}
\,\mathcal{D}_{Y}(q)\;.
\end{equation*}
\end{lemma}

\begin{proof}
Write $\mathcal{D}_{N}(h_{N}^{q})$ as 
\begin{equation*}
\mathcal{D}_{N}(h_{N}^{q})\;=\;
\sum_{\boldsymbol{\sigma}\in\mathfrak{S}} \mathcal{D}_{N}(h_{N}^{q};
\mathring{\mathcal{B}}_{N}^{\boldsymbol{\sigma}})
\;+\; \mathcal{D}_{N}(h_{N}^{q};\widehat{\Xi}_{N}\setminus
\cup_{\boldsymbol{\sigma}\in\mathfrak{S}}
\mathring{\mathcal{B}}_{N}^{\boldsymbol{\sigma}})\;.
\end{equation*}
By the proof of Lemma \ref{s23}, the second term is bounded by
$\kappa_{N}\, o_{N}(1)$. On the other hand, for each $i<j$ and
$\bs\sigma \in\mf S_{i,j}$, by Proposition \ref{s02},
\begin{equation*}
\mathcal{D}_{N}(h_{N}^{q};
\mathring{\mathcal{B}}_{N}^{\boldsymbol{\sigma}})
\;=\;\big[1+o_{N}(1)\big]\, \kappa_{N}\,
\omega_{\boldsymbol{\sigma}}\, [q(i)-q(j)]^{2}\;.
\end{equation*}
Hence, 
\begin{equation*}
\mathcal{D}_{N}(h_{N}^{q}) \;=\; 
\big[1+o_{N}(1)\big]\, \kappa_{N}\, \sum_{1\le i<j\le M}
\sum_{\boldsymbol{\sigma}\in\mathfrak{S}_{i,j}}
\omega_{\boldsymbol{\sigma}}\, [q(i)-q(j)]^{2} \;+\; 
 \kappa_{N}\,o_{N}(1)\;,
\end{equation*}
which completes the proof of the lemma. 
\end{proof}

In view of the previous result, to minimize
$\mathcal{D}_{N}(h_{N}^{q})$ among all functions which vanish at $\mc
E_N(B)$ and which are equal to $1$ at $\mc E_N(A)$, we have to choose
$q$ as the equilibrium potential between $A$ and $B$ for the random
walk $Y(t)$.

Recall from Theorem \ref{s24} the definition of the flows
$\widetilde{\Phi}_{N}^{\boldsymbol{\sigma}}$, $\bs\sigma\in\mf S$. For
each function $q:S\to \bb R$, define the flow $\Upsilon_{N}^{q}$ which
approximates $\Phi_{h_{N}^{q}}^{*}$ by
\begin{equation*}
\Upsilon_{N}^{q}  \;=\;  \sum_{i\in S}\,
\sum_{\boldsymbol{z}\in\mathring{\mathcal{W}}_{N}^{1}}
\Phi_{q(i),\boldsymbol{z}}^{*}
\;+\; \sum_{1\le i<j\le M}\,
\sum_{\boldsymbol{\sigma}\in\mathfrak{S}_{i,j}}
\Big\{ (q(i)-q(j))\, \widetilde{\Phi}_{N}^{\boldsymbol{\sigma}}
\,+\,
\sum_{\boldsymbol{z}\in\mathring{\mathcal{B}}_{N}^{\boldsymbol{\sigma}}}
\Phi_{q(j),\boldsymbol{z}}^{*}\Big\}\;.
\end{equation*}
Recall that the flow $\Phi_{c,\boldsymbol{z}}^{*}$ is obtained by
inserting the constant function $f\equiv c$ into
$\Phi_{f,\boldsymbol{z}}^{*}$, the flow defined at the beginning of
Section \ref{sec:flow}.

\begin{lemma}
\label{con73}
The flow $\Upsilon_{N}^{q}$ is divergence-free on
$\{\boldsymbol{m}_{N}^{i}:i\in S\}^c$ and for each $i\in S$, 
\begin{equation*}
(\textup{div }\Upsilon_{N}^{q})(\boldsymbol{m}_{N}^{i})
\;=\;-\kappa_{N}\, \mu(i)\, (\mathcal{L}_{Y}q) (i)\;.
\end{equation*}
Furthermore, $\Upsilon_{N}^{q}-\Phi_{h_{N}^{q}}^{*}$ is a
negligible flow.
\end{lemma}

\begin{proof}
For $c\in\mathbb{R}$ and $\boldsymbol{z}\in\widehat{\Xi}_{N}$, the
flow $\Phi_{c,\boldsymbol{z}}^{*}$ is cyclic and thus
divergence-free. Hence, the divergence of $\Upsilon_{N}^{q}$ is equal
to the divergence of the flow
\begin{equation*}
\sum_{1\le i<j\le M}\, \sum_{\boldsymbol{\sigma}\in\mathfrak{S}_{i,j}}\,
\sum_{\boldsymbol{z}\in\mathring{\mathcal{B}}_{N}^{\boldsymbol{\sigma}}}
\big[q(i)-q(j)\big]\, \widetilde{\Phi}_{N}^{\boldsymbol{\sigma}}\;.
\end{equation*}
Hence, the first assertion of the lemma follows from Theorem \ref{s24}
and the definitions of $\mathcal{L}_Y$ and $\mu$.

It remains to show that the flow $\Upsilon_{N}^{q} -
\Phi_{h_{N}^{q}}^{*}$ is negligible.  As in the proof of Lemma
\ref{s22}, rewrite this flow as
\begin{equation}
\label{fl1}
\sum_{1\le i<j\le M}\sum_{\boldsymbol{\sigma}\in\mathfrak{S}_{i,j}}
\sum_{\boldsymbol{z}\in\mathring{\mathcal{B}}_{N}^{\boldsymbol{\sigma}}}
(q(i)-q(j))\left(\widetilde{\Phi}_{N}^{\boldsymbol{\sigma}}
-\Phi_{N}^{\boldsymbol{\sigma}}\right) \;+\; \varDelta_{N}\;,
\end{equation}
where $\varDelta_{N} = \widehat{\Upsilon}_{N}^{q} -
\Phi_{h_{N}^{q}}^{*}$, with $\widehat{\Upsilon}_{N}^{q}$ being the
flow obtained from $\Upsilon_{N}^{q}$ by replacing
$\widetilde{\Phi}_{N}^{\boldsymbol{\sigma}}$ by
$\Phi_{N}^{\boldsymbol{\sigma}}$. By Theorem \ref{s24}, the first flow
in \eqref{fl1} is negligible. As in the proof of Lemma
\ref{s22}, the flow $\varDelta_{N}$ is negligible because the
discrepancies between $\Upsilon_{N}^{q}$ and $\Phi_{h_{N}^{q}}^{*}$
appear only at the boundaries of $\mathcal{W}_{N}^{i}$, $i\in S$, and
of $\mathcal{B}_{N}^{\boldsymbol{\sigma}}$,
$\boldsymbol{\sigma\in\mathfrak{S}}$.
\end{proof}

\subsection{Proof of Theorem \ref{s133}} 
\label{sub_pt}

In view of the remark below Lemma \ref{con71}, choose $q$ as the
equilibrium potential between $A$ and $B$, denoted by $q_{A,B}$. From
now on, we write $h_{N}^{q_{A,B}}$, $\Upsilon_{N}^{q_{A,B}}$ simply as
$h_{N}^{A,B}$, $\Upsilon_{N}^{A,B}$, respectively.

\begin{lemma}
The function $h_{N}^{A,B}$ satisfies
\begin{equation*}
\mathcal{D}_{N}(h_{N}^{A,B})\;=\;\big[1+o_{N}(1)\big]\, \kappa_{N}
\, \textup{cap}_{Y}(A,B)\;. 
\end{equation*}
Moreover, the flow $\Upsilon_{N}^{A,B}$ is divergence-free on
$(\mathcal{E}_{N}(A)\cup\mathcal{E}_{N}(B))^{c}$ and
\begin{equation*}
(\textup{div }\Upsilon_{N}^{A,B})(\mathcal{E}_{N}(A)) 
\;=\;-\, (\textup{div }\Upsilon_{N}^{A,B})(\mathcal{E}_{N}(B))
\;=\; \kappa_{N}\, \mbox{\textup{cap}}_{Y}(A,B)\;. 
\end{equation*}
\end{lemma}

\begin{proof}
The first part follows from \eqref{qq1} and Lemma \ref{con71},
and the second part is a consequence of \eqref{qq2}, \eqref{qq3}
and Lemma \ref{con73}.
\end{proof}

\begin{proof}[Proof of Theorem \ref{s133}]
The upper bound is achieved by using 
\begin{align*}
&f_{N}^{A,B}  \;=\;  \frac{1}{2}(h_{N}^{*,A,B}+h_{N}^{A,B})
\;\in\;\mathfrak{C}_{1,0}(\mathcal{E}_{N}(A),\mathcal{E}_{N}(B))\;,\\
&\quad\phi_{N}^{A,B} \;=\;  \frac{1}{2}(\Upsilon_{N}^{*,A,B}
-\Upsilon_{N}^{A,B})\;\in\;
\mathfrak{U}_{0}(\mathcal{E}_{N}(A),\mathcal{E}_{N}(B))\;,
\end{align*}
as the test function and the test flow in Theorem \ref{s07}. On
the other hand, the lower bound can be proven by using
\begin{align*}
&g_{N}^{A,B}\;=\; \frac{1}{\kappa_{N}\mbox{\textup{cap}}_{Y}(A,B)}
\,\frac{h_{N}^{*,A,B}-h_{N}^{A,B}}{2}\;\in\;
\mathfrak{C}_{0,0}(\mathcal{E}_{N}(A),\mathcal{E}_{N}(B))\;,\\
&\quad\psi_{N}^{A,B}  \;=\; \frac{1}{\kappa_{N}
\mbox{\textup{cap}}_{Y}(A,B)}\,
\frac{\Upsilon_{N}^{*,A,B}+\Upsilon_{N}^{A,B}}{2}\;
\in\;\mathfrak{U}_{1}(\mathcal{E}_{N}(A),\mathcal{E}_{N}(B))\;,
\end{align*}
as the test function and flow in Theorem \ref{s10}. The details
of the proof are identical to the ones of Theorem \ref{s13}. 
\end{proof}

\section{Mean Jump Rate}
\label{sec8}

In this section we compute the mean jump rates of the process. Fix
$1\le m\le l$, and write $\mathcal{E}_{N}^{(m)} =
\mathcal{E}_{N}(S_{m}) = \cup_{i\in S_m} \mc E^i_N$. Denote by $\bb
T_m(t)$, $t\ge 0$, the amount of time the process $X_N(s)$ remains in
the set $\mc E _{N}^{(m)}$ in the interval $[0,t]$:
\begin{equation*}
\bb T_m(t) \;:=\;\int_{0}^t \mb {1} \big \{X_N(s) \in
\mc E^{(m)}_N \big\} \, ds\,,\quad t\ge 0\,.
\end{equation*}
Let $\bb S_m(t)$ be the generalized inverse of $\bb T_m(t)$:
\begin{equation*}
\bb S_m(t) \;:=\;\sup\{s\ge 0 :  \bb T_m(s) \le t\}\; .
\end{equation*}
The \emph{trace} of the process $X_N(t)$ on the set $\mc E^{(m)}_N$,
denoted by $X^{(m)}_N(t)$, is defined as
\begin{equation*}
X^{(m)}_N(t) \;=\; X_N(\bb S_m(t))\;.
\end{equation*}

The process $X^{(m)}_N(t)$ is an $\mc E^{(m)}_N$-valued Markov
process. Denote by $R^{(m)}_N(\bs x, \bs y)$, $\bs x\not = \bs y\in
\mc E^{(m)}_N$, the jump rates of this process which can be expressed
in terms of hitting times. Let $\bs r^{(m)}_N(i,j)$, $i\not = j\in
S_m$, be the mean jump rates from the well $\mathcal{E}_{N}^{i}$ to
the valley $\mathcal{E}_{N}^{j}$:
\begin{equation}
\label{48}
r_{N}^{(m)}(i,j)\;=\;\frac{1}{\mu_{N}(\mathcal{E}_{N}^{i})}
\,\sum_{\boldsymbol{x}\in\mathcal{E}_{N}^{i}}\,
\sum_{\boldsymbol{y}\in\mathcal{E}_{N}^{j}}
\mu_{N}(\boldsymbol{x}) \, R_{N}^{(m)}(\boldsymbol{x},\boldsymbol{y})\;.
\end{equation}
We define $r_{N}^{(m)}(i,i)=0$ just for convenience. 

Denote by $\lambda_{N}^{(m)}(i)$, $i\in S_{m}$, the mean holding
time at $\mathcal{E}_{N}^{i}$ for the trace process $X_{N}^{(m)}:$
\begin{equation*}
\lambda_{N}^{(m)}(i)\;=\;\sum_{j\in S_{m}}r^{(m)}_{N}(i,j)\;.
\end{equation*}
By \cite[display (A.8)]{BL2}, 
\begin{equation*}
\mu_{N}(\mathcal{E}_{N}^{i}) \, \lambda_{N}^{(m)}(i)
\;=\;\mbox{\textup{cap}}_{N}(\mathcal{E}_{N}^{i},
\mathcal{E}_{N}^{(m)}\setminus\mathcal{E}_{N}^{i})\;,
\end{equation*}
so that, by \eqref{thm: estimate of meta mass} and Theorem \ref{s133},
\begin{equation}
\label{10}
\lambda_{N}^{(m)}(i)\;=\;\big[1+o_{N}(1)\big]\,
\frac{e^{-N(H-h_{i})}}{2\pi N}\,\frac{1}{\nu_{i}}\,
\textup{cap}_{Y}(\{i\},S_{m}\setminus\{i\})\;.
\end{equation}

The main result of this section provides a sharp estimate for the
mean jump rates of the trace process $X_{N}^{(m)}(t)$. Recall the
definition of $c_{m}(\cdot,\cdot)$ from \eqref{c_m}.

\begin{theorem}
\label{s04}For $1\le m\le l$ and $i\neq j\in S_{m}$, 
\begin{equation*}
r_{N}^{(m)}(i,j)\;=\;\big[1+o_{N}(1)\big]\,
\frac{e^{-N(H-h_{i})}}{2\pi N}\,\frac{c_{m}(i,j)}{\nu_{i}}\;.
\end{equation*}
\end{theorem}

Without loss of generality, we assume that $1,\,2\in S_{m}$ and then
prove Theorem \ref{s04} for $(i,j)=(1,2)$. We also assume that
$\textup{cap}_{Y}(\{1\},S_{m}\setminus\{1\})\neq0$ because if this is
not the case, Theorem \ref{s04} is a direct consequence of \eqref{10}.

\subsection{Collapsed Process}
\label{81}

We present in this subsection some general results on collapsed
processes needed to prove Theorem \ref{s04}.  To avoid
introducing new notation, we present all results in the context of the
$\Xi_N$-valued Markov chain $X_N(t)$, but all assertions of this
subsection hold for general continuous-time Markov chains.

Fix a point $\mf o\not \in \Xi_N$ and let $\overline{\Xi}_{N}$ be the
set in which the valley $\mathcal{E}_{N}^{1}$ is collapsed to a point
$\mf{o}$: $\overline{\Xi}_N = (\Xi_{N}\setminus \mathcal{E}_{N}^{1})
\cup \{\mathfrak{o}\}$.  Recall from \eqref{11} that we denote by
$R_{N}$ the jump rates of the chain $X_N(t)$. Let $\overline{X}_{N}
(t)$ be the $\overline{\Xi}_{N}$-valued Markov chain whose jump rates
are given by
\begin{equation}
\label{eq:coll rate}
\begin{cases}
\;\overline{R}_{N}(\boldsymbol{x},\boldsymbol{y})\;=\;
R_{N}(\boldsymbol{x},\boldsymbol{y})\;, & 
\boldsymbol{x},\,\boldsymbol{y}\in\Xi_{N}\setminus\mathcal{E}_{N}^{1}\;,
\\
\;\overline{R}_{N}(\boldsymbol{x},\mathfrak{o})\;=\;
\sum_{\boldsymbol{z}\in\mathcal{E}_{N}^{1}}R_{N}(\boldsymbol{x},\boldsymbol{z}) 
\;,& \boldsymbol{x}\in\Xi_{N}\setminus\mathcal{E}_{N}^{1}\;, \\
\;\overline{R}_{N}(\mathfrak{o},\boldsymbol{y})\;=\;
\big[\mu_{N}(\mathcal{E}_{N}^{1})\big]^{-1}\sum_{\boldsymbol{z}\in\mathcal{E}_{N}^{1}}
\mu_{N}(\boldsymbol{z})\, R_{N}(\boldsymbol{z},\boldsymbol{y}) \;, &
\boldsymbol{y}\in\Xi_{N}\setminus\mathcal{E}_{N}^{1}\;.
\end{cases}
\end{equation}
Denote by $\overline{\mathbb{P}}_{\boldsymbol{x}}^{N}$ the law of
$\overline{X}_{N}$ starting from $\boldsymbol{x}$, and by
$\mathcal{\overline{L}}_{N}$, $\overline{\mathcal{D}}_{N}(\cdot)$,
$\overline{\textup{cap}}_N(\cdot,\cdot)$ the generator, the Dirichlet
form and the capacity, respectively, corresponding to the collapsed
process. One can easily verify that the invariant measure
$\overline{\mu}_{N}(\cdot)$ for $\overline{X}_{N}(t)$ is given by
\begin{equation}
\label{eq:coll_inv}
\overline{\mu}_{N}(\boldsymbol{x})\;=\;\mu_{N}(\boldsymbol{x}) 
\;,\;\; \boldsymbol{x}\in\Xi_{N}\setminus\mathcal{E}_{N}^{1}\;, \quad
\overline{\mu}_{N}(\mathfrak{o})\;=\;\mu_{N}(\mathcal{E}_{N}^{1})\;.
\end{equation}

Denote by $\overline{c}(\boldsymbol{x},\boldsymbol{y})$ the
conductances of the chain $\overline{X}_{N}$:
$$
\overline{c}(\boldsymbol{x},\boldsymbol{y})
=\overline{\mu}_{N} (\boldsymbol{x}) \overline{R}_{N}
(\boldsymbol{x},\boldsymbol{y})\;.
$$
In view of the previous relations,
\begin{equation}
\label{eq:coll conduct}
\begin{cases}
\;\overline{c}(\boldsymbol{x},\boldsymbol{y})\;=\;
c(\boldsymbol{x},\boldsymbol{y}) \;, & \boldsymbol{x},\,
\boldsymbol{y}\in\Xi_{N}\setminus\mathcal{E}_{N}^{1}\;, \\
\;\overline{c}(\boldsymbol{x},\mathfrak{o})\;=\;
\sum_{\boldsymbol{z}\in\mathcal{E}_{N}^{1}}
c(\boldsymbol{x},\boldsymbol{z}) \;, & 
\boldsymbol{x}\in\Xi_{N}\setminus\mathcal{E}_{N}^{1}\;,\\
\;\overline{c}(\mathfrak{o},\boldsymbol{y})\;=\;
\sum_{\boldsymbol{z}\in\mathcal{E}_{N}^{1}}
c(\boldsymbol{z},\boldsymbol{y}) \;, & \boldsymbol{y}\in
\Xi_{N}\setminus\mathcal{E}_{N}^{1}\;.
\end{cases}
\end{equation}
The symmetrized conductance is defined by
$\overline{c}^{s}(\boldsymbol{x}, \boldsymbol{y})=
[\overline{c}(\boldsymbol{x}, \boldsymbol{y}) + \overline{c}
(\boldsymbol{y}, \boldsymbol{x})]/2$ for
$\boldsymbol{x},\,\boldsymbol{y}\in\overline{\Xi}_{N}$. Let
$\overline{E}_{N} = \{(\boldsymbol{x},\boldsymbol{y}):
\overline{c}^{s} (\boldsymbol{x},\boldsymbol{y})>0\}$ be the set of
edges and let $\overline{\mathcal{F}}_{N}$ be the set of flows on
$\overline{E}_{N}$ endowed with a scalar product analogous to the one
introduced in Section \ref{sec:Preliminaries}.  Denote the scalar
product and the norm by $\left\langle \cdot,\cdot\right\rangle
_{\mathcal{C}}$ and $||\cdot||_{\mathcal{C}}$, respectively.

For each flow $\phi\in\mathcal{F}_{N}$, define the collapsed flow
$\overline{\phi}\in\overline{\mathcal{F}}_{N}$ by
\begin{equation}
\label{eq:coll_flow}
\begin{cases}
\;\overline{\phi}(\boldsymbol{x},\boldsymbol{y})\;=\;
\phi(\boldsymbol{x},\boldsymbol{y}) \;, & 
\boldsymbol{x},\,\boldsymbol{y}\in\Xi_{N}\setminus\mathcal{E}^1_{N}\;, \\
\;\overline{\phi}(\boldsymbol{x},\mathfrak{o})\;=\;
\sum_{\boldsymbol{z}\in\mathcal{E}_{N}^{1}}
\phi(\boldsymbol{x},\boldsymbol{z}) \;, & 
\boldsymbol{x}\in\Xi_{N}\setminus\mathcal{E}^1_{N}\;, \\
\;\overline{\phi}(\mathfrak{o},\boldsymbol{y})\;=\;
\sum_{\boldsymbol{z}\in\mathcal{E}_{N}^{1}}
\phi(\boldsymbol{z},\boldsymbol{y}) \;, & 
\boldsymbol{y}\in\Xi_{N}\setminus\mathcal{E}^1_{N}\;.
\end{cases}
\end{equation}
Clearly, $\overline{\phi}$ inherits from $\phi$ the anti-symmetry.
Moreover,
\begin{equation}
\label{18}
(\mbox{div }\overline{\phi})(\boldsymbol{x})\;=\;
(\mbox{div }\phi)(\boldsymbol{x})\;,\;\; 
\boldsymbol{x}\in\Xi_{N}\setminus\mathcal{E}_{N}^{1}\;,\quad
(\mbox{div }\overline{\phi})(\mathfrak{o})\;=\;
(\mbox{div }\phi)(\mathcal{E}_{N}^{1})\;.
\end{equation}

\begin{lemma}
\label{lem82}
For every flow $\phi\in\mathcal{F}_{N}$, $\left\Vert
  \overline{\phi}\right\Vert _{\mathcal{C}}\le\left\Vert
  \phi\right\Vert $.
\end{lemma}

\begin{proof}
Decompose the flow norm of $\phi$ as $\left\Vert \phi\right\Vert
^{2}= (1/2) (A_{1}+A_{2}+A_{3})$ where 
\begin{align*}
&A_{1}\;=\;
\sum_{\boldsymbol{x},\,\boldsymbol{y}\in\Xi_{N}\setminus\mathcal{E}_{N}^{1}}
\frac{\phi(\boldsymbol{x},\boldsymbol{y})^{2}}
{c^{s}(\boldsymbol{x},\boldsymbol{y})}\;, \quad 
A_{2} \;=\; 2\, \sum_{\boldsymbol{x}\notin\mathcal{E}^1_{N}}
\sum_{\boldsymbol{y}\in\mathcal{E}^1_{N}} \frac{\phi(\boldsymbol{x},\boldsymbol{y})^{2}}
{c^{s}(\boldsymbol{x},\boldsymbol{y})} \;,\\
&\qquad A_{3} \;=\;
\sum_{\boldsymbol{x},\,\boldsymbol{y}\in\mathcal{E}_{N}^{1}}
\frac{\phi(\boldsymbol{x},\boldsymbol{y})^{2}}
{c^{s}(\boldsymbol{x},\boldsymbol{y})}\;,
\end{align*}
and the flow norm of $\overline{\phi}$ as $\left\Vert
  \overline{\phi}\right\Vert
_{\mathcal{C}}^{2}= (1/2) (\overline{A}_{1}+\overline{A}_{2})$
where
\begin{equation*}
\overline{A}_{1} \;=\;
\sum_{\boldsymbol{x},\,\boldsymbol{y}\notin\mathcal{E}^1_{N}}
\frac{\overline{\phi}(\boldsymbol{x},\boldsymbol{y})^{2}}
{\overline{c}^{s}(\boldsymbol{x},\boldsymbol{y})}\;, \quad 
\overline{A}_{2} \;=\; 2\, \sum_{\boldsymbol{x}\notin\mathcal{E}^1_{N}}
\frac{\overline{\phi}(\boldsymbol{x},\mathfrak{o})^{2}}
{\overline{c}^{s}(\boldsymbol{x},\mathfrak{o})}\;\cdot
\end{equation*}
The previous sums are all carried over bonds $(\bs x, \bs y)$ for
which $c^s(\bs x, \bs y)$ and $\overline{c}^{s}(\bs x, \bs y)$ are
strictly positive.  By (\ref{eq:coll conduct}) and
(\ref{eq:coll_flow}), it is clear that
$A_{1}=\overline{A}_{1}$. Thereby, to complete the proof, it suffices
to prove $A_{2}\ge\overline{A_{2}}$.  For each $\boldsymbol{x}$
adjacent to at least one point of $\mathcal{E}_{N}^{1}$, by the
Cauchy-Schwarz inequality,
\begin{equation*}
\sum_{\boldsymbol{y}\in\mathcal{E}^1_{N}}
\frac{\phi(\boldsymbol{x},\boldsymbol{y})^{2}}
{c^{s}(\boldsymbol{x},\boldsymbol{y})} \;\ge\; 
\frac{\Big(\sum_{\boldsymbol{y}\in\mathcal{E}^1_{N}} 
\phi(\boldsymbol{x},\boldsymbol{y})\Big)^{2}}
{\sum_{\boldsymbol{y}\in\mathcal{E}^1_{N}}
c^{s}(\boldsymbol{x},\boldsymbol{y})}\;=\;
\frac{\overline{\phi}(\boldsymbol{x},\mathfrak{o})^{2}}
{\overline{c}^{s}(\boldsymbol{x},\mathfrak{o})}\;.
\end{equation*}
By adding these inequality over
$\boldsymbol{x}\notin\mathcal{E}_{N}^{1}$ such that $(\bs x,
\mathfrak{o})\in \overline{E}_N$, we obtain that
$A_{2}\ge\overline{A}_{2}$.
\end{proof}

If a function $f:\Xi_{N}\rightarrow\mathbb{R}$ is constant over the
set $\mathcal{E}_{N}^{1}$, it is possible to collapse it to a function
$\overline{f}:\overline{\Xi}_{N}\rightarrow\mathbb{R}$ by setting
$\overline{f}(\mathfrak{o}) = f(\boldsymbol{z})$ for some $\bs z$ in
$\mathcal{E}_{N}^{1}$:
\begin{equation}
\label{eq: col f}
\overline{f}(\boldsymbol{x})\;=\;f(\boldsymbol{x})\;, \;
\boldsymbol{x}\in\Xi_{N}\setminus\mathcal{E}_{N}^{1}\;, \quad
\overline{f}(\mathfrak{o})\;= f(\boldsymbol{z}) \text{ for any } \bs
z \in \mathcal{E}_{N}^{1}\;.
\end{equation}
For a function $g:\overline{\Xi}_{N}\rightarrow\mathbb{R}$, denote by
$\overline{\Phi}_{g},\,\overline{\Phi}_{g}^{*}$ and
$\overline{\Psi}_{g}$, the flows in $\overline{\mathcal{F}}_{N}$
defined by
\begin{align*}
&\overline{\Phi}_{g}(\boldsymbol{x},\boldsymbol{y}) \;=\;
g(\boldsymbol{x})\,\overline{c}(\boldsymbol{x},\boldsymbol{y})
-g(\boldsymbol{y})\,\overline{c}(\boldsymbol{y},\boldsymbol{x})\;,\\
&\quad\overline{\Phi}_{g}^{*}(\boldsymbol{x},\boldsymbol{y}) \;=\;
g(\boldsymbol{x})\,\overline{c}(\boldsymbol{y},\boldsymbol{x})
-g(\boldsymbol{y})\,\overline{c}(\boldsymbol{x},\boldsymbol{y})\;,
\\
&\quad\quad \overline{\Psi}_{g}(\boldsymbol{x},\boldsymbol{y}) \;=\; 
\overline{c}^{s}(\boldsymbol{x},\boldsymbol{y})(g(\boldsymbol{x})-
g(\boldsymbol{y}))\;.
\end{align*}

\begin{lemma}
\label{s08}
Suppose that the function $f:\Xi_{N}\rightarrow\mathbb{R}$ is constant
over the set $\mathcal{E}_{N}^{1}$. Then, the flow obtained by
collapsing the flow $\Phi_{f}$, denoted by $\overline{\Phi_{f}}$,
coincides with the flow $\overline{\Phi}_{\,\overline{f}}$. The same
result holds for the flows $\Phi^{*}$ and $\Psi$.
\end{lemma}

\begin{proof}
It suffices to check that these flows coincide on the edges of the
form $(\boldsymbol{x},\mathfrak{o})$,
$\boldsymbol{x}\notin\mathcal{E}_{N}^{1}$. Indeed,
\begin{align*}
& \overline{\Phi_{f}}(\boldsymbol{x},\mathfrak{o})\;=\;
\sum_{\boldsymbol{z}\in\mathcal{E}_{N}^{1}}\Phi_{f}(\boldsymbol{x},\boldsymbol{z})
\;=\; \sum_{z\in\mathcal{E}_{N}^{1}} \big[f(\boldsymbol{x})
c(\boldsymbol{x},\boldsymbol{z})-f(\boldsymbol{z})
c(\boldsymbol{z},\boldsymbol{x})\, \big]\\
&\qquad =\; \overline{f}(\boldsymbol{x})\,
\overline{c}(\boldsymbol{x},\mathfrak{o}) -
\overline{f}(\mathfrak{0})\, \overline{c}(\mathfrak{o},\boldsymbol{x})
\;=\; \overline{\Phi}_{\overline{f}}(\boldsymbol{x},\mathfrak{o})\;.
\end{align*}
The proofs for $\Phi^{*}$ and $\Psi$ are analogous. 
\end{proof}

\subsection{Mean Jump Rates}
\label{sub:MJR_comp}

Recall from \cite[Proposition 4.2]{BL2} that 
\begin{equation}
\frac{r_{N}^{(m)}(1,2)}{\lambda_{N}^{(m)}(1)}\;=\;\overline{\mathbb{P}}_{\mathfrak{o}}^{N}\left[H_{\mathcal{E}_{N}^{2}}<H_{\mathcal{E}_{N}(S_{m}\setminus\{1,2\})}\right]\;.\label{12}
\end{equation}
In particular, in view of \eqref{10}, the asymptotic analysis of
the mean jump rate $r_{N}^{(m)}(1,2)$ is reduced to the one of the
right hand side of this equation. The following proposition provides
this sort of analysis.

\begin{proposition}
\label{s05}
For disjoint subsets $A,\,B$ of $S\setminus\{1\}$ satisfying
$\textup{cap}_{Y}(\{1\},A\cup B)\neq 0$, we have that
\begin{equation*}
\lim_{N\rightarrow\infty}\overline{\mathbb{P}}_{\mathfrak{o}}^{N}\left[H_{\mathcal{E}_{N}(A)}<H_{\mathcal{E}_{N}(B)}\right]\;=\;q_{A,B}(1)\;.
\end{equation*}
\end{proposition}
We divide the proof of Proposition \ref{s05} into several steps.  The
first step is to compute the capacity between $\mathcal{E}_{N}(A)$ and
$\mathcal{E}_{N}(B)$ with respect to the collapsed chain. Recall from
Section \ref{sub_pt} the notations $h_{N}^{A,B}$, $f_{N}^{A,B}$,
$\phi_{N}^{A,B}$, $g_{N}^{A,B}$ and $\psi{}_{N}^{A,B}$. Note that
$h_{N}^{A,B}$, $f_{N}^{A,B}$ and $g_{N}^{A,B}$ are constant on
$\mathcal{E}_{N}^{1}$ and therefore we can collapse them as in the
previous subsection. These collapsed functions are denoted
respectively by $\overline{h}_{N}^{A,B}$, $\overline{f}_{N}^{A,B}$ and
$\overline{g}_{N}^{A,B}$. Note that by the definitions, we have that
\begin{equation}
\overline{h}_{N}^{A,B}(\mathfrak{o})\;=\;
\overline{f}_{N}^{A,B}(\mathfrak{o})\;=\;q_{A,B}(1)\;.\label{int1}
\end{equation}
Recall that $\overline{\textup{cap}}_N(\cdot,\cdot)$ represents the
capacity with respect to the collapsed process $\overline{X}_N(t)$.

\begin{lemma}
\label{p002}
We have that 
\begin{equation}
\overline{\textup{cap}}_{N}(\mathcal{E}_{N}(A),\mathcal{E}_{N}(B))
\;=\;\big[1+o_{N}(1)\big]\kappa_{N}\textup{cap}_{Y}(A,B)\label{re1}
\end{equation}
and
\begin{equation}
\left\Vert
  \overline{\Psi}_{\overline{h}_{N}^{A,B}}\right\Vert_{\mathcal{C}}^{2}
\;=\;\big[1+o_{N}(1)\big]\kappa_{N}\textup{cap}_{Y}(A,B)\;.\label{re2}
\end{equation}
\end{lemma}

\begin{proof}
As a by-product of the proof of Theorem \ref{s133}, we obtain that
\begin{equation}
\left\Vert \Phi_{f_{N}^{A,B}}-\phi_{N}^{A,B}\right\Vert ^{2}\;=\;\big[1+o_{N}(1)\big]\,\kappa_{N}\,\textup{cap}_{Y}(A,B)\;.\label{col1}
\end{equation}
For $a,\,b\in\mathbb{R}$, denote by
$\overline{\mathfrak{C}}_{a,b}(\cdot,\cdot)$ and
$\overline{\mathfrak{U}}_{a}(\cdot,\cdot)$ the collapsed versions of
the sets $\mathfrak{C}_{a,b}(\cdot,\cdot)$ and
$\mathfrak{U}_{a}(\cdot,\cdot)$ introduced in \eqref{eq:C_ab},
\eqref{eq:U_ab}.  Then, it is easy to check that
$\overline{f}_{N}^{A,B}$ belongs to
$\overline{\mathfrak{C}}_{1,0}(\mathcal{E}_{N}(A),\mathcal{E}_{N}(B))$.
By \eqref{18}, we can also verify that $\overline{\phi}_{N}^{A,B}$
belongs to
$\overline{\mathfrak{U}}_{0}(\mathcal{E}_{N}(A),\mathcal{E}_{N}(B))$.
Hence, by Theorem \ref{s07}, Lemmata \ref{lem82}, \ref{s08} and
\eqref{col1}, we obtain that
\begin{equation}
\label{col_u}
\overline{\textup{cap}}_{N}(\mathcal{E}_{N}(A),
\mathcal{E}_{N}(B))\;\le\;\left\Vert 
\overline{\Phi}_{\overline{f}_{N}^{A,B}}-
\overline{\phi}_{N}^{A,B}\right\Vert_{\mathcal{C}}^{2}
\;\le\;\big[1+o_{N}(1)\big]\,\kappa_{N}\,\textup{cap}_{Y}(A,B)\;.
\end{equation}

On the other hand, again by the proof of Theorem \ref{s133}, 
\begin{equation*}
\frac{1}{\left\Vert \Phi_{g_{N}^{A,B}}-\psi_{N}^{A,B}\right\Vert ^{2}}\;=\;\big[1+o_{N}(1)\big]\,\kappa_{N}\,\textup{cap}_{Y}(A,B)\;.
\end{equation*}
By the similar argument as above, we obtain that 
\begin{equation}
\overline{\textup{cap}}_{N}(\mathcal{E}_{N}(A),\mathcal{E}_{N}(B))\;\ge\;\frac{1}{\left\Vert \overline{\Phi}_{\overline{g}_{N}^{A,B}}-\overline{\psi}_{N}^{A,B}\right\Vert_{\mathcal{C}}^{2}}\;\ge\;(1+o_{N}(1))\kappa_{N}\textup{cap}_{Y}(A,B)\;.\label{col_l}
\end{equation}
Now, \eqref{re1} is a direct consequence of \eqref{col_u} and \eqref{col_l}.

Now we prove \eqref{re2}. By the definitions of $f_{N}^{A,B}$, $\phi_{N}^{A,B}$
and Lemma \ref{con73}, we can write $\Phi_{f_{N}^{A,B}}-\phi_{N}^{A,B}=\Psi_{h_{N}^{A,B}}+R_{N}$
where $R_{N}$ is a negligible flow. By Lemma \ref{s08}, collapsing
this relation we obtain that 
\begin{equation}
\overline{\Phi}_{\overline{f}_{N}^{A,B}}-\overline{\phi}_{N}^{A,B}\;=\;\overline{\Psi}_{\overline{h}_{N}^{A,B}}+\overline{R}_{N}\label{23}
\end{equation}
By Lemma \ref{lem82}, the flow $\overline{R}_{N}$ inherits
the negligibility from $R_{N}$. On the other hand, we obtain from
\eqref{re1} and \eqref{col_u} that 
\begin{equation*}
\left\Vert \Phi_{\overline{f}_{N}^{A,B}}-\overline{\phi}_{N}^{A,B}\right\Vert_{\mathcal{C}}^{2}\;=\;\big[1+o_{N}(1)\big]\,\kappa_{N}\,\textup{cap}_{Y}(A,B)\;.
\end{equation*}
The second assertion of the lemma follows from the two previous displayed
formulas, from the fact that $\overline{R}_{N}$ is negligible, and
from Lemma \ref{lem:neg_lemma}.
\end{proof}

The next two results will be used in the proof of Proposition
\ref{s05}.  First, by \cite[display (3.10)]{GL}, for any
$A\subset\Xi_{N}\setminus\mathcal{E}_{N}^{1}$,
\begin{equation}
\label{20}
\overline{\mbox{\textup{cap}}}_{N} (\mathfrak{o},A) \;=\;
\textup{cap}_{N}(\mathcal{E}_{N}^{1},A)\;.
\end{equation}

\begin{lemma}
\label{lem:col_sector}
For two disjoint subsets $A$, $B$ of $\overline{\Xi}_{N}$,
\begin{equation*}
\overline{\textup{cap}}_{N}(A,B) \;\le\;
4\, L^{2}\, \overline{\textup{cap}}^{s}_{N}(A,B) \;,
\end{equation*}
where $\overline{\textup{cap}}_{N}^{s}(\cdot,\cdot)$ stands for the
capacity with respect to $\overline{\mathcal{L}}_{N}^{s} = (1/2)
(\overline{\mathcal{L}}_{N} + \overline{\mathcal{L}}_{N}^{*})$.
\end{lemma}

\begin{proof}
By \cite[Lemma 2.6]{GL}, it suffices to prove that the generator
$\overline{\mathcal{L}}_{N}$ inherits the sector condition from $\mathcal{L}_{N}$.
Fix two functions $f,\,g:\overline{\Xi}_{N}\rightarrow\mathbb{R}$, and
define their extensions $F,\,G:\Xi_{N}\rightarrow\mathbb{R}$ by
\begin{equation*}
F(\boldsymbol{x}) \; = \; \begin{cases}
f(\boldsymbol{x}) & \mbox{if }\boldsymbol{x}\notin\mathcal{E}_{N}^{1}\\
f(\mathfrak{o}) & \mbox{if }\boldsymbol{x}\in\mathcal{E}_{N}^{1}
\end{cases}\;\;\;\mbox{and\;\;\;}G(\boldsymbol{x})\;=\;\begin{cases}
g(\boldsymbol{x}) & \mbox{if }\boldsymbol{x}\notin\mathcal{E}_{N}^{1}\\
g(\mathfrak{o}) & \mbox{if }\boldsymbol{x}\in\mathcal{E}_{N}^{1}\;.
\end{cases}
\end{equation*}
By \cite[display (3.8)]{GL} and by Lemma \ref{lem: sector condition},
\begin{align*}
\langle f,(-\mathcal{\overline{L}}_{N})\, g \rangle^2 _{\overline{\mu}_{N}}
\; &=\; \langle F,(-\mathcal{L}_{N})\, G\rangle^2 _{\mu_{N}} \;\le\; 
4\,L^{2}\, \langle F,(-\mathcal{L}_{N})\, F \rangle _{\mu_{N}} \,
\langle G, (-\mathcal{L}_{N}) \, G \rangle _{\mu_{N}}\\
 & =\;4\, L^{2} \langle f, (-\mathcal{\overline{L}}_{N})\, f\rangle
 _{\overline{\mu}_{N}} \, 
\langle g, (-\mathcal{\overline{L}}_{N})\, g\rangle _{\overline{\mu}_{N}}\;.
\end{align*}
\end{proof}

\begin{proof}[Proof of Proposition \ref{s05}] Denote by $U_{N}^{A,B}$
the equilibrium potential between $\mathcal{E}_{N}(A)$ and $\mathcal{E}_{N}(B)$
for the collapsed process. Hence, by Lemma \ref{p002}, 
\begin{equation}
\left\Vert \overline{\Psi}_{U_{N}^{A,B}}\right\Vert _{\mathcal{C}}^{2}\;=\;\overline{\textup{cap}}_{N}(\mathcal{E}_{N}(A),\mathcal{E}_{N}(B))\;=\;(1+o_{N}(1))\kappa_{N}\textup{cap}_{Y}(A,B)\;.\label{24}
\end{equation}

On the other hand, since $\overline{\phi}_{N}^{A,B}$ belongs to $\overline{\mathfrak{U}}_{0}(\mathcal{E}_{N}(A),\mathcal{E}_{N}(B))$,
\begin{equation*}
\left\langle \overline{\phi}_{N}^{A,B},\overline{\Psi}_{U_{N}^{A,B}}\right\rangle_{\mathcal{C}} \;=\;0\;.
\end{equation*}
Therefore, by \eqref{22}, by the fact that $\overline{\mathcal{L}}_{N}U_{N}=0$
on $(\mathcal{E}_{N}(A)\cup\mathcal{E}_{N}(B))^{c}$ and that $\overline{f}_{N}$
and $U_{N}$ coincide on $\mathcal{E}_{N}(A)\cup\mathcal{E}_{N}(B)$,
\begin{align*}
 & \left\langle \overline{\Phi}_{\overline{f}_{N}^{A,B}}-\overline{\phi}_{N}^{A,B},\,\overline{\Psi}_{U_{N}^{A,B}}\right\rangle _{\mathcal{C}}\;=\;\left\langle \overline{f}_{N}^{A,B},\,(-\overline{\mathcal{L}}_{N})U_{N}^{A,B}\right\rangle _{\overline{\mu}_{N}}\;=\;\left\langle U_{N}^{A,B},\,(-\overline{\mathcal{L}}_{N})U_{N}^{A,B}\right\rangle _{\overline{\mu}_{N}}\\
 & \quad=\;\overline{\textup{cap}}_{N}(\mathcal{E}_{N}(A),\mathcal{E}_{N}(B))\;.
\end{align*}
In particular, by \eqref{23} and \eqref{24}, 
\begin{equation*}
\left\langle \overline{\Psi}_{\overline{h}_{N}^{A,B}},\overline{\Psi}_{U_{N}^{A,B}}\right\rangle _{\mathcal{C}}\;=\;(1+o_{N}(1))\kappa_{N}\,\textup{cap}_{Y}(A,B)\;.
\end{equation*}

Let $\Delta_{N}=U_{N}^{A,B}-\overline{h}_{N}^{A,B}$. By Lemma \ref{p002},
by \eqref{24} and by the previous identity, 
\begin{equation*}
\left\Vert \overline{\Psi}_{\Delta_{N}}\right\Vert _{\mathcal{C}}^{2}\;=\;\left\Vert \overline{\Psi}_{U_{N}^{A,B}}\right\Vert _{\mathcal{C}}^{2}+\left\Vert \overline{\Psi}_{\overline{h}_{N}^{A,B}}\right\Vert _{\mathcal{C}}^{2}-2\left\langle \overline{\Psi}_{\overline{h}_{N}^{A,B}},\overline{\Psi}_{U_{N}^{A,B}}\right\rangle _{\mathcal{C}}\;=\;\kappa_{N}\,o_{N}(1)\;.
\end{equation*}
Since $\Delta_{N}=0$ on $\mathcal{E}_{N}(A\cup B)$, we may write
$\Delta_{N}$ as $\Delta_{N}=\Delta_{N}(\mathfrak{o})\cdot\widetilde{\Delta}_{N}$
for some $\widetilde{\Delta}_{N}$ in $\overline{\mathfrak{C}}_{1,0}(\{\mathfrak{o}\},\mathcal{E}_{N}(A\cup B))$.

By the Dirichlet principle for the reversible process $\overline{\mathcal{L}}_{N}^{s}$,
\begin{equation*}
\left\Vert \overline{\Psi}_{\Delta_{N}}\right\Vert _{\mathcal{C}}^{2}\;=\;\overline{\mathcal{D}}_{N}(\Delta_{N})\;=\;[\Delta_{N}(\mathfrak{o})]^{2}\,\overline{\mathcal{D}}_{N}(\widetilde{\Delta}_{N})\;\ge\;[\Delta_{N}(\mathfrak{o})]^{2}\,\overline{\textup{cap}}_{N}^{s}(\mathfrak{o},\mathcal{E}_{N}(A\cup B))\;.
\end{equation*}
By Lemma \ref{lem:col_sector}, by \eqref{20}, and by Theorem \ref{s133},
\begin{align*}
 & [\Delta_{N}(\mathfrak{o})]^{2}\,\overline{\textup{cap}}_{N}^{s}(\mathfrak{o},\mathcal{E}_{N}(A\cup B))\;\ge\;(2L)^{-2}[\Delta_{N}(\mathfrak{o})]^{2}\,\overline{\textup{cap}}_{N}(\mathfrak{o},\mathcal{E}_{N}(A\cup B))\\
 & \quad=\;(2L)^{-2}[\Delta_{N}(\mathfrak{o})]^{2}\,\textup{cap}_{N}(\mathfrak{\mathcal{E}}_{N}^{1},\mathcal{E}_{N}(A\cup B))\\
 & \quad\quad=\;\big[1+o_{N}(1)\big]\,(2L)^{-2}\,[\Delta_{N}(\mathfrak{o})]^{2}\,\kappa_{N}\,\textup{cap}_{Y}(\{1\},A\cup B)\;.
\end{align*}
By the last three displayed equations and the condition of the proposition,
we obtain that $[\Delta_{N}(\mathfrak{o})]^{2}=o_{N}(1)$. This completes
the proof since 
\begin{equation*}
\Delta_{N}(\mathfrak{o})\;=\;\overline{\mathbb{P}}_{\mathfrak{o}}^{N}\left[H_{\mathcal{E}_{N}(A)}<H_{\mathcal{E}_{N}(B)}\right]-q_{A,B}(1)
\end{equation*}
by the definition of the equilibrium potential and \eqref{int1}.
\end{proof}

\smallskip\noindent{\bf Proof of Theorem \ref{s04}.} By \eqref{10},
\eqref{12}, and by Proposition \ref{s05}, we obtain that
\begin{equation*}
r_{N}^{(m)}(1,2)\;=\;\big[1+o_{N}(1)\big]\frac{e^{-N(H-h_{1})}}{2\pi N}\,\frac{1}{\nu_{1}}\,\textup{cap}_{Y}(\{1\},S_{m}\setminus\{1\})\,q_{\{2\},S_{m}\setminus\{1,2\}}(1)\;.
\end{equation*}
Hence, it suffices to prove
\begin{equation}
c_{m}(1,2)\;=\;\textup{cap}_{Y}(\{1\},S_{m}\setminus\{1\})\,q_{\{2\},S_{m}\setminus\{1,2\}}(1)\label{t01}
\end{equation}
in order to complete the proof.

Recall the reversible Markov chain $Y(t)$ on $S$. Denote
by $Y^{(m)}(t)$ the trace of $Y(t)$ on the set $S_{m}$ and by $r^{(m)}(\cdot,\cdot$)
the jump rate of $Y^{(m)}(t).$ Then, by \cite[Proposition 6.2]{BL1},
the left hand side of \eqref{t01} can be written as $\mu(1)r^{(m)}(1,2)$. 

On the other hand, by \cite[display (A.8)]{BL2}, 
\begin{equation*}
\textup{cap}_{Y}(\{1\},S_{m}\setminus\{1\})\;=\;\mu(1)\sum_{j\in S_{m}\setminus\{1\}}r^{(m)}(1,j)\;,
\end{equation*}
and by \cite[Proposition 4.2]{BL2}, 
\begin{equation*}
q_{\{2\},S_{m}\setminus\{1,2\}}(1)\;=\;\frac{r^{(m)}(1,2)}{\sum_{j\in S_{m}\setminus\{1\}}r^{(m)}(1,j)}\;.
\end{equation*}
By the last two displayed equations, we can verify that the right hand
side of \eqref{t01} is $\mu(1)r^{(m)}(1,2)$ as well.

\section{Metastability}
\label{sec9}

We present in this section the proof of Theorem \ref{mr1}, which is
very similar to the one of the reversible model \cite{LMT}. It relies
on the precise asymptotic estimates of the mean jump rate between
metastable sets obtained in the previous section.  We first recall
\cite[Theorem 2.1]{BL2} in the present context. All proofs which are
omitted below can be found in \cite{BL1, BL2, GL, LMT}.

Recall the notation introduced at the beginning of the previous
section.  Denote by $\widehat{\Psi}_{N}^{(m)}: \mc E^{(m)}_N = \cup_{i\in
  S_m} \ms E^i_N \rightarrow S_{m}$ the projection given by
\begin{equation*}
\widehat{\Psi}_{N}^{(m)}(\boldsymbol{x})\;=\;
\sum_{i\in S_{m}}i\,
\mathbf{1}\{\boldsymbol{x}\in\mathcal{E}_{N}^{i}\}\;,
\end{equation*}
and by $Y_{N}^{m,T}(t)$ the $S_{m}$-valued, hidden Markov
chain obtained by projecting the trace process $X^{(m)}_N (t)$ with
$\widehat{\Psi}_{N}^{(m)}$:
\begin{equation*}
Y_{N}^{m,T}(t)\;=\;\widehat{\Psi}_{N}^{(m)}(X^{(m)}_N (t))\;.
\end{equation*}

Recall that $\bs m^N_i$, $1\le i\le M$, stands for the bottom of the
well $\mc E^i_N$ with respect to the potential $F$.

\begin{theorem}[{\cite[Theorem 2.1]{BL2}}]
\label{bl2-th1}
Suppose that there exists a sequence $\bs \theta=(\theta_N : N\ge 1)$
of positive numbers such that, for every pair $i \not = j \in S_m$,
the following limit exists
\begin{equation}
\label{bl2a}
r^{(m)} (i,j) \;:= \; \lim_{N\to\infty} \theta_N\, r^{(m)}_N(i,j) \,.
\end{equation}
Suppose, furthermore, that for each $i\in S_m$, 
\begin{equation}
\label{bl2b}
\lim_{N\to \infty} \sup_{\bs y \in \mc E^i_N }
\frac{\Cap_N \big(\mc E^i_N, \mc E_N(S_m \setminus \{i\})\big )}
{\Cap_N(\bs m^N_i, \bs y)}\;=\;0\; .
\end{equation}
Then, for any sequence $\{\bs x_N : N \ge 1\}$, $\bs x_N \in \ms
E^i_N$, under the measure $\bb P^N_{\bs x_N}$, the rescaled process
$Y_{N}^{m,T} (t\theta_N)$ converges in the Skorohod topology to a
$S_m$-valued Markov chain with jump rates $r^{(m)} (j,k)$ and which
starts from $i$.
\end{theorem}

We claim that conditions \eqref{bl2a}, \eqref{bl2b} are fulfilled for
$\theta_N := \beta_N^{(m)} = 2\pi N \exp\{\theta_m N\}$ and $r^{(m)}
(j,k) := r_m (j,k)$, where $\beta_N^{(m)}$ has been introduced in
\eqref{eq:1} and $r_m (j,k)$ in \eqref{r_m}.

Indeed, on the one hand, by Theorem \ref{s04}, 
\begin{equation}
\tag*{\bf (H0)}
\lim_{N\rightarrow\infty} \beta_{N}^{(m)} \, r_N^{(m)}(i,j)\;=\;
\mb 1\{i\in T_m\} \,\frac{c_m (i,j)}{\nu_i} \;\;;\;i\neq j\in S_m\;.
\end{equation}
Note that the right hand side is precisely the rate $r_m (i,j)$.

On the other hand, we claim that for each $1\le i\le M$,
\begin{equation}
\tag*{\bf (H1)}
\lim_{N\rightarrow\infty} \sup_{\boldsymbol{x}\in\mathcal{E}_{N}^{i}}
\frac{\textup{cap}_{N} (\mathcal{E}_{N}^{i},
  \mc E_N(S_m \setminus \{i\}))}
{\textup{cap}_{N}(\{\bs m^N_i\},\{\boldsymbol{x}\})} 
\;=\; 0\;.
\end{equation}
In view of the sector condition presented in Lemma \ref{lem: sector
  condition} and of \cite[Lemmata 2.5 and 2.6]{GL}, it is enough to
prove the previous estimate for the symmetric capacities. This is
precisely the content of \cite[Lemma 6.2]{LMT}.

It follows from Theorem \ref{bl2-th1} and from {\bf (H0)}, {\bf (H1)}
that the the rescaled process $Y_{N}^{m,T} (t \beta_N^{(m)})$
converges in the Skorohod topology to the $S_m$-valued Markov chain
with jump rates $r_m (j,k)$.

The previous convergence does not provide much information on the
original process $X_N(t)$ if it spends a non-negligible amount of time
outside the set $\mathcal{E}^{(m)}_{N}$. But this is not the case. We
claim that for every $1\le i\le M$ and for every $T>0$,
\begin{equation}
\tag*{\bf (M3)}
\lim_{N\rightarrow\infty} \sup_{\boldsymbol{x}\in\mathcal{E}_{N}^{i}}
\mathbb{E}_{\boldsymbol{x}}^{N} \Big[\int_{0}^{T} \mathbf{1}\left\{ 
X_{N}(\beta^{(m)}_{N}t)\in\Xi_{N}\setminus \mc E^{(m)}_N \right\} \,
dt\, \Big] \;=\; 0\;.
\end{equation}
The proof of this assertion is similar to the one of \cite[Proposition
6.1]{LMT} and is therefore omitted.

One may prefer to state results on the original process
$X_{N}(\beta^{(m)}_{N}t)$ than on its trace
$X^{(m)}_{N}(\beta^{(m)}_{N}t)$. This issue is extensively discussed
in the introduction of \cite{Lan1}. To do that, we may either
introduce a weaker topology and prove the convergence of the hidden
Markov chain $Y^{(m)}_N(\beta^{(m)}_{N}t)$, introduced above
\eqref{r_m}, to the $S_m$-valued Markov chain characterized by the
jump rates $r_m$, or to prove the convergence of the last-visit
process.

\smallskip\noindent{\bf Last-visit process}. Denote by $Y^{m, \rm
  V}_N(t)$ the process which at time $t$ records the last set $\mc
E^i_N$, $i\in S_m$, visited by the process $X_N(r)$ before time
$t$. More precisely, let
\begin{equation*}
\sigma_N(t) \;:=\; \sup \big \{s\le t : X_N(s) \in \mc E^{(m)}_N \big\}\;, 
\end{equation*}
with the convention that $\sigma_N(t) = 0$ if the set $\{s\le t :
X_N(s) \in \mc E^{(m)}_N \}$ is empty. Assume that $X_N(0)$ belongs to
$\mc E^{(m)}_N$ and define $Y^{m, \rm V}_N(t)$ by
\begin{equation*}
Y^{m, \rm V}_N(t) \;=\; 
\begin{cases}
\widehat{\Psi}_{N}^{(m)}(X_N (\sigma_N(t))) & \text{if $\;X_N (\sigma_N(t))\in \mc E^{(m)}_N$,} \\
\widehat{\Psi}_{N}^{(m)}(X_N (\sigma_N (t)-)) & \text{if $\;X_N(\sigma_N(t))\not\in \mc E^{(m)}_N$.} 
\end{cases}
\end{equation*}
We refer to $Y^{m, \rm V}_N(t)$ as the \emph{last-visit process} (to
$\mc E^{(m)}_N$).  The next result, which is Proposition 4.4 in
\cite{BL1}, asserts that if the process $Y^{m,T}_N (\beta^{(m)}_{N}
t)$ converges in the Skorohod topology, and if the time spent by
$X_N(\beta^{(m)}_{N} t)$ outside $\mc E^{(m)}_N$ is negligible, then
the process $Y^{m, \rm V}_N(\beta^{(m)}_{N} t)$ converges in the
Skorohod topology to the same limit.

\begin{proposition}[{\cite[Proposition 4.4]{BL1}}]
\label{bl2-p1}
Fix $i\in S_m$ and a sequence $\{\bs x_N : N \ge 1\}$, $\bs x_N \in
\ms E^i_N$. Suppose that under the measure $\bb P^N_{\bs x_N}$ the
process $Y^{m,T}_N (\beta^{(m)}_{N} t)$ converges in the Skorohod
topology to a $S_m$-valued Markov chain, denoted by
$Y^{(m)}(t)$. Suppose, furthermore, that {\rm ({\bf M3})} is
fulfilled.  Then, the last-visit process $Y^{m, \rm
  V}_N(\beta^{(m)}_{N} t)$ also converges in the Skorohod topology to
the Markov chain $Y^{(m)}(t)$.
\end{proposition}

\smallskip\noindent{\bf Soft topology}. We introduced in \cite{Lan1}
the soft topology. As the precise definition requires much notation we
do not reproduce it here. Next result is Theorem 5.1 in
\cite{Lan1}. In the present context, it asserts that the hidden Markov
chain $Y^{(m)}_N(\beta^{(m)}_{N} t)$ converges in the soft topology to
a $S_m$-valued Markov chain $Y^{(m)}(t)$ if the process $Y^{m,T}_N
(\beta^{(m)}_{N} t)$ converges in the Skorohod topology to
$Y^{(m)}(t)$, and if condition \textbf{(M3)} is in force.

\begin{proposition}[{\cite[Theorem 5.1]{Lan1}}]
\label{l2-p1}
Fix $i\in S_m$ and a sequence $\{\bs x_N : N \ge 1\}$, $\bs x_N \in
\ms E^i_N$. Suppose that under the measure $\bb P^N_{\bs x_N}$ the
process $Y^{m,T}_N (\beta^{(m)}_{N} t)$ converges in the Skorohod
topology to a $S_m$-valued Markov chain, denoted by
$Y^{(m)}(t)$. Suppose, furthermore, that {\rm ({\bf M3})} is
fulfilled.  Then, the hidden Markov chain $Y^{(m)}_N(\beta^{(m)}_{N}
t)$ converges in the soft topology to $Y^{(m)}(t)$.
\end{proposition}

The assertion of Theorem \ref{mr1} is a consequence of {\rm ({\bf
    H0})}, {\rm ({\bf H1})}, {\rm ({\bf M3})} and of Theorem
\ref{bl2-th1} and Proposition \ref{l2-p1}.

\section{Proof of Theorem \ref{ek}}

Fix two disjoint subsets $\mc A$, $\mc B$ of $\Xi_N$.
Define the harmonic measure $\nu_{\mc A, \mc B}$, $\nu^*_{\mc A, \mc
  B}$ on $\mc A$ as
\begin{align*}
& \nu_{\mc A, \mc B}(\bs x)\,=\, \frac{\mu_N (\bs x) \, \lambda_N (\bs
  x)\, \bb P^N_{\bs x} \big[ \, H^+_{\mc B} < H^+_{\mc A} \, \big]
}{\Cap_N(\mc A,\mc B)}\,, \\
& \quad \nu^*_{\mc A, \mc B}(\bs x)\,=\, \frac{\mu_N (\bs x) \, \lambda_N (\bs
  x)\, \bb P^{*,N}_{\bs x} \big[ \, H^+_{\mc B} < H^+_{\mc A} \, \big]}
{\Cap_N^*(\mc A,\mc B)} \;, \quad \bs x\in \mc A\,.
\end{align*}

Denote by ${\bb E}^N_{\nu}$ the expectation associated to the Markov
process $\{X_N(t) : t\ge 0\}$ with initial distribution $\nu$, where
$\nu$ is a probability measure on $\Xi_N$. Next result is
\cite[Proposition A.2]{BL2}.

\begin{proposition}
\label{pbl}
Fix two disjoint subsets $\mc A$, $\mc B$ of $\Xi_N$.  Let $g:\Xi_N \to\bb R$
be a $\mu_N$-integrable function. Then,
\begin{equation*}
\bb E^N_{\nu^*_{\mc A, \mc B}} \Big[ \int_0^{H_{\mc B}} g(X_N(t))\,dt \Big] 
\;=\; \frac{\langle\, g \,,\, V^*_{\mc A, \mc B}\rangle_{\mu_N}\, }
{\Cap_N(\mc A,\mc B)}\;,  
\end{equation*}
where $\<\cdot , \cdot \>_{\mu_N}$ represents the scalar product in
$L^2(\mu_N)$.
\end{proposition}

In view of this proposition, in order to estimate
$\bb{E}_{\bs{m}_{i}^{N}}^{N}[H_{\mc{M}_{i}^{N}}]$ it suffices to
estimate the capacity
$\textup{cap}_{N}(\{\bs{m}_{i}^{N}\},\mc{M}_{i}^{N})$ and the
expectation of the equilibrium potential
${E}_{\mu_{N}}[V_{\{\bs{m}_{i}^{N}\},\mc{M}_{i}^{N}}^{*}]$.  The
following two lemmata provide these desired estimates.

\begin{lemma}
\label{lem102}
We have that 
\begin{equation*}
\textup{cap}_{N}\left(\{\bs{m}_{i}^{N}\},\,\mc{M}_{i}^{N}\right)
\;=\;\left[1+o_{N}(1)\right]\,Z_{N}^{-1}\, (2\pi N)^{\frac{d}{2}-1}
\, e^{-NH}\, \sum_{j\in S_{u}}c_{u}(i,j)\;.
\end{equation*}
\end{lemma}

\begin{proof}
Note in Lemma \ref{con73} that the test flows used in the proof of
Theorem \ref{s133} were divergence free not only on $\mc{E}_{N}(A\cup
B)^{c}$ but also on $\{\bs{m}_{i}^{N}:i\in A\cup B\}^{c}$.  Hence, the
conclusion of Theorem \ref{s133} is not affected if we replace
$\textup{cap}_{N}(\mc{E}_{N}(A),\,\mc{E}_{N}(B))$ by
$\textup{cap}_{N}(\{\bs{m}_{j}^{N}:j\in A\},\,\{\bs{m}_{j}^{N}:j\in
B\})$.  By inserting $A=\{i\}$ and $B=S_{u}\setminus\{i\}$ into this
modified version of Theorem \ref{s133}, we obtain that
\begin{equation*}
\textup{cap}\left(\{\bs{m}_{i}^{N}\},\,\mc{M}_{i}^{N}\right)
\;=\;\left[1+o_{N}(1)\right] \,Z_{N}^{-1}\, 
(2\pi N)^{\frac{d}{2}-1}e^{-NH}
\,\textup{cap}_{Y}(\{i\},\,S_{u}\setminus\{i\})\; .
\end{equation*}
Thus, it suffices to show that 
\begin{equation*}
\sum_{j\in S_{u}}c_{u}(i,j)\;=\;
\textup{cap}_{Y}(\{i\},\,S_{u}\setminus\{i\})\;.
\end{equation*}

Denote by $Y^{S_{u}}(\cdot)$ the trace process of $Y(\cdot)$ on the
set $S_{u}$, and by $r^{S_{u}}(\cdot,\cdot)$ the jump rate of
$Y^{S_{u}}(\cdot)$. By \eqref{c_m} and \cite[Display (A.8)]{BL2}, we
have that $c_{u}(i,j)=\mu(i)r^{S_{u}}(i,j)$ for all $i,\,j\in S_{u}$.
Hence,
\begin{equation*}
\sum_{j\in S_{u}}c_{u}(i,j)\;=\;\mu(i)\sum_{j\in S_{u}}r^{S_{u}}(i,j)
\;=\;\mu(i)\, r^{S_{u}}(i,S_{u}\setminus\{i\})
\;=\;\textup{cap}_{Y}(\{i\},\,S_{u}\setminus\{i\})\;,
\end{equation*}
where the last equality follows again from \cite[display
(A.8)]{BL2}. This completes the proof of the lemma.
\end{proof}

We recall the following well-known estimate on the equilibrium
potential (cf. \cite[display (3.3)]{LL}): for all $\bs{x}\in\Xi_{N}$
and disjoint sets $A,\,B\subset\Xi_{N}$ with $\bs{x}\notin A\cup B$,
\begin{equation}
\label{e1032}
V_{A,B}(\bs{x}) \;\le\; \frac{\textup{cap}_{N}(\{\bs{x}\},\,A)}
{\textup{cap}_{N}(\{\bs{x}\},\,B)}\;\cdot
\end{equation}
Since $\textup{cap}_{N}=\textup{cap}_{N}^{*}$, the same inequality
also holds for $V_{A,B}^{*}(\bs{x})$. 

\begin{lemma}
\label{lem103}
We have that
\begin{equation*}
{E}_{\mu_{N}} \big[V_{\{\bs{m}_{i}^{N}\},\mc{M}_{i}^{N}}^{*}\big]
\;=\;\left[1+o_{N}(1)\right]\, \vartheta_N\, \nu_{i}\;,
\end{equation*}
where $\vartheta_N=Z_{N}^{-1}\, (2\pi N)^{\frac{d}{2}-1}\, e^{-Nh_{i}}$.  
\end{lemma}

\begin{proof}
In view of \eqref{thm: estimate of meta mass}, it suffices to prove
that  
\begin{equation*}
{E}_{\mu_{N}}\left[V_{\{\bs{m}_{i}^{N}\},\mc{M}_{i}^{N}}^{*}\right]
\,-\, \mu_{N}(\mc{E}_{N}^{i}) \;=\; o_{N}(1) \, \vartheta_N\;.
\end{equation*}
The left hand side can be written as
\begin{equation}
\label{e1031}
\begin{aligned}
&\sum_{\bs{x}\in\mc{E}_{N}^{i}}(V_{\{\bs{m}_{i}^{N}\},\mc{M}_{i}^{N}}^{*}(\bs{x})-1)
\, \mu_{N}(\bs{x}) \;+\; 
\sum_{j:j\neq i} \sum_{\bs{x}\in\mc{E}_{N}^{j}} 
V_{\{\bs{m}_{i}^{N}\},\mc{M}_{i}^{N}}^{*}(\bs{x}) \, \mu_{N}(\bs{x}) \\
&\quad + \, \sum_{\bs{x}\in(\mc{E}_{N})^{c}} 
V_{\{\bs{m}_{i}^{N}\},\mc{M}_{i}^{N}}^{*}(\bs{x})
\, \mu_{N}(\bs{x})\;.
\end{aligned}
\end{equation}
The first sum is equal to
\begin{equation*}
-\, \sum_{\bs{x}\in\mc{E}_{N}^{i}}
V_{\mc{M}_{i}^{N},\{\bs{m}_{i}^{N}\}}^{*}(\bs{x})
\, \mu_{N}(\bs{x})\;.
\end{equation*}
By \eqref{e1032}, by the monotonicity of capacity and by the fact that
the symmetric capacity is bounded by the capacity,
$\textup{cap}_{N}^{s} \le \textup{cap}_{N}$ (cf. \cite[Lemma 2.6]{GL}),
\begin{equation}
\label{e1033}
V_{\mc{M}_{i}^{N},\{\bs{m}_{i}^{N}\}}^{*}(\bs{x}) \; \le\;
\frac{\textup{cap}_{N}(\{\bs{x}\},\,\mc{M}_{i}^{N})}
{\textup{cap}_{N}(\{\bs{x}\},\,\{\bs{m}_{i}^{N}\})}
\; \le\;
\frac{\textup{cap}_{N}(\mc{E}_{N}^{i},\,\check{\mc{E}}_{N}^{i})}
{\textup{cap}_{N}^{s}(\{\bs{x}\},\,\{\bs{m}_{i}^{N}\})}\;\cdot
\end{equation}
To estimate the denominator, let $l_{\bs{x}}$ be a continuous path
connecting $\bs{x}$ and $\bs{m}_{i}^{N}$ such that $F(\bs{y})\le
H-\epsilon$ for all $\bs{y}\in l_{\bs{x}}$.  By discretizing this
path, as in the proof of Proposition \ref{s15}, we obtain a good path
(cf. Section \ref{sub52}) $\Gamma_{x} = (\bs{x}_{0}, \,\bs{x}_{1},
\,\cdots, \,\bs{x}_{L})$ connecting $\bs{x}$ to $\bs{m}_{i}^{N}$ where
$L=O(N)$. Denote by $\chi_{\Gamma_{\bs{x}},1}$ the unit flow from
$\bs{x}$ to $\bs{m}_{i}^{N}$ introduced in \eqref{eq:chi_Ga}. By the
Thomson principle for reversible Markov chains,
\begin{equation}
\label{e1034}
\frac{1}{\textup{cap}_{N}^{s}(\{\bs{x}\},\,\{\bs{m}_{i}^{N}\})}
\;\le\; \Vert \chi_{\Gamma_{\bs{x}},1}\Vert^{2}
\;=\; \sum_{i=0}^{L-1}
\frac{1}{\mu_{N}(\bs{x}_{i})\, R_{N}(\bs{x}_{i},\bs{x}_{i+1})}
\;\le\; C \, N\, Z_{N} \, e^{N(H-\epsilon)}\;.
\end{equation}
Hence, by \eqref{e1033}, \eqref{e1034} and by Theorem \ref{s13},
we obtain 
\begin{equation*}
V_{\mc{M}_{i}^{N},\{\bs{m}_{i}^{N}\}}^{*}(\bs{x})
\;\le\;C\, N^{\frac{d}{2}}\, e^{-N\epsilon}\;.
\end{equation*}
In particular, by \eqref{thm: estimate of meta mass}, the absolute
value of the first sum in \eqref{e1031} is bounded by $o_{N}(1)\,
\vartheta_N$ . By similar arguments, the second summation in
\eqref{e1031} is also of order $o_{N}(1)\vartheta_N$.

On the other hand, as $V_{\{\bs{m}_{i}^{N}\},\mc{M}_{i}^{N}}^{*} \le
1$, the last summation in \eqref{e1031} is bounded by
\begin{equation*}
\sum_{\bs{x}\in(\mc{E}_{N})^{c}} Z_{N}^{-1}\, e^{-NF(\bs{x})}
\;\le\;C\, N^{d}\, Z_{N}^{-1}\, e^{-N(H-\epsilon)}\;.
\end{equation*}
Since $H-\epsilon>h_{i}$, it is obvious that the last summation is
of order $o_{N}(1)\vartheta_N$, which completes the proof of the
lemma. 
\end{proof}

Theorem \ref{ek} follows immediately from Proposition \ref{pbl}
and Lemmata \ref{lem102}, \ref{lem103}.


\section{Appendix}

We present in this appendix a generalization of Sylvester's law of
inertia. 
 
\begin{lemma}
\label{lem: property of M}
Let $\mathbb{X}$, $\mathbf{\mathbb{Y}}$ be $n\times n$ matrices. Assume
that $\mathbb{Y}$ is a non-singular, symmetric matrix which has only one
negative eigenvalue, and that
\begin{equation}
\label{eq: cond U}
\mathbb{X}^{s}=(\mathbb{X}+\mathbf{\mathbb{X}}^{\dagger})/2
\mbox{ is positive definite}\;.
\end{equation}
Then, $\mathbf{\mathbb{X}\mathbb{Y}}$ has one negative
eigenvalue and $(n-1)$ positive eigenvalues.
\end{lemma}

\begin{proof}
By usual diagonalization and conjugation argument, it suffices to
consider the case where
$\mathbf{\mathbb{Y}}=\mbox{diag}(-1, 1, \dots, 1)$.  It is
well-known that a matrix $\mathbf{\mathbb{X}}$ satisfying (\ref{eq:
  cond U}) does not have negative eigenvalue and that
$\det\mathbb{X}>0$. Hence,
$\det(\mathbf{\mathbb{X}\mathbb{Y}})=-\det\mathbb{X}<0$ and
$\mathbf{\mathbb{X}\mathbb{Y}}$ has at least one negative eigenvalue.

Assume that $\bb X \bb Y$ has two negative eigenvalues $-a,\,-b$ and
let $\boldsymbol{u},\,\boldsymbol{w}$ be the corresponding
eigenvectors.  Denote by $u_{1},\,w_{1}$ the first coordinates of
$\boldsymbol{u}$ and $\boldsymbol{w}$. If $u_{1}=0$, then
$\mathbf{\mathbb{Y}}\boldsymbol{u}=\boldsymbol{u}$ and hence
\begin{eqnarray*}
-a\boldsymbol{u} & = & \mathbf{\mathbb{X}\mathbb{Y}}
\boldsymbol{u}\;=\;\mathbf{\mathbb{X}}\boldsymbol{u}\;,
\end{eqnarray*}
which is a contradiction since $\mathbf{\mathbb{X}}$ does not have
negative eigenvalue. Thus, $u_{1}\neq0$ and similarly, $w_{1}\neq0$.

By definition of $a$, $b$ and by \eqref{eq: cond U}, or any
$c\in\mathbb{R}$, 
\begin{equation*}
(\boldsymbol{u}+c\boldsymbol{w})^{\dagger}\mathbf{\mathbb{Y}}
(a\boldsymbol{u}+bc\boldsymbol{w}) \;=\; 
-(\boldsymbol{u}+c\boldsymbol{w})^{\dagger}
\mathbf{\mathbb{Y}\mathbb{X}\mathbb{Y}}
(\boldsymbol{u}+c\boldsymbol{w})<0\;.
\end{equation*}
Let $p=-u_{1}/(bw_{1})$. By substituting $c$ by $ap$ in the previous
equation the first coordinate of $a\boldsymbol{u}+bc\boldsymbol{w}$
vanishes. Thus, since $\mathbf{\mathbb{Y}} \boldsymbol{z}=
\boldsymbol{z}$ if the first coordinate of $\bs{z}$ is zero, 
\begin{equation*}
(\boldsymbol{u}+ap\boldsymbol{w}) \cdot(a\boldsymbol{u} +abp\boldsymbol{w})
\;=\; (\boldsymbol{u}+ap\boldsymbol{w})^{\dagger} \mathbf{\mathbb{Y}}
(a\boldsymbol{u}+abp\boldsymbol{w}) \;<\;0\;.
\end{equation*}
Similarly, substituting $c$ by $bp$ in order to make the first
coordinate of $\boldsymbol{u}+c\boldsymbol{w}$ equal to $0$, we obtain
that
\begin{equation*}
\;(\boldsymbol{u}+bp\boldsymbol{w})\cdot
(a\boldsymbol{u}+b^{2}p\boldsymbol{w}) \;=\;
(\boldsymbol{u}+bp\boldsymbol{w})^{\dagger}\mathbf{\mathbf{\mathbb{Y}}}
(a\boldsymbol{u}+b^{2}p\boldsymbol{w}) \;<\; 0\;.
\end{equation*}
Summing the previous inequality with the penultimate one multiplied by
$b/a$ we obtain that
\begin{equation*}
(a+b)\, |\boldsymbol{u}+bp\boldsymbol{w}|^{2} \;<\; 0 \;,
\end{equation*}
which is a contradiction. 
\end{proof}

\smallskip\noindent{\bf Acknowledgements.} The authors would like to
thank the anonymous referee for her/his comments which helped to
improve the presentation of the results.

\end{document}